\documentclass[12pt]{amsart}%

\usepackage{amsmath, amssymb, amsthm, amsfonts, bbm, dsfont}
\usepackage{graphicx}
\usepackage{wrapfig}
\usepackage{mathrsfs}

\usepackage{hyperref}

\usepackage{a4wide}

\usepackage{mathabx} 

\numberwithin{equation}{section}
\theoremstyle{plain}
\newtheorem{theorem}[subsubsection]{Theorem}
 
 \newtheorem{proposition}[subsubsection]{Proposition}
 \newtheorem{corollary}[subsubsection]{Corollary}
 \newtheorem{conjecture}[subsubsection]{Conjecture}

 \theoremstyle{definition}

\newtheorem{remark}[subsubsection]{Remark}

\usepackage{stackengine}
\stackMath
\usepackage{mathtools}
\usepackage{amssymb}
\usepackage{dsfont}

\usepackage{mathabx}

\newcommand{\rmL}{\mathrm{L}}
\newcommand{\rmT}{\mathrm{T}}
\newcommand{\rmHHH}{\mathrm{HHH}}

\newcommand{\CC}{\mathbb{C}}
\newcommand{\cc}{\mathbb{C}}

\newcommand{\ZZ}{\mathbb{Z}}

\newcommand{\bbL}{\mathbb{L}}

\newcommand{\sfZ}{\mathsf{Z}}

\newcommand{\scZ}{\mathscr{Z}}

\newcommand{\scX}{\mathscr{X}}

\newcommand{\bfX}{\mathbf{X}}
\newcommand{\bfY}{\mathbf{Y}}
\newcommand{\bfZ}{\mathbf{Z}}
\newcommand{\bfW}{\mathbf{W}}
\newcommand{\bfm}{\mathbf{m}}
\newcommand{\bfk}{\mathbf{k}}

\newcommand{\calF}{\mathcal{F}}

\newcommand{\cf}{\mathcal{F}}
\newcommand{\calG}{\mathcal{G}}
\newcommand{\cg}{\mathcal{G}}
\newcommand{\calO}{\mathcal{O}}

\newcommand{\calB}{\mathcal{B}}
\newcommand{\calP}{\mathscr{P}}

\newcommand{\calH}{\mathcal{H}}

\newcommand{\frg}{\mathfrak{g}}

\newcommand{\frb}{\mathfrak{b}}

\newcommand{\frn}{\mathfrak{n}}

\newcommand\aff{\textup{aff}}

\newcommand{\Gr}{\textup{Gr}}

\newcommand\loc{\textup{loc}}

\newcommand\st{\textup{st}}

\newcommand{\Tr}{\textup{Tr}}

\newcommand\Hom{\textup{Hom}}
\newcommand\End{\textup{End}}

\newcommand{\Ext}{\textup{Ext}}

\newcommand{\DG}{\mathrm{DG}}
\newcommand{\HHH}{\mathrm{HHH}}

\newcommand\GL{\textup{GL}}
\newcommand\gl{\mathfrak{gl}}

\newcommand\SO{\textup{SO}}

\newcommand{\Ad}{\textup{Ad}}

\newcommand{\bl}{\bullet}
\newcommand{\RR}{\mathbb{R}}

\newcommand{\quash}[1]{}

\newcommand{\pt}{\mathsf{p}}

\newcommand{\sm}{\textup{sm}}

\newcommand{\Br}{\mathfrak{Br}}

\usepackage{tikz}
\usetikzlibrary{shapes,fit,intersections}
\usetikzlibrary{decorations.markings}
\usetikzlibrary{decorations.pathreplacing}
\usetikzlibrary{patterns}
\usetikzlibrary{matrix,arrows}
\usepgflibrary{decorations.pathmorphing}

\usepackage{tikz-cd}

\newcommand{\Hilb}{\textup{Hilb}}

\newcommand{\MF}{\mathrm{MF}}

\newcommand{\calX}{\mathcal{X}}

\newcommand{\calZ}{\mathcal{Z}}
\newcommand{\calC}{\mathcal{C}}

\newcommand{\Fl}{\mathrm{Fl}}

\newcommand{\Tqt}{\mathbb{T}_{q,t}}

\newcommand{\thc}{\dddot{\mathrm{Cat}}}
\newcommand{\thcs}{\thc_{\tsym}}
\newcommand{\ob}{\mathrm{Obj}}

\newcommand{\thcm}{\thc}

\newcommand{\CH}{\mathsf{CH}}

\newcommand{\HC}{\mathrm{HC}}
\newcommand{\KN}{\mathrm{KN}}
\newcommand{\HHom}{\mathbb{H}\mathrm{om}}
\newcommand{\cHom}{\mathcal{H}\mathrm{om}}
\newcommand{\Sbim}{\mathbb{S}\mathrm{bim}}
\newcommand{\bim}{\mathbb{B}\mathrm{im}}

\newcommand{\Brn}{\Br_n}

\def\Hilb{ \mathrm{Hilb}}

\newcommand{\ti}{\times}

\newcommand{\alg}{\mathrm{alg}}

\newcommand{\tft}{TQFT}
\newcommand{\Def}{\tDef}

\def\rst{ \mathrm{st}}
\def\tper{\mathrm{per}}
\def\tsym{\mathrm{sym}}

\def\tgl{ \mathrm{gl}}
\def\tObj{\mathrm{Obj}}
\def\tid{\mathrm{id}}

\def\tDr{ \mathrm{Dr}}
\def\tlin{ \mathrm{lin}}

\def\ttst{\mathrm{st}}

\def\tfs{\mathrm{fs}}
\def\tloc{\mathrm{loc}}
\def\ttaut{\mathrm{taut}}

\def\rmTs{ \mathrm{T}^* }

\def\sptm{space-time }
\def\tDef{\mathrm{Def}}

\def\tfrm{ \mathrm{f}}
\def\sfZf{\sfZ^{\tfrm}}

\def\rmD{ \mathrm{D}}
\def\Dper{\rmD^{\tper}}

\def\DG{\mathrm{DG}}
\def\tcoh{\mathrm{coh}}

\def\DGcp{ \DG_{\tcoh}^{\tper}}

\def\tCat{\mathrm{Cat}}

\def\thcgb{ \dddot{\Cat}^{\raisebox{-2pt}{$\scriptstyle\bullet$}}_{\tgl} }

\def\tqft{TQFT}

\def\sfp{ \mathsf{p}}
\def\sfpn{ \sfp_n }
\def\sfpnv#1{ \sfp_{n_{#1}}}
\def\sfpno{ \sfpnv{1}}
\def\sfpnt{ \sfpnv{2}}
\def\sfZ{ \mathsf{Z}}
\def\sfZv#1{ \sfZ^{#1}}
\def\sfZb{ \sfZv{\bullet}}

\def\bfn{ \mathbf{n}}
\def\bfnb{ \bfn^{\bullet}}
\def\bfnf{ \bfn^{\tfrm}}
\def\bfnbv#1{ \bfnb_{#1}}
\def\bfnbo{ \bfnbv{1}}
\def\bfnbt{ \bfnbv{2}}

\def\tctg{2-category}
\def\tctgs{2-categories}
\def\hctg{3-category}

\def\GLv#1{\GL_{#1}}

\def\GLn{ \GLv{n}}
\def\GLnv#1{ \GLv{n_{#1}}}
\def\GLno{ \GLnv{1}}
\def\GLnt{ \GLnv{2}}

\def\gln{\gl_{n}}
\def\glnv#1{ \gl_{n_{#1}}}
\def\glno{ \glnv{1}}
\def\glnt{ \glnv{2}}

\def\TsGLn{ \rmTs\GLn }

\def\Adv#1{ \Ad_{#1}}
\def\Gr{ \mathrm{Gr}}
\def\Grv#1{ \Gr_{#1}}

\def\IC{ \mathbb{C}}

\def\xWfr{ \xW^{\rmfr}}

\def\hmf{HOMFLY-PT}
\def\hmfh{\hmf\ homology}
\def\hmflh{\hmf\ link homology}

\def\obOe{ \calO_{\varnothing}}

\def\KRS{KRS}

\newcommand{\cY}{\mathscr{Y}}

\newcommand{\ie}{\textit{i}.\textit{e}.}

\title{3D TQFT and HOMFLYPT homology}
\author{A. Oblomkov}
\address{
A.~Oblomkov\\
Department of Mathematics and Statistics\\
University of Massachusetts at Amherst\\
Lederle Graduate Research Tower\\
710 N. Pleasant Street\\
Amherst, MA 01003 USA
}
\email{oblomkov@math.umass.edu}

\author{L. Rozansky}
\address{
L.~Rozansky\\
Department of Mathematics\\
University of North Carolina at Chapel Hill\\
CB \# 3250, Phillips Hall\\
Chapel Hill, NC 27599 USA
}
\email{rozansky@math.unc.edu}

\begin{document}
\maketitle

\begin{abstract}

We describe a family of 3d topological B-models whose target spaces are Hilbert schemes of points in $\CC^2$. The interfaces separating theories with different numbers of points correspond to braid strands. The Hilbert space of the picture of a closed braid is the HOMFLY-PT homology of the corresponding link.



\end{abstract}




\def\IR{\mathbb{R}}
\def\mTN{\mathrm{mTN}}
\def\IRvv#1#2{ \IR^{#1}_{#2} }
\def\SO{ \mathrm{SO}}
\def\IRo{ \IR^1 }
\def\IRodr{ \IR^1_{\mathrm{dr}}}
\def\IRbr{ \IR_{\mathrm{br}}}
\def\IRbrh{ \IRbr^3 }
\def\IRbrhp{ {\IRbrh}' }
\def\IRhhg{ \IR^3_{\mathrm{h}}}
\def\IRhcl{ \IR^3_{\mathrm{c}}}
\def\xaddt{additional}
\def\Soh{ S^1_3}
\def\Sohp{ S^{1\prime}_3}
\def\xcr{ \alpha }
\def\xcra{ \alpha }
\def\xcraI{ \xcra_{\xInt}}
\def\xmTN#1#2{#1\mathrm{TN}_{#2}}
\def\xmTNcp{ \xmTN{\nsm}{3'7'8'9'} }
\def\xmTNcc{\nsm \mathrm{TN} }
\def\nsm{ m }
\def\tA{ \tilde{A}}
\def\tAm{ \tA_{\nsm-1} }
\def\tAmf{ \tA_{\nsm}}
\def\Amf{ A_{\nsm}}
\def\tAz{\tA_0 }

\def\rU{ \mathrm{U}}

\def\xn{ n }
\def\rmT{ \mathrm{T}}
\def\rmTs{ \rmT^* }
\def\cX{\mathbf{X}}
\def\rmsm{ \mathrm{sym}}
\def\cXsm{ \cY^{\rmsm}}
\def\cXsmv#1{ \cXsm_{#1}}
\def\cXsme{ \cXsmv{e}}
\def\cXQ{ \cX_Q }
\def\cY{ \mathcal{Y}}

 \def\cZsm{ \cZ^{\rmsm}}
\def\cZsm{ \cZ}

\def\cZsmv#1{ \cZsm_{#1}}
\def\cZsme{ \cZsmv{e}}

\def\Cat{\mathrm{Cat}}
\def\tCat{ \ddot{\Cat}}
\def\hCat{ \dddot{\Cat}}

\def\catB{ \ddot{\mathrm{Cat}}}
\def\catBh{ \dddot{\mathrm{Cat}}}
\def\catBv#1{ \catB(#1) }
\def\xG{ G }

 \def\cZ{ \mathscr{Z}}

\def\sm{ \mathrm{sym}}
\def\arr{\mathrm{arr}}
\def\bow{ \mathrm{bow}}
\def\cXar{ \cX^{\arr}}
\def\cXbw{ \cX^{\bow}}
\def\cXsm{ \cY}
\def\cXsmv#1{ \cXsm_{#1}}
\def\cXsmrn{ \cXsmv{\bqn;\bqr}}
\def\cXv#1{ \cX_{#1}}
\def\cXrn{ \cXv{\bqn;\bqr}}
\def\xGv#1{ \xG_{#1}}
\def\GL{ \mathrm{GL}}
\def\GLv#1{ \GL(#1) }
\def\xGbn{ \xGv{\bqn}}
\def\xW{ W }
\def\xWe{ \xW_e }
\def\xWbw{ \xW_{\mathrm{bow}}}
\def\xWsqr{ \xW_{\mathrm{square}}}
\def\xWsm{ \xW_{\rmsm}}
\def\bMF{ \mathbf{MF}}

\def\qr{ r }
\def\qn{ n }
\def\qk{ k }
\def\qi{ i }
\def\qvc{ v }
\def\bqr{ \mathbf{\qr}}
\def\bqn{ \mathbf{\qn}}
\def\xphi{ \varphi }
\def\xbphi{ \boldsymbol{\xphi}}
\def\rmfr{ \mathrm{fr}}
\def\yphi{ \xphi^{\rmfr}}
\def\xpsi{ \psi }
\def\xfph{ \xpsi^{\rmfr}}
\def\tX{\tilde{X}}
\def\tX{\tilde{Y}}
\def\tqvc{ \tilde{\qvc}}

\def\xrho{ \rho }
\def\xbrho{ \boldsymbol{\rho}}

\def\xQ{ Q }
\def\xQv#1{ \xQ(#1)}
\def\xQbrn{ \xQv{\bqr;\bqn}}
\def\xQbrnp{ \xQv{\bqr;\bqn'}}
\def\xQbr{ \xQv{\bqr}}
\def\rmor{\mathrm{or}}
\def\xQor{ \xQ^{\rmor}}
\def\xQorv#1{ \xQor(#1)}
\def\xQorbrn{ \xQorv{\bqr;\bqn}}
\def\xQorbrnp{ \xQorv{\bqr;\bqn'}}
\def\xRp{ \mathrm{Rep}}
\def\RpQorbrn{ \xRp(\xQorbrn) }
\def\cM{ \mathcal{M}}
\def\cMv#1{ \cM_{#1}}
\def\cMbrn{ \cMv{\bqr;\bqn}}
\def\cMbrnp{ \cMv{\bqr;\bqn'}}

\def\dsq{ /\!/ }

\def\IRz{ \IR_0 }
\def\xSg{ \Sigma }
\def\xSgot{ \xSg_{12}}

\def\MtB{ M^{10}_{\mathrm{IIB}}}
\def\MtA{ M^{10}_{\mathrm{IIA}}}
\def\IRtot{ \IRvv{2}{12}}

\def\brt{ b }
\def\ctcr{ \beta }

\def\xInt{ I }

\def\cC{ \mathcal{C}}

\def\gLc{generalized Lagrangian correspondence}
\def\gLcs{\gLc s}
\def\auv{auxiliary variety}
\def\auvs{auxiliary varieties}
\def\sptn{super-potential}

\def\sstr{super-string}
\def\cmvr{commuting variety}
\def\cmvrs{commuting varieties}

\def\xPhi{ \Phi }
\def\xPhior{ \xPhi^{\rmor}}
\def\xPhiorz{ \xPhior_0 }

\def\xmu{ \mu }
\def\xrho{\rho }
\def\xsgm{ \sigma }

\def\rmN{ \mathrm{N}}
\def\rmNve{ \rmN^\vee }
\def\xpip{ \pi' }

\def\cM{ \mathscr{M}}
\def\xxin{ \cM }
\def\xxinv#1{ \xxin_{#1}}
\def\xxinnr{ \xxinv{\qn;\qr}}
\def\xxinnpr{ \xxinv{\qn';\qr}}
\def\xxinnkr{ \xxinv{\qn+\qk;\qr}}
\def\xxinzr{ \xxinv{0;\qr}}
\def\xxinnmor{ \xxinv{\qn-1;\qr}}

\def\obFL{ \mathcal{F}}
\def\obFLv#1{ \obFL_{#1}}
\def\obFLn{ \obFLv{\qn}}
\def\obFLz{ \obFLv{0}}
\def\obFLz{ \mathcal{O}_0 }

\def\obSL{ \mathcal{S}}
\def\obSLv#1{ \obSL_{#1}}
\def\obSLbl{ \obSLv{\bldl}}

\def\aff{ \mathrm{aff}}
\def\gBr{ \mathfrak{Br}}
\def\gBrv#1{ \gBr_{#1}}
\def\gBrn{ \gBrv{\qn}}
\def\gBraf{ \gBr^{\aff}}
\def\gBrafv#1{ \gBraf_{#1}}
\def\gBrafn{ \gBrafv{\qn}}

\def\brb{\beta}
\def\ctm#1{ \hat{#1}}
\def\hbrb{ \ctm{\brb}}
\def\idb{ 1 }
\def\hidb{ \ctm{\idb}}

\def\Lbrb{ L_{\brb}}

\def\rmH{ \mathrm{H}}

\def\Vqn{ V_{\qn}}

\def\xPhi{ \Phi }
\def\xPhiwn{ \xPhiwv{\qn}}
\def\xPhibn{ \xPhibv{\qn}}
\def\xPhiwi{ \xPhiwv{\qi}}
\def\xPhiwo{ \xPhiwv{1}}
\def\xPhiwnp{ \xPhiwv{\qn,\qn'}}
\def\xPhibnp{ \xPhibv{\qn,\qn'}}

\def\xPhiwk{ \xPhiwv{\qk}}
\def\xPhibk{ \xPhibv{\qk}}

\def\xPhiwv#1{ \widehat{\Box}_{#1}}
\def\xPhibv#1{ \widehat{\blacksquare}_{#1}}
\def\xPsiwv#1{ \widecheck{\Box}_{#1}}
\def\xPsibv#1{ \widecheck{\blacksquare}_{#1}}

\def\xPsiwn{ \xPsiwv{\qn}}
\def\xPsibn{ \xPsibv{\qn}}
\def\xPsiwk{ \xPsiwv{\qk}}
\def\xPsibk{ \xPsibv{\qk}}

\def\brsg{\sigma}
\def\brsgv#1{ \brsg_{#1}}
\def\brsgi{ \brsgv{i}}

\def\bldl{ \boldsymbol{\lambda}}
\def\bldl{\lmb}
\def\lmb{\lambda}

\def\Ann{ \mathrm{Ann}}

\section{Introduction}
\label{sec:introduction}
In our previous papers we proposed a  construction of the \hmflh\ which is based on an interaction between two quiver varieties: the cotangent bundle to the flag variety and the Hilbert scheme of points in $\CC^2$. For $\qr,\qn\geq 0$ denote $\xxinnr$ the moduli space of $\qn$ $U(\qr)$ instantons in $\IC^2$. Equivalently, $\xxinnr$ is the Nakajima quiver variety of the framed $\tAz$ quiver:
\begin{equation}
\label{eq:tazq}
\begin{tikzpicture}
\draw[very thick,-] (-0.15,-0.15) to [out=200,in=180] (0,-1.2);
\draw[very thick,-] (0.15,-0.15) to [out=-20,in=0] (0,-1.2);
\node [left] at (-0.15, 0) {$\qn$};
\node[left] at (-0.15,1) {$\qr$};
\draw[very thick,-] (0,0.15) -- (0,0.85);
\draw[very thick,fill=white] (-0.15,0.85) rectangle ++(0.3,0.3);
\draw[very thick,fill=white] (0,0) circle [radius=0.15];
\end{tikzpicture}
\end{equation}
Following~\cite{KapustinSaulinaRozansky09} and~\cite{KapustinRozansky10} we associate to $\xxinnr$ the 2-category $\catB(\xxinnr)$.
 We consider a special flag-related object $\obFLn$ in this category. Using our previous results \cite{OblomkovRozansky16} we construct a monoidal functor
\[\Phi:\gBrn\longrightarrow \Hom(\obFLn,\obFL_n),\qquad \brb\mapsto \Phi(\beta),
\]
where $\gBrn$ is a braid group on $\qn$ strands.

Our main knot theoretic  result  from \cite{OblomkovRozansky16},\cite{OblomkovRozansky20} can be recast into
\begin{theorem}
If     $\brb\in \Brn $ is a braid, then in case of $\qr=1$ there is an isomorphism
\begin{equation}\label{eq:old}
\mathcal{H}om_{\Hom(\obFL_n,\obFL_n)} (\Phi(1_n), \Phi(\beta)\otimes\Lambda^\bullet \Vqn) \cong \rmHHH(L(\beta)).
\end{equation}
Here $\idb_n$ is the identity braid, $\Vqn=\CC^n$ is the defining representation of $\GL(\qn)$, $L(\beta)$ is the link constructed by closing $\brb$ and $\mathrm{HHH}(L(\beta))$ is its \hmfh.
\end{theorem}

In this paper we consider this construction of link homology from a \hctg\ perspective.
We outline a construction of a 3-category \(\hCat(Q)\) for any $\tilde{A}_n$-type quiver \(Q\). In the case \(Q=\tAz\) we provide a conjectural description of a part of
3-categorical structure by relating our setting to the theory of monoidal 2-categories of Kapranov and Voevodsky \cite{KapranovVoevodsky94a}, in section~\ref{sec:monoidal-cat} we construct the relevant monoidal 2-category \(\tCat(Q)\). That allows us to define the 3-category \(\hCat(Q)\) as a subcategory of 
module 2-category over \(\tCat(Q)\) in section~\ref{sec:modular-cat-def}.

The 2-category \(\catB(\xxinnr)\) emerges
as the 2-category of morphisms
between the objects \(\tAz(n;r)\) and \(\tAz(0;0)\), that is \(\catB(\xxinnr)=\HHom(\tAz(n;r),\tAz(0;0))\).
The objects of the \hctg\ $\hCat(\tAz)$ are $\tAz(n;r)$ or, equivalently, pairs of numbers $(\qn;\qr)$. For two pairs $(\qn;\qr)$ and $(\qn+\qk;\qr)$ we define two `box' functors:
\begin{equation}
\label{eq:ns5sq}
\xPhiwk,\xPhibk
\colon \catB(\xxinnr)\longrightarrow \catB(\xxinnkr).
\end{equation}
The functors represent \gLcs\ between $\xxinnr$ and $\xxinnkr$, and those originate from the edges of the second Nakajima quiver which is `transversal' to the $\tAz$ quiver\eqref{eq:tazq} in the sense of Gaiotto-Witten defects \cite{GaiottoWitten09} in Kapustin-Witten theory \cite{KapustinWitten07}, or brane arrangements in the superstring theory, see section~\ref{sec:phys}.

The quiver variety $\xxinzr$ is just a point, hence its \tctg\ $\catB(\xxinzr)$ has a canonical object $\obOe$, and we show in proposition~\ref{prop:comp12} that $\obFLn$ is the result of applying $\xPhiwo$  $\qn$ times
to this object:
\[
\obFLn = \LaTeXunderbrace{\xPhiwo\cdots\xPhiwo}_{\qn} \obOe \in \catB(\xxinnr).
\]

Each individual $\xPhiwo$ corresponds to a braid strand,  and an elementary braid $\brsgi$ which braids the $i$-th and the $(i+1)$-st string, becomes a natural self-transformation of the composition:
\[
\brsgi \rightsquigarrow \xPhiwo\xPhiwo\colon\catB(\xxinv{i-1,\qr})\longrightarrow\catB(\xxinv{i+1,\qr}).
\]

When $\qr=0$ we define two more `box' functors
\begin{equation}
\label{eq:d4sq}
\xPsiwk,\xPsibk
\colon \catB(\xxinv{\qn,0})\longrightarrow \catB(\xxinv{\qn+\qk,0}).
\end{equation}
These functors are related to Cherkis bow edges of the transversal `arrow-bow' quiver.
For an $\qn$-box, $m$-row Young diagram $\lmb$ with rows $\lmb_1\geq \cdots \geq \lmb_m\geq 0$
we consider an object
\[
\obSLbl = \xPsiwv{\lmb_1}\cdots\xPsiwv{\lmb_m}\obOe \in \catB(\xxinv{\qn+\qk,0})
\]
related to the equivariant Slodowy slice of a nilpotent matrix corresponding to $\bldl$. The category of morphisms between $\obFLn$ and $\obSLbl$ is the category of sheaves on (the S-dual of) the Nakajima quiver variety related to the weight space $\bldl$ in the tensor product of $\qn$ fundamental representations of $\GL(m)$.

Most recently,
the related categories were studied by Nakajima, Takayama \cite{NakajimaTakayama17} in the context of 3D mirror symmetry (see section~\ref{sec:annular-homology} for more discussion).
The case of $\obSLv{\frac{\qn}{2},\frac{\qn}{2}}$  is related to the equivariant resolution of the Slodowy slice for the two-block nilpotent orbit in \(\mathfrak{gl}_{2n}\).
In  \cite{AnnoNandakumar16} it was conjectured that the corresponding  category
can be used to construct the annular $\GL(2)$ link homology \cite{AsaedaPrzytyckiSikora04}. We explain
a precise statement of the conjecture and its TQFT interpretation in section~\ref{sec:annular-homology}.




The box functors have a natural interpretation as defects in 3d B-models considered by Kapustin-Rozansky-Saulina \cite{KapustinSaulinaRozansky09},\cite{KapustinRozansky10}.
This allows us to expand our results in three directions.

First, we provide details of the mathematical construction of
some the structural morphisms of 3-category \(\hCat(Q)\) for a linear quiver.

Second, we use basic TQFT axioms to give a physical interpretation of the
Chern functor from our previous work \cite{OblomkovRozansky18a} and thus provide an explanation for the appearance of the Hilbert scheme of points on $\IC^2$
in knot homology \cite{OblomkovRozansky20}.

Third, we unify several homology theories  of the closures of an affine braid \(\beta\in \Br_n^{\aff}\) by interpreting them as the TQFT partitions functions on the annulus with external boundary consisting of \(n\) NS5 defects braided into the braid \(\beta\). More precisely, by setting an appropriate configuration of the defects on the internal boundary of the annulus  we obtain: Hochschild homology of the Rouquier complex \cite{KhovanovRozansky08b}, the Hilbert scheme version of the triply braided homology \cite{OblomkovRozansky20} and the trace of the braid action on \(\mathfrak{sl}_2\) conformal blocks.

While exploring the categorical foundations of Kapustin-Rozansky-Saulina (KRS) theory \cite{KapustinSaulinaRozansky09}, we avoid working with 3-category of KRS directly. Instead, we work
with the de-looped version \(\tCat(Q)\)  of the 3-category. The 3-categorical nature of the 2-category \(\tCat(Q)\) is reflected in the
2-categorical monoidal structure of \(\tCat(Q)\).
The categorical details can be found in section~\ref{sec:cats},
in the rest of the introduction we discuss the last two directions that are concerned the most topologically motivated case TQFT that is based on the three category
\(\hCat(\hat{A}_0)\).

\subsection{Partition-function evaluation for HOMFLY-PT homology}
The  initial construction in \cite{OblomkovRozansky16}
is based on the homomorphism from the braid group to a special category of matrix factorizations
which appears as
the category of endomorphism of the object  \(\calF_n\in\catB(\xxinv{n;1})\) \(\MF^{\rst}:=\Hom(\calF_n,\calF_n)\):
\begin{equation}\label{eq:MFst}
\Br_n\to \MF^{\rst}=\MF_{\GL_n}\bigl((\gl_n\ti \CC^n \ti \mathrm{T}^*\Fl\ti \mathrm{T}^*\Fl)^{\rst},W\bigr)\footnote{In the original
  paper we worked with the affine version of this category of matrix factorizations. },\quad W=\mu_1-\mu_2,
\end{equation}
where \(\Fl\) is the flag variety, \(\mu\) is the moment map of the $\GL_n$ action on $\rmTs\Fl$ and the index `st'  stands for the stability
condition.


The object \(\calF_n\) appears as a boundary condition for the \KRS\ theory with the target \[\xxinnr=\bigl(\rmTs\gl_n\times \Hom(\CC^r,\CC^n)\bigr)^{\rst}//\GL_n.\]
The cases \(r=0\) and \(r=1\) are the most relevant for the link homology. We refer to the first case as {\it unframed} and to the second case as {\it  framed}.
We choose the \sptm of all our 3d TQFTs to be
\(S^2\ti\mathbb{R}\).

The total \sptm\ \(S^2\ti\mathbb{R}\)
can have `defect surfaces' separating different models, that is, the connected components of the complement of defect surfaces are labeled by integers $n$. The defects can intersect along curves and curves can intersect at points. In the framed category, the defects have orientation
and in both settings the curves of the intersection of the defects have signs.

The  simplest case of the defect picture is when the defect is equal to \(\tDef=C\ti \mathbb{R}\subset S^2\ti \RR\) where \(C\) is an immersed
curve (possibly with many connected components). In this case we define the partition-function evaluation \(\sfZ^\bullet\) where
\(\bullet\) is \(\emptyset\) in the unframed case and
\(\bullet=\tfrm\)  in the framed case. Our choices in the construction of the partition-function evaluation are driven by the topological applications \cite{OblomkovRozansky16}
of the homomorphism \eqref{eq:MFst}:

\begin{theorem}\label{thm:MF}
  Let \(S^1\ti 0\subset S^2\ti \{0\}\subset S^2\times \RR \) be an embedded circle such that it intersect the defect \(\Def=C\ti \RR\) (\(C\) can have many connected components) transversally at \(2n\) points and the labels of the connected components of \(S^1\setminus S^1\cap \Def\) are \( 0,1,2,\dots,n,n-1,\dots,2,1\), as for example,  in the picture below\footnote{On the picture the red circle is \(S^1\) and the black lines are the defect curve \(C\)}:
  \[
    \begin{tikzpicture}
  \node (A0) at (-2,-1) {};
  \node (B0) at (-2,1)  {};
  \node (A) at (-1,-1) {};
  \node (B) at (-1,1) {};
  \node (A1) at (0,-1) {};
  \node (B1) at (0,1) {};
  \node (A2) at (1,-1) {};
  \node (B2) at (1,1)  {};
  \node (A3) at (2,-1) {};
  \node (B3) at (2,1) {};
  \node (Ln) at (-2.5,0) {\(\scriptstyle{n}\)};
  \node (Lnm1) at (-1.5,0) {\(\scriptstyle{n-1}\)};
  \node (Ldts) at (-0.5,0) {\(\dots\)};
  \node (L3) at (0.5,0) {\(\scriptstyle{2}\)};
  \node (L2) at (1.5,0.3) {\(\scriptstyle{1}\)};
  \node (L5) at (1.5,-0.3) {\(\scriptstyle{1}\)};
  \node (L0) at (2.5,0) {\(\scriptstyle{0}\)};
  \draw [->] (A0) edge (B0);
  \draw [->] (A) edge (B);
  \draw [->] (A1) edge (B1);
  \draw [->] (A2) edge (B3);
  \draw [->] (A3) edge (B2);
  \draw[red] (0,0) ellipse (3 and 0.5);
\end{tikzpicture}\]
 then
  \[\sfZ(S^1)=\MF_{\GL_n}(\gl_n \ti \mathrm{T}^*\Fl\ti \mathrm{T}^*\Fl,W).\]
  Moreover, if the first \(n\) intersection points of \(S^1\cap  \Def\) are oriented up and the rest down, as in the above picture, then
  \begin{equation}
  \label{eq:catzso}
  \sfZf(S^1)=\MF_{\GL_n}\bigl((\gl_n \ti\CC^n \times\mathrm{T}^*\Fl\ti \mathrm{T}^*\Fl)^{\rst},W\bigr).
  \end{equation}
\end{theorem}

Let us remark that in the last statement we evaluated  partition functions on the circle \(S^1\) and the values do not depend on the part of curve \(C\) that is away from 
\(S^1\). On the other hand the values of the partitions function on a disc \(D\subset S^2\times \{0\}\), \(\partial D=S^1\) depends on the part of curve \(C\) that is
inside \(D\). We discuss  values of the partition function on the discs below.


Let us use short-hand notation \(\MF^{\st}\) for \(\MF_{\GL_n}\bigl((\gl_n \ti\CC^n \times\mathrm{T}^*\Fl\ti \mathrm{T}^*\Fl)^{\rst},W\bigr).\)
In~\cite{OblomkovRozansky18a} we constructed the pair of adjoint functors:
\begin{equation}
\label{eq:mnchcchf}
  \begin{tikzcd}
    \MF^{\st}\arrow[rr,bend left,"\CH^{\st}_{\loc}"]&&\Dper(\Hilb_n(\CC^2))\arrow[ll,bend left,"\HC^{\st}_{\loc}",pos=0.435]
  \end{tikzcd},
\end{equation}
where \(\Hilb_n(\CC^2)\) is the Hilbert scheme of \(n\) points on \(\CC^2\), while \(\Dper(\Hilb_n(\CC^2))\) is the derived category of
two-periodic \(\Tqt\)-equivariant complexes on the Hilbert scheme. We also showed that
\begin{equation}\label{eq:Hilb}
  \HHH(\beta)=\mathcal{H}om(\CH^{\st}_{\loc}(\beta), \Lambda^\bullet \calB).
  \end{equation}

  Let us recall that \(\Hilb_n(\CC^2)\)  is a manifold that parameterizes ideals \(I\subset \CC[x,y]\) of codimension \(n\). Respectively, \(\calB\) is a rank \(n\) vector bundle
  with a fiber at \(I\) equal to the vector space dual to
  \(\CC[x,y]/I\). The vector bundle \(\calB\) is called {\it tautological } vector bundle.

The \tft\ picture gives a natural interpretation of the isomorphism~\eqref{eq:Hilb} as the result of gluing the same disc in two different ways. The appearance of \(\Hilb_n(\CC^2)\) is due to the following:
%

\begin{theorem}\label{thm:DrZ}
  Let \(K=S^1\ti 0\subset S^2\ti \RR \) be an embedded circle which does not  intersect the defect \(\Def=C\ti \RR\) and lies
  in the connected component of the complement of the defect with the label \(n\)
  (see Figure~\ref{fg:circlcut} for \(n=2\) example) then
  \[\sfZf(K)=\Dper(\Hilb_n(\CC^2)).\]
  Moreover, if \(D_\emptyset\) is a disc in the complement of the defect, such that \(\partial D_\emptyset=K\), then
  \[\sfZf(D_\emptyset)=\calO\in \Dper(\Hilb_n(\CC^2)).\]
\end{theorem}

To explain the appearance of the exterior powers of \(\calB\)
we introduce a special line of defect in our theory: \((0,0)\ti \RR\subset \RR^2\times \RR\subset 
S^2\ti \RR\), \(S^2=\RR^2\cup \infty\).
We assume that this line of defect does not intersect the surface of defect.
For a  small disc \(D_{taut}\) transversally   intersecting this line of defect we have:
\begin{equation}\label{eq:line-def}
\sfZf(D_{taut})=\Lambda^\bullet \calB\in \Dper(\Hilb_n(\CC^2)). \end{equation}
The link homology emerges as the vector space associated to a disc which intersects defect surfaces and a defect line:
\begin{theorem}\label{thm:HHH-TQFT}
  Suppose that the curve \(C\subset \RR^2\setminus (0,0)\subset \RR^2\cup \infty=S^2\) in the defect \(\Def=C\times \RR\) is the
  picture of the natural closure of the braid  \(\beta\in \Br_n\), the closure goes around the
  line of defect.
  We also assume the following constraint on labeling of the
  connected components of \(S^2\times \RR\setminus \Def\).
  The connected component that contains the infinite point \(\infty\) is labeled by  \(0\).
  The connected component that contains \((0,0)\) has label \(n\).
  The labels decrease by
  \(1\) as we cross the surface of defect while moving along a  ray that starts at \((0,0)\) and avoids singular points of \(\Def\).
  In the picture below
  \[\begin{tikzpicture}
      \node (nn) at (0,0.3) {\(\scriptstyle{n}\)};
      \node (nnn) at (0,0.7) {\(\scriptstyle{n-1}\)};
      \node (dd) at (0,1.3)  {\(\scriptstyle{\vdots}\)};
      \node (ff) at (0,1.8) {\(\scriptstyle{1}\)};
      \node (fff) at (0,2.2) {\(\scriptstyle{0}\)};
      \draw[->] (0.5,0) arc (0:359:0.5);
      \draw[->] (1,0) arc (0:359:1);
      \draw[->] (1.5,0) arc (0:359:1.5);
      \draw[->] (2,0) arc (0:359:2);
      \fill[white] (-0.3,-0.3) rectangle (-2.2,0.3);
      \fill[green] (0,0) circle (0.1);
      \draw[blue] (-0.3,-0.3) rectangle (-2.2,0.3);
      \node (bb) at (-1.3,0) {\(\scriptstyle{\beta}\)};
    \end{tikzpicture}
  \]
   the green dot represents a cross-section of the defect line.
  Then
  \begin{equation}\label{eq:geo-trace}
    \sfZf(S^2)=\HHH(\beta).\end{equation}
\end{theorem}


Thus the  formulas \eqref{eq:old} and \eqref{eq:Hilb} correspond to two ways to present \(S^2\) as a gluing of two
\(D^2\) along their common boundary $S^1$.  The  formula \eqref{eq:old} is given by  cutting along \(S^1\) with \(S^1\) as in theorem~\ref{thm:MF} and
the formula \eqref{eq:Hilb} is given by cutting along \(S^1\) that is boundary of a tubular neighborhood of the line of defect. The pictures \ref{fg:manyred}  and \ref{fg:circlcut} represent examples of these two cuts.

Let us denote by \(\pt_n\) a point in a connected component of \(S^2\times \RR\setminus \Def\) that has label \(n\).
Note that theorem~\ref{thm:DrZ} interprets \(\Dper (\Hilb_n(\CC^2))\) as the category of endomorphisms of the identity functor in
the two category \(\sfZf(\pt_n)\):
\[ \mathbb{H}om(\mathbb{I}d,\mathbb{I}d)=\Dper(\Hilb_n(\CC^2))\]
where \(\mathbb{I}d\) is the identity endomorphism of 2-category \(\sfZf(\pt_n)\).
 Thus it is reasonable to
 call \(\Dper(\Hilb_n(\CC^2))\) the Drinfeld center of 2-category \(\sfZf(\pt_n)\). Since it is not
 common to work with the Drinfeld centers of 2-categories, we spell out the expected property of such a
 center in corollary~\ref{cor:Dr}.

 \subsection{Unified perspective on traces}\label{sec:unified-tr}

 The space \(\HHH(\beta)\), \(\beta\in \Br_n\) is triply graded, two gradings come from \(\Tqt\)-action and the third is from
 the exponent in the exteriour power \(\Lambda^\bullet\calB\). The graded dimension \(\calP(\beta)=\dim_{q,t,a}\HHH(\beta)\) is sometimes
 called {\it super-polynomial} of the closure of \(\beta\),
 \cite{DunfieldGukovRasmussen06}. It is shown in \cite{OblomkovRozansky18c} that \(\calP(\beta)|_{t=-1}\) is equal to the HOMFLYPT polynomial \cite{Jones87}.
 
 Thus the equation \eqref{eq:geo-trace} can be interpreted as a categorification of the Jones-Oceanu trace \cite{Jones87} on \(\bigcup_n \Br_n\).
 Another construction for a categeorification of the Jones-Ocneanu trace The triply graded HOMFLYPT homology was given the papers \cite{KhovanovRozansky08b}, \cite{Khovanov07}. In the last cited paper, Rouquier's realization of
 braids inside Soergel bimodule is used. We denote the homology from the paper \cite{Khovanov07} by \(\HHH_{\mathrm{alg}}(\beta)\), \(\beta\in \Br_n\).

 We have shown in \cite{OblomkovRozansky20} that \(\HHH(\beta)=\HHH_{\mathrm{alg}}(\beta)\)
 for any \(\beta\in \Br_n\). So far, We do not have a physical explanation for this equality but both categorifications of the Jones-Oceanu trace
 fit in our TQFT picture.
 In this section we briefly discuss construction of \(\HHH_{\mathrm{alg}}\) and explain how  we construct \(\HHH_{\mathrm{alg}}\) by a slight modification of
 the TQFT construction for \(\HHH\) from the previous section. We also explain how one can modify the TQFT construction to obtain other interesting traces on
 \(\bigcup_n \Br_n\).

 Let us fix notations for the traces we plan to discuss. Soergel explained \cite{Soergel00} a realization of the braid graphs \(\gamma\in\Br_n^\flat\)
inside the additive category of Soergel bimodules \(\mathbb{S}\mathrm{bim}_n\). Thus the Hochschild homology functor \(\mathrm{HH}_*\) yields a trace on
\(\Br_n^\flat\):
\[\Phi^S: \Br_n\to \mathbb{S}\mathrm{bim}_n,\quad \mathrm{HHH}_{\alg}(\gamma)=\mathrm{HH}_*(\Phi^S(\gamma))=\mathcal{H}om(\Phi^S(\beta),\Phi^S(I)).\]

The last example that we discuss is related to the space of conformal blocks. Let us denote by \(V_\lambda\) and \(L_\lambda\) the Verma module and its irreducible  quotient for quantum group
\(U_q(\mathfrak{sl}_2)\) and \(\lambda\in \CC=\mathfrak{h})^*\). In particular, \(L_1=\CC^2\) is the vector representation of \(U_q(\mathfrak{sl}_2)\) and
\(L_\lambda\) is irreducible for not integer \(\lambda\) and \(L_\lambda=V_\lambda\). The space of conformal on the cylinder is defined as
\[H^\lambda_{1^n,\mu}=\Hom_{U_q(\mathfrak{sl}_2)}(L_\lambda,L_1\otimes\dots\otimes L_1\otimes L_\mu),\]
where we assume that \(\lambda\), \(\mu\) are not integers.

The \(R\)-matrix yields the braid group \(\Br_n\) action on the last space. Let's denote the trace of the action by \(\mathrm{Tr}(\beta)[H^\lambda_{1^n,\mu}]\). The space \(H^\lambda_{1^n,\mu}\) is trivial unless \(|\lambda-\mu|\le n\). On the other hand if \(\lambda-\mu=k\), \(|k|\le n\) then the above mentioned
trace does not depend on the value of \(\lambda\) as long as  \(\lambda\) is not integer.

The vector space \(H^\lambda_{1^n,\mu}\) has two gradings that originate
from the weight grading on \(L_\mu\) and the decomposition of \(L_1^{\otimes n}\) on the irreducible components. The braid action respects these gradings.
We propose a construction of  doubly graded vector space \(\mathcal{T}r(\beta)[H^\lambda_{1^n,\mu}]\) that categorifies the above trace.

Now let us turn to TQFT interpretation of the traces. A pictorial presentation of  the configuration of the defects that participate in the formula \eqref{eq:geo-trace} is the annulus with two boundary circles:
 \begin{equation}\label{pic:circles}
   \begin{tikzpicture}
    \draw (2,2) circle (1.5cm);
    \draw (2,2) circle (1.45cm);
    \draw  (2,2) circle (1.4cm);
    \draw[red,fill=red] (2,2) circle (0.5cm);
    \draw[blue,fill=white] (0.25,1.75) rectangle (0.75,2.25);
    \draw[white, fill=white] (1.5,0) rectangle (2.5,1);
    \node (b) at (0.5,2) {\(\beta\)};
    \node (B) at (2,2)  {\(\mathbb{B}\)};
    \node (NS) at (2,0.6) {\(n \mathrm{NS5}\)};
    \node (2) at (2,4) {\(\mathbf{0}\)};
    \node (z) at  (2,3) {\(\mathbf{n}\)};
    \end{tikzpicture}
  \end{equation}
  here \(\beta\in \Br_n\), the numeric labels indicate the type TQFT in the connected component.

  The picture is topologically equivalent to the pictures in theorem~\eqref{thm:HHH-TQFT} and for \(n=3\) and \(\beta=\sigma_1^3\) to the picture~\ref{fg:circlcut}. Let us explain appearance of labels NS5 in the last picture.
  In the section~\ref{sec:phys} we explain that in a string-theoretic
  presentation of our TQFT the defect \(\Def=C\times \RR\) is an \(S^1\)-reduction of the stack of \(n\) NS5 branes. 
Thus  in the picture we have \(n\) \(\mathrm{NS5}\) defects
following the outer boundary of the annulus and braid themselves into the braid \(\beta\).

The red domain contains the inner circle boundary defect configuration  \(\mathbb{B}\). In the case   treated by the equation \eqref{eq:geo-trace} the boundary configuration is the domain with the line defect
that carries the vector bundle \(\Lambda^\bl \calB\). The line defect is reflected by the equation \eqref{eq:line-def}.

It turns out that by varying the configuration of defects \(\mathbb{B}\) in the center we can categorify several other interesting trace evaluations on the
braid group as well as algebra of the braid graphs \(\Br_n^\flat\)\footnote{The elements of \(\Br_n\) are related to the braid graphs
by the MOY relations \cite{MurakamiOhtsukiYamada98}.}.

The three constructions are
partition-function evaluations  of  defect configuration as in the picture  and the internal disc defect \(\mathbb{B}\) as in the table
\begin{center}
  \begin{tabular}{c|c|c|c}
    \(\mathbb{B}\)& \(\Lambda^\bl V_n\) & \(\mathrm{NS5}^{(n)}\) & \(\mathrm{D5}^{(k)}+\mathrm{D5}^{(n-k)}\)\\
    \hline
    \(\beta\)& \(\Br_n\)&\(\Br_n^\flat\)&\(\Br_n\)\\
\(r\)&\(1\)&\(1\)&\(0\)                                                           \\
\(    \mathsf{Z}\)& \(\mathrm{HHH}(\beta)\)&\(\mathrm{HHH}_{\alg}(\beta)\)&\(\mathcal{T}r(\beta)[H^\lambda_{1^n,\lambda+k}]\)
  \end{tabular}
\end{center}
where \(\mathrm{NS5}^{(k)}\), \(\mathrm{D5}^{(k)}\) are NS5 and D5 defects of charge \(k\) discussed in the section~\ref{sec:ns5d5}.
The first two columns of the table reflect mathematically proven statements and the last column seems to be new mathematical construction.
We explain the mathematical details of the relations in the table in section~\ref{sec:traces}.

In section~\ref{sec:phys} we discuss the physics theories that motivate our key results. In details, we discuss in section~\ref{sec:KWandGW}
Kapustin-Witten theory with Gaiotto-Witten interfaces and give mathematical description of the interfaces in section~\ref{sec:interfaces1}.
In section~\ref{sec:4d-to-3d} we explain a reduction to 3d theory. Finally in sections \ref{sec:strings} and \ref{sec:CS-phys} we explain
a string theoretic perspective and the relation with the Chern-Simons theory, respectively.

In section~\ref{sec:algebraic-model-krs} we recall some basics of the constructions in~\cite{KapustinRozansky10} 
In the section~\ref{sec:cats} we explain how our particular example of \tft\ fits into the setting of
 \cite{KapustinRozansky10}. In section \ref{sec:defects-knot-invar} we
 construct the partition-function \(\sfZf\) and prove the results that we mentioned in the introduction. We also discuss
 the Drinfeld center subtleties in the subsection~\ref{sec:value-closed-curves}.

{\bf Acknowledgments}
We would like to thank Dmitry Arinkin, Tudor Dimofte, Eugene Gorsky, Sergey Gukov, Tina Kanstrup, Ivan Losev, Roman Bezrukavnikov and Andrei Negu{\c t} for useful discussions.
The authors also extremely grateful to an anonymous referee for many corrections
and important suggestion on the structure of the paper.
The work of A.O. was supported in part by  the NSF CAREER grant DMS-1352398, NSF FRG grant DMS-1760373 and Simons Fellowship.
The work of L.R. was supported in part by  the NSF grant DMS-1760578.

\def\Mf{M^4}
\def\gU{ \mathrm{U}}
\def\BGL{\mathbf{BGL}}
\def\dcell{cell}
\def\dcells{\dcell s}
\def\xobv#1{ (#1) }
\def\cXs{\mathcal{X}}
\def\cX{\mathcal{X}}

\def\cXsv#1#2{ \cXs_{n_{#1},n_{#2}}}

\def\rsm{\mathrm{sym}}

\def\cXsm{ \mathbf{X}^{\rsm}}
\def\cXsmvv#1#2{ \cXsm_{#1,#2}}
\def\cXsmot{ \cXsmvv{1}{2}}
\def\cXsmth{ \cXsmvv{2}{3}}

 \def\Bdd{ \ddot{\mathbf{B}}}

\def\Bddd{ \dddot{\mathbf{B}}}

\def\hmr{ /\!/ }
\def\GLv#1{\GL(#1) }
\def\GLn{ \GLv{n}}
\def\GLnv#1{ \GLv{n_{#1}}}
\def\GLno{ \GLnv{1}}
\def\GLnt{ \GLnv{2}}
\def\md{ m }
\def\KW{Kapustin-Witten}
\def\GW{Gaiotto-Witten}
\def\HW{Hanany-Witten}
\def\Hombv#1{ \mathbf{H}\mathrm{om}\bigl(#1\bigr)}
\def\cA{ \mathcal{A}}
\def\cAv#1#2{ \cA_{#1,#2}}
\def\cAnot{ \cAv{n_1}{n_2}}
\def\cB{ \mathcal{B}}
\def\cBv#1#2{ \cB_{#1,#2}}
\def\cBnot{ \cBv{n_1}{n_2}}
\def\cBnn{ \cBv{n}{n}}
\def\cC{ \mathcal{C} }
\def\cN{ \mathcal{N} }
\def\HmCv#1#2{ \Hom(\IC^{#1},\IC^{#2}) }
\def\gL{generalized Lagrangian}
\def\Lo{Lagrangian object}
\def\Lf{Lagrangian functor}
\def\cL{ \mathcal{L} }
\def\tcL{ \tilde{\cL}}
\def\avert{vertical}
\def\ahor{horizontal}
\def\tcZ{ \widetilde{\cZ}}
\def\tW{ \widetilde{W}}
\def\xcong{  = }
\def\tY{ \tilde{Y}}
\def\xDCoh{\mathrm{DCoh}}
\def\WCS{Witten-Chern-Simons}

\section{Physics Background}\label{sec:phys}

\subsection{The hierarchy of type-B categories}
From the mathematical perspective, we work within a 4-category $\BGL$ which, accoring to Kapustin and Witten~\cite{KapustinWitten07} represents the Galois side of the geometric Langlands duality. Let us review the pyramid of $\BGL$ from the bottom to the top. 

\subsubsection{The category of matrix factorizations}
At the first level lies the category of $G$-equivariant matrix factorizations $\MF^G(\cX;W)$, where $\cX$ is a (affine) variety with an action of an algebraic group $G$, while $W$ is a $G$-invariant function on $\cX$: $W\in\IC[\cX]^G$. The category $\MF^G(\cX;W)$ is the category of boundary conditions of the gauged Landau-Ginzburg B-model with the target space $\cX$ and the \emph{superpotential} $W$.

\subsubsection{The equivariant 2-category of a symplectic variety}
At the second level lies the 2-category $\Bdd^G(\cXsm)$ of a symplectic variety $\cXsm$ with a Hamiltonian action of $G$. If $\cXsm$ is a cotangent bundle
\[\cXsm =\rmTs\cX\] and the Hamiltonian action comes from the action of $G$ on $\cXs$, then an object of $\Bdd^G(\cXsm)$ is a pair $(\cZ;W)$, where $\cZ$ is an \emph{auxiliary} variety with an action of $G$, while $W\in\IC[\cZ\times\cX]^G$ in this context is the \emph{action} describing a Lagrangian submanifold in classical mechanics and $\cZ$ is the source of \emph{auxiliary variables}. Intuitively, the pair $(\cZ;W)$ represents a \emph{\Lo} in the Hamiltonian reduction $\cXsm\hmr G$, that is, a $G$-equivariant fibration $\tcL\rightarrow \cL$ whose base is a Lagrangian subvariety $\cL\subset \cXsm$. Here $\tcL\subset\rmTs(\cZ\times\cX)$ is the graph of $dW$.

The category of morphisms between two \Lo s is the category of $G$-equivariant matrix factorizations
\begin{equation}
\label{eq:tctzw}
\Hom\bigl((\cZ_1;W_1),(\cZ_2;W_2) \bigr) = \MF^G(\cZ_1\times\cZ_2\times\cX;W_2 - W_1).
\end{equation}

For two  varieties $\cXsm_1$ and $\cXsm_2$ a \Lo\ $(\cZ_{12};W_{12})$ in the 2-category $\Bdd^{G_1\times G_2}\bigl(\rmTs(\cX_1\times \cX_2)\bigr)$ determines a \emph{\Lf}
\[
\begin{tikzcd}[column sep=huge]\Bdd^{G_1}(\rmTs\cX_1)\ar[r,"(\cZ_{12};W_{12})"] & \Bdd^{G_2}(\rmTs\cX_2)
\end{tikzcd}
\]
stemming from the Lagrangian correspondence between $\rmTs\cX_1$ and $\rmTs\cX_2$.

The 2-category $\Bdd^G(\cXsm)$ is the category of boundary condition of the gauged 3d B-model with target $\cXsm$.

\subsubsection{The 3-category of $G$-equivariant symplectic varieties}
At the third level is a 3-category $\Bddd(G)$. Its objects are symplectic varieties $\cXsm$ with a Hamiltonian action of $G$. The 2-category of morphisms between two such varieties is the 2-category
of their product:
\[
\Hom(\cXsm_1,\cXsm_2) = \Bdd^G(\cXsm_1 \times \cXsm_2).
\]
Equivalently, a morphism is a (generalized) Lagrangian correspondence, and the morphisms are composed accordingly.
The 3-category $\Bddd(G)$  is the category of the boundary conditions of the \KW\ theory at $t=i$, which is related to the Galois side of the geometric Langlands duality.

\subsubsection{The 4-category of Lie groups of type A}
Finally, at the top of the pyramid is the 4-category $\BGL$ of the groups of type A:
\[\mathrm{Obj}(\BGL)=\mathbb{Z}_{\ge 0}\] that is an 
 object \(n\) corresponds to the group $\GLn$. The morphisms between two groups form the 3-category of their product:
\begin{equation}\label{eq:thrcat-gl-gl}
\mathbf{H}\mathrm{om}\bigl(\GLnv{1},\GLnv{2} \bigr) = \Bddd\bigl(\GLnv{1}\times\GLnv{2}\bigr)
\end{equation}
and the composition of morphisms $\cXsmot$ and $\cXsmth$ is their joint Hamiltonian reduction:
\[
\cXsmth\circ\cXsmot = (\cXsmot\times\cXsmth)\hmr \GLnt.
\]
The 4-category $\BGL$ is the category of \GW\ interfaces in the $t=i$ \KW\ theory.

\subsection{Kapustin-Witten theory with Gaiotto-Witten interfaces}\label{sec:KWandGW}
The 4-category $\BGL$ comes from the TQFT of Kapustin and Witten~\cite{KapustinWitten07} combined with the interfaces of Gaiotto and Witten~\cite{GaiottoWitten09} and inspired by the string theory setup of Hanany and Witten~\cite{HananyWitten97}.

Following~\cite{KapustinWitten07}, we consider a family of 4-dimensional (4d) $N=4$ supersymmetric Yang-Mills (SYM) theories with gauge groups $\gU(n)$, $n=0,1,\ldots$. The space-time is a 4d manifold $\Mf$ containing 3d (possibly mutually intersecting) submanifolds (interfaces) of two types: NS5 and D5. These interfaces separate $\Mf$ into \dcells, each \dcell\ being assigned a particular value of $n$. The interface submanifolds may contain their own 2d interfaces-submanifods, splitting 3d interfaces into 3d \dcells, etc. Topologically, the whole construction is an example of a 4d `foam' (a smooth $CW$ complex) with a particular property that the total ambient space $\Mf$ is just a manifold.

Following~\cite{KapustinWitten07}, we consider the topological version of our SYM theory: we apply the GL-twist and choose the differential corresponding to $t=i$. As explained by Kapustin and Witten, the resulting theory corresponds to the Galois side of the Langlands duality. The gauge groups $\gU(n)$ are effectively complexified to $\GL(n)$.

Mathematically, a TQFT on a 4d manifold with interfaces is equivalent to a 4-category. Each 4d \dcell\ is colored by its object, a \dcell\ of a 3d interface is colored by a morphism between the adjoint \dcells, etc. In order to determine a category whose objects should be assigned to an $\md$-dimensional \dcell, one has to compactify the full theory on its ($3-\md$)-dimensional link. The \dcell\ becomes a boundary of the resulting $(\md+1)$-dimensional TQFT and the category of the \dcells\ emerges as the category of its boundary conditions.

Consider a 3d interface in the \KW\ theory  separating the gauge groups $\GLno$ and $\GLnt$. Consider the \KW\ theory at $t=i$. The link of a 3d interface is two points, and the compactification amounts to folding the 4d space across the interface resulting in a 4d half-space whose boundary is the 3d interface and whose gauge group is $\GLno\times\GLnt$. According to \GW, the boundary condition is described by a symplectic variety $\cXsm$ with the hamiltonian action of $\GLno\times\GLnt$, that is, by an object of $\BGL$.

Suppose that the 3d interface has a 2d interface inside, separating the boundary conditions $\cXsm_1$ and $\cXsm_2$. The link of this interface is a semicirlce ending at the boundary of the 4d halfspace, so compactifications results in a 3d half-space whose theory is the 3d gauged B-model with the target $\cXsm_1\times\cXsm_2$ and the gauge group $\GLno\times\GLnt$. The 2-category of the boundary conditions of this theory is~\eqref{eq:tctzw} as explained (in the non-equvariant case) in~\cite{KapustinSaulinaRozansky09,KapustinRozansky10}. In particular, an object $(\cZ;W)$ describes a combination of a Lagrangian boundary condition with an added 2d B-model at the boundary, while the category of matrix factorizations describing interfaces between two such boundary conditions is the category of boundary conditions of the 2d Landau-Ginzburg B-model emerging after the folding of the 3d halfspace carrying the 3d B-model.

\subsection{Bow (D5) and arrow (NS5) interfaces}\label{sec:interfaces1}
Gaiotto and Witten suggest the version of the 4-category $\BGL$ in which all morphisms (that is, 3d interfaces) are compositions of two types of elementary ones: the NS5 interface and the D5 interface. Recall that any interface in $\Hombv{\GLno,\GLnt}$ is represented by a symplectic variety with the Hamiltonian action of $\GLno\times\GLnt$. The NS5 and D5 varieties also emerge as edges in the Nakajima-Cherkis quiver varieties, the NS5 variety corresponding to an arrow edge and the D5 variety corresponding to a bow edge. Ignoring the stability conditions, these edge-related varieties have a form 
\[
\prod_{(i,j)}\rmTs \HmCv{n_i}{n_j}\hmr\prod_{i} G_i = \rmTs
\Bigl( \,\prod_{(i,j)}\HmCv{n_i}{n_j}\Bigl/\prod_i G_i \Bigr).
\]
In subsequent pictures we mark only the arrows of $\HmCv{n_i}{n_j}$ (ignoring the cotangent arrows) because only these arrows (together with the auxiliary ones marked as dashed) are the variables of the superpotential $W$ in \Lo s $(\cZ;W)$.

Ignoring the stability condition, the arrow variety is
\begin{equation}
    \label{eq:arrvar}
\cAnot = \rmTs\Hom(\IC^{n_1},\IC^{n_2}),
\end{equation}
the action of $\GLno\times\GLnt$ on $\Hom(\IC^{n_1},\IC^{n_2})$ being
\[
(g_1,g_2)\cdot X = g_2\,X g_1^{-1}.
\]
Pictorially,
\[
\begin{tikzpicture}[baseline=0cm]
\draw[ultra thick,-] (-2,0) -- (2,0);
\draw[thick,fill=white] (0,-1.5) circle [radius=0.15];
\draw[thick,fill=white] (0,1.5) circle [radius=0.15];
\draw[thick,<-] (0,1.35) -- (0,-1.35);
\node[right] at (0.1,1.5) {$n_2$};
\node[right] at (0.1,-1.5) {$n_1$};
\node[right] at (2,0) {NS5};
\node[right] at (0,0.75) {$X$};
\end{tikzpicture}
\]
where $X\in \Hom(\IC^{n_1},\IC^{n_2})$. 

The description of the bow variety $\cBnot$ is more complicated, but the case of $n_1=n_2=n$ is straightforward:
\[
\cBnn=\rmTs\GL(n)\times \rmTs\IC^n,\qquad
(g_1,g_2)\cdot (g,v) = (g_2\,g\,g_1^{-1},g_2\,v)
\]
(one can replace $g_2\,v$ with $g_1\,v$). The role of $\rmTs\GL(n)$ factor is to reduce the gauge group $\GL(n)\times\GL(n)$ down to its diagonal subgroup $\GL(n)$. Pictorially,
\[
\begin{tikzpicture}[baseline=0cm]
\draw[ultra thick,-] (-2,0) -- (2,0);
\draw[thick,fill=white] (0,-1.5) circle [radius=0.15];
\draw[thick,fill=white] (0,1.5) circle [radius=0.15];
\draw[thick,<-] (0,1.35) -- (0,-1.35);
\node[right] at (0.1,1.5) {$n$};
\node[right] at (0.1,-1.5) {$n$};
\node[right] at (2,0) {D5};
\node[right] at (0,0.75) {$g$};
\draw[thick,fill=white] (0.85,-0.15) rectangle ++(0.3,0.3);
\draw[thick,->] (0.1,-1.35) -- (1,-0.2);
\node[right] at (0.5,-0.85) {$v$};
\node[above] at (1,0.15) {$1$};
\node at (4,0) {or};
\draw[ultra thick,-] (6,0) -- (10,0);
\node[right] at (10,0) {D5};
\draw[thick,fill=white] (7,0) circle [radius=0.15];
\draw[thick,fill=white] (8.85,-0.15) rectangle ++(0.3,0.3);
\draw[thick,->] (7.11,0.11) .. controls (7.7,0.6) and (8.3,0.6) .. (8.85,0.15);
\node[above] at (8.0,0.5) {$v$};
\node[below] at (7,-0.2) {$n$};
\node[below] at (9,-0.2) {$1$};
\end{tikzpicture}
\]
the second picture corresponding to the diagonal $\GL(n)$ equivariance.

With these choices, the composition of elementary morphisms in the category $\BGL$ produces a Nakajima quiver variety of type $A$:
\begin{equation}
    \label{eq:quivo}
\begin{tikzpicture}
\draw[ultra thick,-] (-2,-1) -- (-2,2) (0,-1) -- (0,2) (2,-1) -- (2,2);
\draw[thick,fill=white] (0,0) circle [radius=0.15];
\draw[thick,fill=white] (-0.15,1.35) rectangle ++(0.3,0.3);
\draw[thick,fill=white] (-4,0) circle [radius=0.15];
\draw[thick,fill=white] (4,0) circle [radius=0.15];
\node[above] at (-2,2) {NS5};
\node[above] at (2,2) {NS5};
\node[above] at (0,2) {$r$ D5};
\draw[thick,->] (-0.12,0.12) .. controls (-0.4,0.5) and (-0.4,1) .. (-0.15,1.35);
\node [left] at (-0.35,0.75) {$v$};
\draw[thick,->] (-3.85,0) -- (-0.15,0);
\draw[thick,->] (0.15,0) -- (3.85,0);
\node[below] at (-4,-0.17) {$n_1$};
\node[below] at (-0.3,-0.17) {$n_2$};
\node[below] at (4,-0.17) {$n_3$};
\node[above] at (-3,0) {$X_1$};
\node[above] at (3,0) {$X_2$};
\end{tikzpicture}
\end{equation}

\subsection{From 4d to 3d}\label{sec:4d-to-3d}
\subsubsection{Compactification}\label{sec:compact}

We consider the \KW\ theory on the 4-manifold
\begin{equation}
\label{eq:mfour}
M^4 = \IRz\times\xSg\times \Soh,
\end{equation}
where $\xSg$ is a 2-dimensional surface, $\IRvv{k}{i_1\cdots i_k}$ denotes $\IR^k$ with coordinates $x_{i_1},\ldots,x_{i_k}$ and $S^1_i$ denotes the circle $S^1$ with the periodic coordinate $x_i$. In fact, we will be mostly interested in the case of $\xSg = \IRvv{2}{12}$, which is a local form of a general $\xSg$.


We place simultaneously two arrangements of parallel elementary NS5 and D5 interfaces  in $M^4$.
The first \emph{\ahor} arrangement consists of parallel NS5 and D5-interfaces, each interface sweeping a 3-dimensional subspace
$
\IRz\times \xSg\times\{s_I\}$, where $s_I\in \Soh$ is a point on the circle corresponding to $x_3 = s_I$. We assume that the values of $\xn$ are the same on both sides of each D5-interface, but they can vary across NS5-interfaces.

An interface $\xInt$ of the second \emph{\avert} arrangement has the form $\IRvv{1}{0}\times\xcraI\times\Soh$, where $\xcraI\subset \xSg$ is a curve. The curves of different additional interfaces may intersect.


Fix a point $\brt\in \IRz$ and a curve $\ctcr\subset \xSg$ possibly intersecting the curves $\xcraI$ of additional interfaces. Consider a 2d submanifold $\cC$ and a 3d submanifold $\cN$ of $M^4$:
\[
\cC = \{\brt\}\times \ctcr \times \Soh,\qquad \cN = \{\brt\}\times \xSg \times \Soh.
\]
We want to describe the category of $\cC$ and the vector space (that is, the Hilbert space) of $\cN$. Since both manifolds have $\Soh$ as a factor, we can first compactify $M^4$ on $\Soh$, thus reducing it to
$M^3 = \IRz\times\xSg$, which is split by \avert\ interfaces into domains. Each domain of $M^3$ carries a 3d topological `Coulomb-twisted' $N=4$ SYM theory with matter fields.

The B-twisted Higgs branch of this theory 
is the Nakajima quiver variety $\cM_{\bqn,\bqr}(\tAm)$ whose affine quiver $\tAm$ and the dimensions at circles $\bqn$ and at framing boxes $\bqr$
are determined by the positions of the \ahor\ arrangement interfaces on $\Soh$ and the gauge groups $\GL(n_i)$ between them in accordance with the picture~\eqref{eq:quivo}:
$\nsm$ is the number of NS5-interfaces in the basic arrangement, the dimensions of the circles are $\bqn=(n_1,\ldots,n_{\nsm})$ and the dimensions of the framing boxes $\bqr=(r_1,\ldots,r_{\nsm})$ equal the numbers of D5 interfaces between each pair of adjacent NS5 interfaces. If one of the numbers $n_i$ is zero, then the affine quiver $\tAm$ becomes an $A_{m-1}$ quiver.

Let $Q$ be either the $\tAm$ or the $A_{m-1}$ quiver. We denote by $\xQbr$ the assignment of dimensions $\bqr$ to the framing boxes and by $\xQbrn$ the additional assignment of dimensions $\bqn$ to the circles.

In mathematical language the compactification on a circle $\Soh$ intersecting the \avert\ brane arrangement means that we reduce the 4-category $\BGL$ to the 3-category $\hCat(\xQbr)$.
An object of $\hCat(\xQbr)$ is $\xQbrn$ and the 2-category of morphisms is the subcategory
\[\mathbb{H}\mathrm{om}\bigl(\xQbrn,\xQbrnp\bigr)\subset\Bdd\bigl( \cM_{\bqn,\bqr}\times\cM_{\bqn',\bqr}\bigr)\] formed by \Lf s corresponding to the NS5 branes of the \avert\ arrangement and described in subsection~\ref{sbs:lfs}.

\subsubsection{Elementary intersections}
From the 3d TQFT perspective, the original 3d \avert\ interfaces become 2d interfaces between quiver-related gauged 3d B-models. Hence a \avert\ interface is a \Lf\ between two $\tAm$ quiver varieties with different dimensions at circles. This \Lf\ is a combination of \Lf s assigned to the intersections of the \avert\ interface with the \ahor\ interfaces.

\def\cZot{ \cZ_{12}}
\def\cWot{ W_{12}}
\def\rmav{\mathrm{vert}}
\def\rmah{\mathrm{hor}}
There are two types of the vertical NS5 branes
that appear in our discussion, see section~\ref{sec:single-brane} for a string-theoretic interpretation of these two types of NS5 branes. The first type of branes we denote \(NS5\), in the language of section~\ref{sec:single-brane} these are
branes that wrap cigar \(\RR^2_{\dot{3}\dot{7}}\). A vertical NS5
brane of the second type wraps the cigar \(\RR^2_{\dot{8}\dot{9}}\) and we use notation
\(NS5'\) for such branes.

The intersection of  a vertical NS5 brane with a horizontal NS5 brane  results in a 2d defect that is an object in the 2-category $\Bdd$ of the $\tA_3$ quiver.
This 2-category has two simplest objects: $(\cZ;W)$ and its Legendre dual   $(\tcZ;\tW)$
which correspond to the intersection of a horizontal NS5 brane with a vertical \(NS5\)
and \(NS5'\), respectively:
%
\begin{equation}\label{eq:crcrc}
\begin{tikzpicture}
\draw[ultra thick,-] (0,2) -- (0,-2);
\draw[ultra thick,-] (-3,0) -- (3,0);
\draw[thick,fill=white] (-2,1.5) circle [radius=0.15];
\draw[thick,fill=white] (-2,-1.5) circle [radius=0.15];
\draw[thick,fill=white] (2,1.5) circle [radius=0.15];
\draw[thick,fill=white] (2,-1.5) circle [radius=0.15];
\node[right] at (3,0) {$\mathrm{NS5}_{\mathrm{\rmah}}$};
\node[above] at (0,2.2) {$\mathrm{NS5}_{\mathrm{\rmav}}$};
\draw[thick,->] (-1.85,1.5) -- (1.85,1.5);
\draw[thick,->] (-1.85,-1.5) -- (1.85,-1.5);
\draw[very thick,->] (-2,-1.35) -- (-2,1.35);
\draw[very thick,->] (2,-1.35) -- (2,1.35);
\node[left] at (-2.2,-1.5) {$n_1$};
\node[right] at (2.2,-1.5) {$n_1'$};
\node[left] at (-2.2,1.5) {$n_2$};
\node[right] at (2.2,1.5) {$n_2'$};
\node[left] at (-2.1,0.75) {$X$};
\node[right] at (2.1,0.75) {$X'$};
\node[above] at (1,1.5) {$\varphi_2$};
\node[below] at (1,-1.5) {$\varphi_1$};
\draw[thick,dashed,->] (1.9,1.4) .. controls (0,0.5) and (-0.7,0) .. (-1.9,-1.4);
\node[above, left] at (-0.3,0.3) {$\psi$};
\end{tikzpicture}
\qquad
\begin{tikzpicture}
\draw[ultra thick,-] (0,2) -- (0,-2);
\draw[ultra thick,-] (-3,0) -- (3,0);
\draw[thick,fill=white] (-2,1.5) circle [radius=0.15];
\draw[thick,fill=white] (-2,-1.5) circle [radius=0.15];
\draw[thick,fill=white] (2,1.5) circle [radius=0.15];
\draw[thick,fill=white] (2,-1.5) circle [radius=0.15];
\node[right] at (3,0) {$\mathrm{NS5}_{\mathrm{hor}}$};
\node[above] at (0,2.2) {$\mathrm{NS5'}_{\mathrm{vert}}$};
\draw[thick,->] (-1.85,1.5) -- (1.85,1.5);
\draw[thick,->] (-1.85,-1.5) -- (1.85,-1.5);
\draw[very thick,->] (-2,-1.35) -- (-2,1.35);
\draw[very thick,->] (2,-1.35) -- (2,1.35);
\node[left] at (-2.2,-1.5) {$n_1$};
\node[right] at (2.2,-1.5) {$n_1'$};
\node[left] at (-2.2,1.5) {$n_2$};
\node[right] at (2.2,1.5) {$n_2'$};
\node[left] at (-2.1,0.75) {$X$};
\node[right] at (2.1,0.75) {$X'$};
\node[above] at (1,1.5) {$\varphi_2$};
\node[below] at (1,-1.5) {$\varphi_1$};
\draw[thick,dashed,<-] (-1.9,1.4) .. controls (0,0.5) and (0.7,0) .. (1.9,-1.4);
\node[above, left] at (-0.3,0.3) {$\psi$};
\draw[thick,dashed,->] (-1.9,1.4) .. controls (-1.5,0.5) and (-1.5,-0.5) .. (-1.9,-1.4);
\draw[thick,dashed,->] (1.9,1.4) .. controls (1.5,0.5) and (1.5,-0.5) .. (1.9,-1.4);
\node[right] at (-1.7,-1) {$\tX$};
\node[left] at (1.7,1) {$\tX'$};
\end{tikzpicture}
\end{equation}
The auxiliary variety and the superpotential of the first object are
\[
\cZ = \Hom(\IC^{n_2'},\IC^{n_1}),\qquad
W = \Tr\psi(X'\varphi_1-\varphi_2X).
\]
The Legendre dual object has  auxiliary variety 
\[
\tcZ = \cZ\times\HmCv{n_2}{n_1}\times\HmCv{n_2'}{n_1'},
\]
while the superpotential becomes
\[
\tW =  W 
- \Tr\tX X - \Tr\tX' X'.
\]


The intersection of an NS5 interface with a stack of $r$ D5 interfaces results in a 2d defect from the category of the quiver $A_2$ with the common framing space, the corresponding object being $(\cZ;W)$:
\begin{equation}\label{eq:crbx}
\begin{tikzpicture}
\draw[ultra thick,-] (0,2.5) -- (0,-1.5);
\draw[ultra thick,-] (-3,0) -- (3,0);
\draw[thick,fill=white] (-2,0) circle [radius=0.15];
\draw[thick,fill=white] (2,0) circle [radius=0.15];
\draw[thick,fill=white] (-0.15,-0.15) rectangle ++(0.3,0.3);
\draw[thick,->] (-1.9,0.1) .. controls (-1.4,0.7) and (-0.6,0.7) .. (-0.15,0.15);
\draw[thick,->] (1.9,0.1) .. controls (1.4,0.7) and (0.6,0.7) .. (0.15,0.15);
\draw[thick,->] (-1.9,-0.1) .. controls (-1,-1.5) and (1,-1.5) .. (1.9,-0.1);
\draw[thick,dashed,<-] (-1.9,0.1) .. controls (-1.35,2) and (-0.65,2) .. (0,0.15);
\node[above] at (0,2.5) {NS5};
\node[right] at (3,0) {$r$D5};
\node[below] at (1.2,-0.7) {$\varphi$};
\node[above] at (1,0.5) {$v_2$};
\node[above] at (-1,0.5) {$v_1$};
\node[above] at (-1.2,1.6) {$\xfph$};
\node[below right] at (0.15,-0.15) {$r$};
\node[below] at (-2,-0.15) {$n_1$};
\node[below] at (2,-0.15) {$n_2$};
\end{tikzpicture}
\qquad
\qquad
\begin{tikzpicture}[scale=0.75]
\draw[thick, fill=white] (-2,-2) circle [radius = 0.15];
\draw[thick, fill=white] (2,-2) circle [radius = 0.15];
\draw[thick, fill=white] (-2.15,1.85) rectangle ++(0.3,0.3);
\draw[thick, fill=white] (1.85,1.85) rectangle ++(0.3,0.3);
\draw[thick,->] (-1.85,-2) -- (1.85,-2);
\draw[very thick,<->] (-1.85,2) -- (1.85,2);
\draw[thick,->] (-2,-1.85) -- (-2,1.85);
\draw[thick,->] (2,-1.85) -- (2,1.85);
\node[above] at (0,2) {$\xcong$};
\draw[thick,dashed,->] (1.85,1.85) -- (-1.90,-1.90);
\node[left] at (-2,0) {$v_1$};
\node[right] at (2,0) {$v_2$};
\node[below] at (-2,-2.15) {$n_1$};
\node[below] at (2,-2.15) {$n_2$};
\node[above] at (-2,2.15) {$r$};
\node[above] at (2,2.15) {$r$};
\node[above] at (0,-2) {$\varphi$};
\node[left] at (-0.05,0) {$\xfph$};
\end{tikzpicture}
\end{equation}
the second picture omitting the branes and being more quiver-style,
\[
\cZ = \HmCv{n_2}{n_1},\qquad
W = \Tr\,\xfph(v_2\,\varphi-v_1).
\]
Note that the spaces $\IC^r$ in the second picture of~\eqref{eq:crbx} are canonically identical.

\subsubsection{A Lagrangian correspondence between two quiver varieties}
\label{sbs:lfs}
Now we can describe a (formerly \avert) interface between two $\tAm$ quiver related domains in the 3d TQFT. We assume that both $\tAm$ quivers share the same dimensions at boxes while having generally different dimensions at circles. 

In order to assign a pair $(\cZ;W)$ to the interface, we put both quivers side by side. For each pair of matching circle to circle edges we put one of two squares~\eqref{eq:crcrc} and for each pair of matching circle to box edges we put a square~\eqref{eq:crbx}. Here is a portion of the resulting diagram with two such squares:
\begin{equation}\label{pic:vert-hor-squares}
\begin{tikzpicture}
\draw[thick, fill=white] (0,0) circle [radius = 0.15];
\draw[thick, fill=white] (3,0) circle [radius = 0.15];
\draw[thick, fill=white] (2,1.5) circle [radius = 0.15];
\draw[thick, fill=white] (5,1.5) circle [radius = 0.15];
\draw[thick,fill=white] (-0.15,2.85) rectangle ++(0.3,0.3);
\draw[thick,fill=white] (1.85,4.35) rectangle ++(0.3,0.3);
\draw[thick,->] (0.15,0) -- (2.85,0);
\draw[thick,->] (2.15,1.5) -- (4.85,1.5);
\draw[thick,->] (0,0.15) -- (0,2.85);
\draw[thick,->] (2,1.65) -- (2,4.35);
\draw[very thick,<->] (0.15,3) -- (1.85,4.5);
\draw[thick,dashed,->] (1.9,1.4) -- (0.1,0.1);
\draw[thick,dashed,->] (4.9,1.4) -- (3.1,0.1);
\draw[thick,dashed,->] (0.15,2.85) -- (1.9,1.6);
\draw[thick,dashed,->] (2.1,1.4) -- (2.9,0.1);
\draw[very thick,dotted,-] (-0.15,0) -- (-0.75,0);
\draw[very thick,dotted,-] (1.85,1.5) -- (1.25,1.5);
\draw[very thick,dotted,-] (3.15,0) -- (3.75,0);
\draw[very thick,dotted,-] (5.15,1.5) -- (5.75,1.5);
\draw[very thick,dotted,-] (3,0.15) -- (3,0.75);
\draw[very thick,dotted,-] (5,1.65) -- (5,2.25);
\node[below] at (0,-0.15) {$n_i'$};
\node[below] at (3,-0.15) {$n_{i+1}'$};
\node[above right] at (2.1,1.6) {$n_i$};
\node[above right] at (5.1,1.6) {$n_{i+1}$};
\node[above] at (0,3.15) {$r_i$};
\node[above] at (2,4.65) {$r_i$};
\node[below] at (1.5,0) {$X_i'$};
\node[above] at (3.5,1.5) {$X_i$};
\node[above left] at (1,3.75) {$\xcong$};
\node [above left] at (1,0.75) {$\varphi_i$};
\node [below right] at (4,0.75) {$\varphi_{i+1}$};
\node[left] at (0,1.5) {$v_i'$};
\node[right] at (2,3) {$v_i$};
\node[above right] at (1,2.25) {$\xfph_i$};
\node[below left] at (2.6,0.85) {$\psi_i$};
\end{tikzpicture}
\end{equation}
The auxiliary variety $\cZ$ is formed by all dashed arrows: $\varphi_i$, $\psi_i$ and $\xfph_i$ as well as $\tX_i$ and $\tX_i'$, while the superpotential $W$ is the sum of all individual superpotentials.

\subsection{Examples of inter-quiver Lagrangian correspondences}
\subsubsection{NS5 interfaces between Grassmannians}
\label{sec:Grass-interf}
This example was communicated to us by Tudor Dimofte and it represents an interface between two $A_1$ quivers, obtained from $\tA_1$ quivers by setting $\qn_2=\qr_2=0$.
If we impose
the Grassmannian stability conditions requiring that $\qvc$ and $\qvc'$ have highest ranks ($\qn$ and $\qn'$), then the corresponding quiver varieties are
cotangent bundles $\rmTs\Grv{\qn,\qr}$ and $\rmTs\Grv{\qn',\qr}$. If $\qn\leq \qn'$, then the first diagram of~\eqref{eq:dgsq2} yields a well-known \Lf:
\[
\begin{tikzpicture}[baseline=1.5cm]
\draw[thick,->] (0,0.15) -- (0,2.85);
\draw[thick,->] (3,0.15) -- (3,2.85);
\node[left] at (0,1.5) {$\qvc$};
\node[right] at (3,1.5) {$\qvc'$};
\draw[dashed, ->] (0.15,0) -- (2.85,0);
\draw[very thick, <->] (0.15,3) -- (2.85,3);
\node[below] at (1.5,0) {$\xphi$};
\node[above] at (1.5,3) {$\xcong$}; 
\draw[dashed, ->] (2.85,2.85) -- (0.15,0.15);
\node[above] at (1.4,1.5) {$\xfph$};
\draw[thick,fill=white] (0,0) circle [radius=0.15];
\draw[thick,fill=white] (3,0) circle [radius=0.15];
\draw[thick,fill=white] (-0.15,2.85) rectangle ++(0.3,0.3);
\draw[thick,fill=white] (2.85,2.85) rectangle ++(0.3,0.3);
\node[left] at (-0.15,0) {$\qn$};
\node[left] at (-0.15,3) {$\qr$};
\node[right] at (3.15,0) {$\qn'$};
\node[right] at (3.15,3) {$\qr$};
\end{tikzpicture},\qquad
\xW = \Tr\, \xfph(\qvc'\xphi -   \qvc).
\]
whose Lagrangian correspondence is the conormal bundle to the subvariety of $\Grv{\qn,\qr}\times\Grv{\qn',\qr}$ determined by the condition 
\begin{equation}
    \label{eq:cnnin}
    \IC^{n}\subset\IC^{n'}\subset\IC^{r}.
\end{equation}
%
Indeed, the criticality condition for $\xW$ with respect to $\xfph$ implies the commutativity of the vertical and horizontal arrows, so $v = v'\varphi$ and  subdiagram of the arrows $\varphi$ and $v'$
determines the partial flag formed by the subspaces of both Grassmannians, as stated in~\eqref{eq:cnnin}. $\IC^{\qn}\subset\IC^{\qn'}\subset\IC^{\qr}$. In addition, the criticality condition with respect to $\xphi$ guarantees that the map $\xfph$ represents the vector of the conormal bundle to the condition~\eqref{eq:cnnin}.


The superpotential and the
space of maps \(\varphi,
\) realizes NS5 defect of charge \(n-n'\),
\(\mathrm{NS5}^{(n-n')}\)
in three category \(\thc(A_0)\). We use the super-index  in the notation of
NS5 branes to indicate a difference of labels on two sides of the defect:
the difference of labels on two sides of  \(\mathrm{NS5}^{(k)}\)  is \(k\).

%
%

\subsubsection{NS5  interfaces between instanton moduli spaces}\label{sec:ns5-inst}

Now we consider the interfaces between two $\tAz$ quivers with $\qn'=\qn+k$ and $\qr=\qr'$ which yield the functors $\xPhiwk$ and $\xPhibk$ of \eqref{eq:ns5sq}.
The NS5 interfaces \(\mathrm{NS5}^{(n-n')}\) of charge \(n-n'\) are described by diagrams
\begin{equation}
\label{eq:instcr}
\begin{tikzpicture}[baseline=1.5cm]
\draw[thick,->] (0,0.15) -- (0,2.85);
\draw[thick,->] (3,0.15) -- (3,2.85);
\node[left] at (0,1.5) {$\qvc$};
\node[right] at (3,1.5) {$\qvc'$};
\draw[dashed,->] (0.15,0) to [out=-15,in=195] (2.85,0);
\draw[dashed,<-] (0.15,0) to [out=15,in=165] (2.85,0);
\draw[very thick,<->] (0.15,3) -- (2.85,3);
\node[below] at (1.5,-0.2) {$\xphi$};
\node[above] at (1.5, 0.2) {$\xpsi$};
\node[above] at (1.5,3) {$\xcong$};
\draw[dashed,->] (2.85,2.85) -- (0.15,0.15);
\node[above] at (1.4,1.5) {$\xfph$};
\draw[thick,-] (-0.15,-0.15) to [out=200,in=180] (0,-1.2);
\draw[thick,<-] (0.15,-0.15) to [out=-20,in=0] (0,-1.2);
\node[below] at (0,-1.2) {$X$};
\draw[ thick,-] (2.85,-0.15) to [out=200,in=180] (3,-1.2);
\draw[thick,<-] (3.15,-0.15) to [out=-20,in=0] (3,-1.2);
\node[below] at (3,-1.2) {$X'$};
\draw[thick,fill=white] (0,0) circle [radius=0.15];
\draw[thick,fill=white] (3,0) circle [radius=0.15];
\draw[thick,fill=white] (-0.15,2.85) rectangle ++(0.3,0.3);
\draw[thick,fill=white] (2.85,2.85) rectangle ++(0.3,0.3);
\node[left] at (-0.15,0) {$\qn$};
\node[left] at (-0.15,3) {$\qr$};
\node[right] at (3.25,0) {$\qn'$};
\node[right] at (3.15,3) {$\qr$};
\end{tikzpicture}\qquad\qquad
\begin{tikzpicture}[baseline=1.5cm]
\draw[thick,->] (0,0.15) -- (0,2.85);
\draw[thick,->] (3,0.15) -- (3,2.85);
\node[left] at (0,1.5) {$\qvc$};
\node[right] at (3,1.5) {$\qvc'$};
\draw[->,dashed] (0.15,0) to [out=-15,in=195] (2.85,0);
\draw[<-,dashed] (0.15,0) to [out=15,in=165] (2.85,0);
\draw[<->,very thick] (0.15,3) -- (2.85,3);
\node[below] at (1.5,-0.2) {$\xphi$};
\node[above] at (1.5, 0.2) {$\xpsi$};
\node[above] at (1.5,3) {$\xcong$}; 
\draw[->,dashed] (2.85,2.85) -- (0.15,0.15);
\node[above] at (1.4,1.5) {$\xfph$};
\draw[thick,-] (-0.15,-0.15) to [out=200,in=180] (0,-1.2);
\draw[thick,<-] (0.15,-0.15) to [out=-20,in=0] (0,-1.2);
\node[below] at (0,-1.2) {$X$};
\draw[<-,dashed] (-0.15,-0.15) to [out=200,in=180] (0,-2);
\draw[-,dashed] (0.15,-0.15) to [out=-20,in=0] (0,-2);
\node[below] at (0,-2) {$\tY$};
\draw[thick,-] (2.85,-0.15) to [out=200,in=180] (3,-1.2);
\draw[thick,<-] (3.15,-0.15) to [out=-20,in=0] (3,-1.2);
\node[below] at (3,-1.2) {$X'$};
\draw[<-,dashed] (2.85,-0.15) to [out=200,in=180] (3,-2);
\draw[-,dashed] (3.15,-0.15) to [out=-20,in=0] (3,-2);
\node[below] at (3,-2) {$\tY'$};
\draw[thick,fill=white] (0,0) circle [radius=0.15];
\draw[thick,fill=white] (3,0) circle [radius=0.15];
\draw[thick,fill=white] (-0.15,2.85) rectangle ++(0.3,0.3);
\draw[thick,fill=white] (2.85,2.85) rectangle ++(0.3,0.3);
%
\node[left] at (-0.15,0) {$\qn$};
\node[left] at (-0.15,3) {$\qr$};
\node[right] at (3.25,0) {$\qn'$};
\node[right] at (3.15,3) {$\qr$};
\end{tikzpicture},\qquad
\end{equation}
and the superpotentials are
\[
\xW_1 = \xW(X,X') + \xWfr,\qquad \xW_2 = \Tr( X\tY) + \Tr( X'\tY') - \xW(\tY,\tY') + \xWfr,
\]
where
\[
\xW(X,X') = \Tr ( X'\xphi\xpsi) - \Tr ( X \xpsi\xphi) ,\qquad \xWfr = 
\Tr\,\xfph (\qvc'\xphi-\qvc).
\]
Note that the second \Lo\ is the Legendre transform of the first one with respect to the loop arrows $X,X'$.

\subsubsection{NS5 and D5 interfaces between commuting varieties}\label{sec:ns5d5}
The following discussion is parallel to that of~\cite{RimanyiRozansky21}.

The commuting variety $\cM_n = \rmTs\gln\hmr\GLn$ is the instanton moduli space for $\qr=0$, the quiver diagram being a single loop:
\begin{equation}
    \label{eq:cmqv}
\begin{tikzpicture}
\draw[thick,fill=white] (0,0) circle [radius=0.15];
\node[above] at (0,0.15) {$\qn$};
\draw[thick,-] (-0.15,-0.15) to [out=200,in=180] (0,-1.2);
\draw[thick,<-] (0.15,-0.15) to [out=-20,in=0] (0,-1.2);
\node[below] at (0,-1.2) {$X$};
\end{tikzpicture}
\end{equation}
For the family of these quivers, in addition to  NS5 \Lf s we introduce a pair of D5-related \Lf s, borrowing them from bow edges of Nakajima-Cherkis quiver varieties. 
The commuting variety $\cM_n$ is intimately related to the Hamiltonian reduction of a symplectic variety  by the action of $\GL(n)$, and the composition of NS5-related and D5-related \Lf s results in a Nakajima-Cherkis quiver variety, this time related to the \avert\ arrangement of interfaces.

A symplectic variety $\cZsm$ with the Hamiltonian action of $\GL(n)$ and the moment map $\mu\colon\cZsm\rightarrow \gl(n)$ determines an object $(\cZsm;\Tr X\mu)$ in the 2-category $\Bdd(\cM_n)$. Assuming that the action of $\GL(n)$ is free, the category of morphsisms between two such objects is
\[
\Hom\bigl( (\cZsm_1;\Tr X\mu_1), (\cZsm_2;\Tr X \mu_2)\bigr) = \xDCoh\bigl((\cZsm_1\times\cZsm_2)\hmr\GL(n)\bigr).
\]
This relation is due to the fact the category of matrix factorizations of a linear superpotential $W=\Tr X\mu$ is equivalent to the category of coherent sheaves on $\mu=0$.

Any symplectic variety $\cZsm$ with the hamiltonian action of $\GL(n_1)\times\GL(n_2)$ and momenta $\mu_1$, $\mu_2$ determines a \Lf\ 
between $\cM_{n_1}$ and $\cM_{n_2}$:
\begin{gather}
\label{eq:zsmm}
(\cZsm,\xWsm),\quad \xWsm(X_1,X_2) = \Tr X_2 \mu_2 - \Tr X_1 \mu_1,
\\
\begin{tikzpicture}[baseline=-0.75cm]
\draw[thick,fill=white] (0,0) circle [radius=0.15];
\node[above] at (0,0.15) {$\qn_1$};
\draw[thick,-] (-0.15,-0.15) to [out=200,in=180] (0,-1.2);
\draw[thick,<-] (0.15,-0.15) to [out=-20,in=0] (0,-1.2);
\node[below] at (0,-1.2) {$X_1$};
\draw[thick,fill=white] (5,0) circle [radius=0.15];
\node[above] at (5,0.15) {$\qn_2$};
\draw[thick,-] (4.85,-0.15) to [out=200,in=180] (5,-1.2);
\draw[thick,<-] (5.15,-0.15) to [out=-20,in=0] (5,-1.2);
\node[below] at (5,-1.2) {$X_2$};
\draw[thick,fill=white] (2.5,0) circle [radius=0.5];
\node at (2.5,0) {$\cZsm$};
\draw[dashed] (0.15,0) -- (2,0);
\draw[dashed] (3,0) -- (4.85,0);
\node[above] at (1,0) {$\mu_1$};
\node[above] at (4,0) {$\mu_2$};
\end{tikzpicture}
\end{gather}
and its Legendre transform
\begin{gather*}
\bigl(\cZsm\times\rmTs\gln\times\rmTs\gln;\Tr(X_1\tY_1) + \Tr(X_2\tY_2) - W(\cY_1,\cY_2) \bigr)
\\
\begin{tikzpicture}[baseline=-0.75cm]
\draw[thick,fill=white] (0,0) circle [radius=0.15];
\node[above] at (0,0.15) {$\qn_1$};
\draw[thick,-] (-0.15,-0.15) to [out=200,in=180] (0,-1.2);
\draw[thick,<-] (0.15,-0.15) to [out=-20,in=0] (0,-1.2);
\node[below] at (0,-1.2) {$X_1$};
\draw[thick,fill=white] (5,0) circle [radius=0.15];
\node[above] at (5,0.15) {$\qn_2$};
\draw[thick,-] (4.85,-0.15) to [out=200,in=180] (5,-1.2);
\draw[thick,<-] (5.15,-0.15) to [out=-20,in=0] (5,-1.2);
\node[below] at (5,-1.2) {$X_2$};
\draw[thick,fill=white] (2.5,0) circle [radius=0.5];
\node at (2.5,0) {$\cZsm$};
\draw[dashed] (0.15,0) -- (2,0);
\draw[dashed] (3,0) -- (4.85,0);
\node[above] at (1,0) {$\mu_1$};
\node[above] at (4,0) {$\mu_2$};
\draw[<-,dashed] (-0.15,-0.15) to [out=200,in=180] (0,-2);
\draw[-,dashed] (0.15,-0.15) to [out=-20,in=0] (0,-2);
\node [below] at (0,-2) {$\tY_1$};
\draw[<-,dashed] (4.85,-0.15) to [out=200,in=180] (5,-2);
\draw[-,dashed] (5.15,-0.15) to [out=-20,in=0] (5,-2);
\node [below] at (5,-2) {$\tY_2$};
\end{tikzpicture}
\end{gather*}

If the auxiliary variety $\cZsm$ is the arrow edge variety~\eqref{eq:arrvar}: $\cAnot = \rmTs\Hom(\IC^{n_1},\IC^{n_2})$, then the \Lf~\eqref{eq:zsmm} is that of~\eqref{eq:instcr} for $\qr=0$.

A general bow edge variety has a more complicated form: ignoring the stability conditions, it is the critical locus of the pair
\begin{equation*}
\begin{tikzpicture}[baseline=-0.5cm]
\draw[->] (0.15,0) to [out=-15,in=195] (2.85,0);
\draw[<-] (0.15,0) to [out=15,in=165] (2.85,0);
\node[below] at (1.5,-0.2) {$\xphi$};
\node[above] at (1.5, 0.2) {$\xpsi$};
\draw[,-] (-0.15,-0.15) to [out=200,in=180] (0,-1.2);
\draw[,<-] (0.15,-0.15) to [out=-20,in=0] (0,-1.2);
\node[below] at (0,-1.2) {$\tY$};
%
%
\draw[-] (2.85,-0.15) to [out=200,in=180] (3,-1.2);
\draw[<-] (3.15,-0.15) to [out=-20,in=0] (3,-1.2);
\node[below] at (3,-1.2) {$\tY'$};
%
%
\draw[->] (0.15,0.15) to (1.35,1.35);
\draw[->] (1.65,1.35) to (2.85,0.15);
\node[left] at (0.75,0.75) {$\xrho$};
\node[right] at (2.25,0.75) {$\xsgm$};
\node[above] at (1.5,1.65) {$1$};
\draw[very thick,fill=white] (0,0) circle [radius=0.15];
\draw[very thick,fill=white] (3,0) circle [radius=0.15];
\draw[very thick,fill=white] (1.35,1.35) rectangle ++(0.3,0.3);
%
\node[left] at (-0.15,0) {$\qn$};
\node[right] at (3.25,0) {$\qn'$};
\end{tikzpicture}\quad,\qquad
\xWbw(\tY,\tY') = \Tr\,\xpsi(\tY'\xphi - \xphi \tY + \xsgm\xrho).
\end{equation*}
Hence its \Lf\ has the form
\begin{equation*}
\begin{tikzpicture}[baseline=-0.5cm]
\draw[dashed,->] (0.15,0) to [out=-15,in=195] (2.85,0);
\draw[dashed,<-] (0.15,0) to [out=15,in=165] (2.85,0);
\node[below] at (1.5,-0.2) {$\xphi$};
\node[above] at (1.5, 0.2) {$\xpsi$};
\draw[-] (-0.15,-0.15) to [out=200,in=180] (0,-1.2);
\draw[<-] (0.15,-0.15) to [out=-20,in=0] (0,-1.2);
\node[below] at (0,-1.2) {$X$};
\draw[dashed,<-] (-0.15,-0.15) to [out=200,in=180] (0,-2);
\draw[dashed,-] (0.15,-0.15) to [out=-20,in=0] (0,-2);
\node[below] at (0,-2) {$\tY$};
\draw[-] (2.85,-0.15) to [out=200,in=180] (3,-1.2);
\draw[<-] (3.15,-0.15) to [out=-20,in=0] (3,-1.2);
\node[below] at (3,-1.2) {$X'$};
\draw[dashed,<-] (2.85,-0.15) to [out=200,in=180] (3,-2);
\draw[dashed,-] (3.15,-0.15) to [out=-20,in=0] (3,-2);
\node[below] at (3,-2) {$\tY'$};
\draw[dashed,->] (0.15,0.15) to (1.35,1.35);
\draw[dashed,->] (1.65,1.35) to (2.85,0.15);
\node[left] at (0.75,0.75) {$\xrho$};
\node[right] at (2.25,0.75) {$\xsgm$};
\node[above] at (1.5,1.65) {$1$};
\draw[very thick,fill=white] (0,0) circle [radius=0.15];
\draw[very thick,fill=white] (3,0) circle [radius=0.15];
\draw[very thick,fill=white] (1.35,1.35) rectangle ++(0.3,0.3);
%
\node[left] at (-0.15,0) {$\qn$};
\node[right] at (3.25,0) {$\qn'$};
\end{tikzpicture}\quad,\qquad
\xW = \Tr X' \tY' + \Tr X\tY  - \xWbw(\tY,\tY').
\end{equation*}
while is Legendre transform cousin has fewer arrows:
\begin{equation*}
\label{eq:edcrxl}
\begin{tikzpicture}[baseline=-0.5cm]
\draw[dashed,->] (0.15,0) to [out=-15,in=195] (2.85,0);
\draw[dashed,<-] (0.15,0) to [out=15,in=165] (2.85,0);
\node[below] at (1.5,-0.2) {$\xphi$};
\node[above] at (1.5, 0.2) {$\xpsi$};
\draw[-] (-0.15,-0.15) to [out=200,in=180] (0,-1.2);
\draw[<-] (0.15,-0.15) to [out=-20,in=0] (0,-1.2);
\node[below] at (0,-1.2) {$X$};
%
%
\draw[-] (2.85,-0.15) to [out=200,in=180] (3,-1.2);
\draw[<-] (3.15,-0.15) to [out=-20,in=0] (3,-1.2);
\node[below] at (3,-1.2) {$X'$};
%
%
\draw[dashed,->] (0.15,0.15) to (1.35,1.35);
\draw[dashed,->] (1.65,1.35) to (2.85,0.15);
\node[left] at (0.75,0.75) {$\xrho$};
\node[right] at (2.25,0.75) {$\xsgm$};
\node[above] at (1.5,1.65) {$1$};
\draw[very thick,fill=white] (0,0) circle [radius=0.15];
\draw[very thick,fill=white] (3,0) circle [radius=0.15];
\draw[fill=white] (1.35,1.35) rectangle ++(0.3,0.3);
%
\node[left] at (-0.15,0) {$\qn$};
\node[right] at (3.25,0) {$\qn'$};
\end{tikzpicture}\quad,\qquad
\xW = \xWbw(X,X').
\end{equation*}

These two correspondences yield the functors $\xPsiwk$ and $\xPsibk$, \(k=n-n'\) of \eqref{eq:d4sq}
which realize the D5 interfaces \(\mathrm{D5}^{(k)}\) of charge \(k\). We use the super-index  in the notation of
D5 branes to indicate a difference of labels on two sides of the defect:
the difference of labels on two sides of  \(\mathrm{D5}^{(k)}\)  is \(k\).

\subsection{String theory perspective}\label{sec:strings}
\subsubsection{IIB construction}
The Galois version of the \KW\ model with two interface arrangements describes the physics of D3 branes stretched between two \HW\ type NS5 and D5 brane arrangements in IIB string theory.


%
%
%

Consider the IIB string theory on $\MtB=\IRvv{4}{0123}\times\IRvv{4}{4567}\times\IRvv{2}{89}$,  where $\IRvv{k}{i_1\cdots i_k}$ denotes $\IR^k$ with coordinates $x_{i_1},\ldots,x_{i_k}$. The twisting identifies the first two factors, so that the $\SO (4)$ rotation of $\IRvv{4}{0123}$ is accompanied by the matching rotation of $\IRvv{4}{4567}$.
D3-branes are stretched along $\IRvv{4}{0123}$.
Consider an orthogonal splitting $\IRvv{4}{0123} = \IRbrh \oplus \IRo$, where $\IRbrh$ subspace is tangent to an NS5 or a D5 brane, and the corresponding splitting $\IRvv{4}{4567} = \IRbrhp\oplus {\IRo}'$. For \emph{this choice} of the brane directions within the stack of D3-branes, the D3-transverse subspace $\IRvv{6}{456789}$ splits into the sum of the Coulomb subspace $\IRhcl = \IRbrhp$ and the Higgs subspace $\IRhhg = {\IRo}'\times\IRvv{2}{89}$.
Now an NS5-brane must be parallel to $\IRbrh\times\IRhcl$, while a D5-brane must be parallel to $\IRbrh\times\IRhhg$.
%


\subsubsection{IIA construction}
Replace $\IRvv{4}{0123}$ with $M^4$ of \eqref{eq:mfour}. Now the IIB space-time is
\[
\MtB = \IRz\times\IRtot\times \Soh \times \IRvv{4}{4567}\times\IRvv{2}{89}.
\]
The \ahor\ \HW\ brane arrangement corresponds to the choice $\IRbrh = \IRvv{3}{012}$, so its $\nsm$ NS5-branes stretch along $\IRvv{3}{012}\times\IRvv{3}{456}$ and its D5-branes stretch along $\IRvv{3}{012}\times\IRvv{3}{789}$. The \avert\ \HW\ arrangement has the form $\IRbrh =\IRvv{1}{0}\times\IRodr\times \Soh$, where the `directional' subspace $\IRodr\subset\IRvv{2}{12}$ is tangent to the brane interface which is represented by its curve $\xcraI$ in the $\IRvv{2}{12}$ plane.


In sting theory the compactification of $\Soh$ within the KW theory is a result of performing $T$-duality on $\Soh$. This duality turns the IIB theory into IIA theory and
the stack of D3-branes turns into the stack of D2-branes streched along $\IRvv{3}{012}$. Now their movements are described by a 3d TQFT. Since the circle $\Soh$ intersects $\nsm$  NS5-branes of the \ahor\  arrangement,
the $T$-duality turns their orthogonal space $\Soh\times\IRvv{3}{789}$ into the $\nsm$-centered Taub-NUT space $\xmTNcc$, so the IIA space-time is
\[
\MtA = \IRz\times\IRtot\times\IRvv{3}{456}\times\xmTNcc.
\]
D5 branes of the \ahor\ arrangement become D6-branes wrapping $\IRvv{3}{012}\times\xmTNcc$.

\def\dtnv#1{ \dot{#1}}
\def\dth{ \dtnv{3}}
\def\dts{ \dtnv{7}}
\def\dte{ \dtnv{8}}
\def\dtn{ \dtnv{9}}
\def\Rfdt{ \IRvv{4}{\dth\dts\dte\dtn}}
\def\Rths{ \IRvv{2}{\dth\dts}}
\def\Rten{ \IRvv{2}{\dte\dtn}}
After the $T$-duality, the NS5 and D5 branes of the \avert\ HW arrangement turn into NS5 and D4-branes of IIA. They separate the stack of D2-branes into domains. In other words, they become the interfaces of the 3d B-model and hence they are described by Lagrangian correspondences. The form of the Lagrangian correspondece for an additional brane depends on which 2-cycle this brane  wraps inside the Taub-NUT space $\xmTNcc$.
%

\subsubsection{The case of a single NS5 brane}\label{sec:single-brane}
The commuting variety of subsection~\ref{sec:ns5d5} emerges when the \ahor\ arrangement consists of a single NS5 brane. If we ignore the metric, identifying the Taub-NUT space with $\Rfdt$, then after the T-duality the IIA space-time becomes
\[
\MtA = \IRz\times\IRtot\times\IRvv{3}{456}\times\Rfdt\cong\IR^{10}.
\]
Here we use dotted indices to separate the new coordinates from the $T$-dual old ones.
The manifold $\MtA\cong\IR^{10}$ this time contains a single (formerly, \avert) \HW-type arrangment of NS5 and D4 branes, and the Higgs branch of D2 branes stretched between them is the Nakajima quiver variety corresponding to the \emph{\avert} brane arrangement. As explained in subsection~\ref{sec:ns5d5}, this time the zero moment conditions at quiver circles are due to the criticality of the superpotential.

The origin of Legendre pairs of interfaces is now transparent. The Taub-NUT space $\Rfdt$ has two cigar subspaces: $\Rths$ and $\Rten$ that correspond to the arrow $X$ in the commuting variety quiver diagram~\eqref{eq:cmqv} and to its omitted symplectic dual arrow. The NS5 brane of the interface $\xPhiwk$ and the D5 brane of the interface $\xPsibk$ wrap the cigar $\Rths$, while the NS5 brane of $\xPhibk$ and the D5 brane of $\xPsiwk$ wrap the cigar $\Rten$ (the latter branes are often denoted in string theory litearture as NS5' and D5').

\def\Rcob{ \IR_{\mathrm{cob}}}
\def\CYh{ \mathrm{CY}^3 }
\def\HKt{ \mathrm{HK}^2 }
\def\hK{hyper-K\"{a}hler}
\def\xC{ \mathcal{C}}
\def\Lh{ L }
\def\bCSW{CSW}
\def\TN{ \mathrm{TN}}
\def\Mh{ M^3 }
\def\Mve{M^5}
\def\UCYo{ \mathrm{U}(1)_{\mathrm{CY}}}
\def\CDs{ \mathcal{C}_{\mathrm{D6}}}
\def\pMt{ p_{\mathrm{M2}}}
\def\MeM{ M^{11}_{\mathrm{M}}}

\subsection{Link homology from a stack of D2 branes}\label{sec:CS-phys}
\subsubsection{A general setup}
A stack of D3 branes sandwiched between two \HW\ arrangments and a stack of D2 branes split by NS5 and D4 related interfaces into domains carrying quiver-based 3d B-models, appear in theories of link homologies. Generally, link homology is a special example of (categorified) Donaldson-Thomas invariants of a Calabi-Yau 3-fold $\CYh$ in presence of special Lagrangian submanifolds $L\subset \CYh$. Hence we begin with M-theory 
on a manifold of the form
\begin{equation} \label{eq:melv}
\MeM = \CYh\times\HKt\times \Rcob,
\end{equation}
where 
$\HKt$ is a (complex) 2-dimensional \hK\ manifold and $\Rcob$ is the cobordism time. \HW\ brane arrangements come from M5 branes. Each M5 brane has a form $\Lh\times\xC\times\Rcob$, where $\Lh\subset\CYh$ is Lagrangian, while $\xC\subset\HKt$ is a complex curve. 

Assume that $\CYh$ is "small" (that is, all relevant physics happens within a small domain inside $\CYh$). Then link homology is (a part of) the space of states of BPS particles attached to M5-related defects in the effective 5-dimensional theory on $\HKt\times\Rcob$. The BPS particles are M2 branes of the form 
\begin{equation}
    \label{eq:mtshape}
\xSg\times \{\pMt\}\times\Rcob, 
\end{equation}
where $\xSg\subset\CYh$ is a holomorphic curve attached by its boundary components to the Lagrangian submanifolds $\Lh$, while $p\in\HKt$ is a point.

Our approach to the description of the BPS states involves three steps. 

First, suppose that a group $\UCYo$ acts on $\CYh$ preserving the Calabi-Yau form. Taking a quotient by this action reduces M-theory on $\MeM$ to IIA string theory on 
\[
\MtA = \Mve\times\HKt\times\Rcob,\qquad \Mve = \CYh/\UCYo.
\]
The quotient may produce D6 branes of the form $\CDs\times\HKt\times\Rcob$, where $\CDs\subset\Mve$ is the 2-dimensional surface of $\UCYo$-stable points. Depending on whether $L$ is transverse to $\UCYo$ orbits or contains them, each M5 brane becomes an NS5 brane or a D4 brane. If M2 branes are transverse to $\UCYo$ ortibs or invariant with respect to $\UCYo$ action, then they become D2 branes, possibly lying inside D6 branes and with boundaries on NS5 and D4 branes. 

Second, we place M5 branes inside $\CYh$ in such a way that the resulting D2 branes assemble in a stack, whose collective movements and vibrations are described by a $d=3$, $N=4$ SYM theory. The boundaries of individual D2 branes ending on NS5 and D4 branes now transform into interfaces inside the stack of D2 branes, separating domains with (generally) different "stack thickness".

Third, the SYM theory of the D2 stack may be twisted in such a way that a special scalar supercharge $Q$ would allow us to turn  the theory into a (possibly gauged) topological 3d B-model. If the NS5 and D4 interfaces are $Q$-invariant, then the space of BPS particle states becomes isomorphic to the homology of $Q$, that is to the space of of states of this 3d topological B-model.

\def\rmcyl{ \mathrm{cyl}}
\def\ICcyl{\IC^*_{\rmcyl}}
\def\ICcyld{ \IC_{\rmcyl}^\vee}
\def\rmbr{\mathrm{br}}
\def\IRbr{ \IR_{\rmbr}}
\def\IRbrd{ \IRbr^\vee }
\def\IRcyl{ \IR_{\rmcyl}}
\def\CS{Chern-Simons}
\def\ICx{ \IC_x }
\def\ICy{ \IC_y }
\def\cL{ \mathcal{L}}
\def\cnN{ \mathrm{N}^\vee }
\def\pbr{ p_{\mathrm{br}}}
\def\rmU{ \mathrm{U}}
\def\MfCS{\mathrm{M5}_{\mathrm{CS}}}
\def\Mfbr{\mathrm{M5}_{\mathrm{br}}}
\def\sndw{\bullet\!\! -\!\! \bullet}
\subsubsection{Witten's setup}
According to Witten, in order to relate the graded dimension of BPS states to the Chern-Simons-Witten (\bCSW) partition function, one has to make two specific choices in~\eqref{eq:melv}. First, one has to choose $\HKt=\TN$, $\TN$ being the Taub-NUT space, that is $\TN=\IC^2=\ICx\times\ICy$ with a special ALF metric. Second, one has to choose $\CYh=\rmTs\Mh$, where $\Mh$ is a 3d manifold carrying the \bCSW\ theory. A \CS\ generating stack of $N$ M5 branes wraps $\Mh\times\ICy\times\Rcob$. For each link component $\cL\subset\Mh$ one introduces an M5 brane of the form $\cnN\cL\times\ICx\times\Rcob$, where $\cnN\cL\subset\rmTs\Mh$ is the conormal bundle of $\cL$. One has to push these link-related M5 branes off the zero section $\Mh$, then the BPS particles will be M2 branes of the form~\eqref{eq:mtshape}, where $\xSg$ is an annulus, one of its boundary components being attached to the stack of \CS\ M5 branes and the other being attached to the pushed-off link-related M5.

In order to implement our approach and following M.~Aganagic, we choose $\Mh = \ICcyl\times\IRbr$, where $\ICcyl$ is a cylinder and $\IRbr$ is the braid time. 
A link component is locally a braid component of the form $\{\pbr\}\times\IRbr$,  $\pbr\in\ICcyl$. 
Present $\rmTs\IRbr$ as a product: $\rmTs\IRbr = \IRbr\times\IRbrd$. Each M5 brane is represented by a point at the origin of $\IRbrd$, which is the zero section of $\rmTs\IRbr$. We push all M5 branes off the zero section by giving each M5 brane its own position at $\IRbrd$. It is particularly convenient to keep \CS\ M5 branes to the left of the braid M5 branes. M2 branes are now stretched along $\IRbrd$, being sandwiched between these M5 branes. For example, the following diagram represents the case of a $\rmU(2)$ \CS\ theory and a two-strand braid:
\begin{equation}
    \label{eq:mthrddiag}
\begin{tikzpicture}
\draw[very thin,->] (-6,0) -- (6,0);
\node[right] at (6,0) {$\IRbrd$};
\draw[ultra thick,-] (-5,0) -- (-2,0);
\draw[ultra thick,-] (5,0) -- (2,0);
\draw[ultra thick,-] (-2,0.2) -- (2,0.2);
\draw[ultra thick,-] (-2,-0.2) -- (2,-0.2);
\node[below] at (-5,-0.32) {$\MfCS$};
\node[below] at (-2,-0.32) {$\MfCS$};
\node[below] at (5,-0.32) {$\Mfbr$};
\node[below] at (2,-0.32) {$\Mfbr$};
\node[above] at (3.5,0) {M2};
\node[above] at (-3.5,0) {M2};
\node[above] at (0,0.2) {two M2};
\draw[very thick,fill=white] (-2,0) circle [radius=0.3];
\draw[very thick,fill=white] (2,0) circle [radius=0.3];
\draw[very thick,fill=white] (-5,0) circle [radius=0.3];
\draw[very thick,fill=white] (5,0) circle [radius=0.3];
\end{tikzpicture}
\end{equation}

\def\Socyl{ S^1_{\rmcyl}}
Now we implement our approach. 
Present the cylinder $\ICcyl$ as $\ICcyl=\Socyl\times\IRcyl$. 
The group $\UCYo$ rotates $\Socyl$
and M-theory is reduced to IIA theory on \[
\MtA=\IRcyl\times\ICcyld\times\IRbr\times\IRbrd\times\TN\times\Rcob.
\]
\CS\ related M5 branes become D4 branes spanning $\IRcyl\times\IRbr\times\ICx\times\Rcob$, braid related M5 branes become NS5 branes spanning $\ICcyld\times\IRbr\times\ICx\times\Rcob$ and M2 branes which represent the link homology related BPS particles of the 5d theory on $\TN\times\Rcob$ become D2 branes spanning $\sndw\times\IRbr\times\Rcob$, where $\sndw$ denotes the segments stretched along $\IRbrd$. The diagram~\eqref{eq:mthrddiag} now becomes a familiar \HW\ type arrangement:
\begin{equation}
    \label{eq:mthrddiag1}
\begin{tikzpicture}
\draw[very thin,->] (-6,0) -- (6,0);
\node[right] at (6,0) {$\IRbrd$};
\draw[ultra thick,-] (-5,0) -- (-2,0);
\draw[ultra thick,-] (5,0) -- (2,0);
\draw[ultra thick,-] (-2,0.2) -- (2,0.2);
\draw[ultra thick,-] (-2,-0.2) -- (2,-0.2);
\node[below] at (-5,-0.32) {D4};
\node[below] at (-2,-0.32) {D4};
\node[below] at (5,-0.32) {NS5};
\node[below] at (2,-0.32) {NS5};
\node[above] at (3.5,0) {D2};
\node[above] at (-3.5,0) {D2};
\node[above] at (0,0.2) {two D2};
\draw[very thick,fill=white] (-2,0) circle [radius=0.3];
\draw[very thick,fill=white] (2,0) circle [radius=0.3];
\draw[very thick,fill=white] (-5,0) circle [radius=0.3];
\draw[very thick,fill=white] (5,0) circle [radius=0.3];
\end{tikzpicture}
\end{equation}
\def\IRSod{\IR^\vee_{S^1}}
\def\IRcyld{\IRcyl^\vee}
\def\rmU{ \mathrm{U}}
Following the presentation $\ICcyl=\Socyl\times\IRbr$, we decompose \(\ICcyld=\IRSod\times\IRcyld\).
We twist the 3d theory of D2 branes by identifying their worldspace $\IRbrd\times\IRbr\times\Rcob$ with $\IRcyl\times\IRcyld\times\IRSod$. Now a domain of thickness $n$ in the stack of D2 branes carries the 3d B-model whose target is the commuting variety 
$\cM_n = \rmTs\gln\hmr\GLn$ of subsection~\ref{sec:ns5d5}. These domains are separated by NS5 and D4 interfaces corresponding to the functors $\xPhiwk$ and $\xPsiwk$ respectively. 
If the sandwich~\eqref{eq:mthrddiag1} of NS5 and D4 branes is squeezed, then 3d B-models become a single 2d B-model whose target is the Nakajima quiver variety $\cM_Q$ of the corrsponding bow-arrow quiver. The braiding of the braid strands $\{\pbr\}\times\IRbr$ within $\Mh = \ICcyl\times\IRbr$ results in the action of the affine braid group on the category $\xDCoh(\cM_Q)$ introduced by Rina Anno \cite{AnnoNandakumar16}.

This construction can be generalized, if we allow \CS-related and braid-related M5 branes to wrap both types of cigars: $\ICx$ and $\ICy$.
If $N$ M5 branes wrap $\Mh\times\ICy\times\Rcob$ while $M$ M5 branes wrap
$\Mh\times\ICx\times\Rcob$, then $\Mh$ carries the supergroup $\rmU(N|M)$ \WCS\ theory. The $\ICx$-related D4 branes are denoted D4' and they correspond to the interfaces $\xPsibk$. The original braid M5 branes wrapped on $\ICx$ cigar yield $\rmU(M|N)$ representations $\IC^{N|M}$. If a braid-related M5 brane is wrapped on $\ICy$ cigar, then it becomes an NS5' brane in the IIA string theory corresponding to an interface $\xPsiwk$ and producing the representation $\IC^{N|M}$. Note that in the original setup the group $\rmU(N)$ appears as $\rmU(N|0)$ and a braid-related NS5 generates the odd version $\IC^{0|N}$ of the defining representation.
\def\CYhd{ \CYh_{\mathrm{dc}}}
\def\CYhr{ \CYh_{\mathrm{rc}}}
\def\bPo{ \mathbb{P}^1 }
\def\Omoo{ \mathcal{O}(-1)_1}
\def\Omot{ \mathcal{O}(-1)_2}

\subsubsection{Vafa's setup}
The following is a brief account of a part of \cite{DimofteGarberHilburnOblomkovRozansky23}. If $\Mh=S^3$, then the corresponding Calabi-Yau manifold $\CYhd=\rmTs S^3$ is a deformed conifold. Assume that a link in $S^3$ is a closed `tight' braid winding along a `base unknot'. Then, according to Vafa, this M-theory setup with $N$ \CS\ related M5 branes wrapping the zero section $S^3\subset \rmTs S^3$ can be equivalently replaced by a resolved conifold 
\[
\CYhr = \Bigl( \Omoo\otimes \Omot\longrightarrow \bPo\Bigr).
\]
The \CS-related M5 branes disappear, but the braid-related M5 braids remain as Lagrangian submanifolds of $\CYhr$ in such a way that the BPS-generating M2 branes emerge as a tight disk-shaped stack inside the fiber of $\Omoo$  at the "north pole" of $\bPo$.

We select the symmetry $\UCYo$ of $\CYh$ which rotates $\bPo$ around the polar axis (so that $0$ and $\infty$ (that is, the north and south poles of $\bPo$) remain fixed) while rotating the fibers of $\Omoo$ and $\Omot$ in such a way that the fiber of $\Omoo$ at the north pole is fixed. As a result, the braid-related branes M5 become NS5, while a stack of M2 branes becomes a stack of D2 branes inside a D6 brane sweeping $\Omoo|_{0}\times\TN\times\Rcob$. This stack carries the 3d B-model whose target is the Hilbert scheme of points on $\IC^2=\TN$ and it is separated into domains by braid-related NS5 branes.

The fiber of $\Omot$ at the south pole of $\bPo$ (that is, $\Omot|_{\infty}$) is also stable under the action of $\UCYo$, hence the IIA picture includes another D6 brane $\Omot|_{\infty}\times\TN\times\Rcob$. The M2 branes wrapping the zero section $\bPo$ of $\CYhr$ and attached to M2 branes of link homology become strings stretched between the south pole D6 and the center of the stack of D2 branes at north pole, creating the Wilson line carrying the exterior powers of the tautological bundle over the Hilbert scheme target.

\section{Mathematical perspective}

\label{sec:algebraic-model-krs}
This section provides a motivation for the definition of our main
working category. We do not provide details for the constructions in
this section. The main 2-category that we work with appears in section~\ref{sec:equiv-vers-three} and the details of the main
working category are spelled out  in this section.

In~\cite{KapustinSaulinaRozansky09} and~\cite{KapustinRozansky10} a 2-category $\catB(\mathbf{Y})$ was associated to a symplectic variety $\mathbf{Y}$. This construction admits a generalization to the equivariant case, namely, a 2-category $\catB^{\xG}(\mathbf{Y})$ emerges when a group $\xG$ has a hamiltonian action on $\mathbf{Y}$.
We present an three-categorical description of this category.

\subsection{Large 3-category \(\thcs\): a proposal}
\label{sec:large-3-category}
The  objects of $\ob(\thcs)$ are holomophic symplectic manifolds.
The morphisms between
two such manifolds $\bfX,\bfY\in \ob(\thcs)$ form a 2-category $\Hom(\bfX,\bfY)$
whose objects are fibrations with lagrangian bases:
\begin{equation}\label{eq:sp-Lag}
  (F,L,f\colon F\to L),\quad L\subset \bfX\times \bfY \mbox{ is a  Lagrangian subvariety.}
\end{equation}

As hinted in \cite{KapustinSaulinaRozansky09} the fibration of \(F\) over the Lagrangian \(L\) is a categorical analog of a local system over the Lagrangian in
the Fukaya category theory.

The composition of $(F,L,f)\in\mathbb{H}om(\bfX,\bfY)$ and $(G,L',g)\in\HHom(\bfY,\bfW)$ is defined to be $(H,L'',h)$ where
$$ H:=(F\times \bfW)\times_{\bfY} (\bfX\times G),$$
while $h\colon H\to \bfX\times \bfW$ is the natural projection and $L'':=h(H)$.
The composition is not always defined since the \(L''\) is not always a Lagrangian.


The morphisms between the objects $(F,L,f),(F',L',f')\in \mathbb{H}om(\bfX,\bfY)$ is defined as follows:
 $$\Hom((F,L,f),(F',L',f')):=\mathrm{Qcoh}(F\times_{\bfX\times \bfY} F').$$
 The category of quasi-coherent sheaves can be enriched to \((\infty,1)\)-category. The study of the structure of these categories as well as
 \((2,\infty)\)-category of such \((\infty,1)\)-categories is initiated in \cite{GaitsgoryRozenblyum17}. The above mentioned work is extremely technical and we do not attempt to use infinity categories in our work. Instead, we use the previous definition as a guiding principle.


Let us also remark that the category $\thcs$ contains a final object $pt$ which is just a point.
Thus for every $\bfX\in \ob(\thcs)$ there is a related 2-category \[\ddot{\mathrm{Cat}}(\bfX):=\mathbb{H}om(pt,\bfX).\]

\subsection{Small 3-category: first steps}
\label{sec:small-3-category}
In this paper we construct
smaller category $\thcm$, that has fewer objects and fewer morphisms. As a pay-back we get a mathematically satisfactory theory for two and three-morphisms
in such categories, without use of heavy machinery the derived algebraic geometry mentioned above.

The objects of the category $\thcm$ are smooth algebraic varieties.
The 2-category of morphisms $\HHom(\cX,\cY)$ between two objects $\bfX,\bfY\in \ob(\thcm)$ has objects
\[(\cZ,W),\quad \cZ \mbox{ is algebraic manifold}, \quad W\in\CC[ \bfX\times \cZ\times \bfY].\]
For $\bfX,\bfY,\bfW\in \ob(\thcm)$ and $(\cZ,W)\in \HHom(\bfX,\bfY)$, $(\cZ',W')\in \HHom(\bfY,\bfW)$ the composition  is defined by:
$$(\cZ',W')\circ (\cZ,W)=(\cZ\times \bfY\times \cZ',\pi_{xy}^*(W')-\pi_{yz}^*(W))\in \HHom(\bfX,\bfW),$$
where \(\pi_{xy},\pi_{yz}\) are the projections from \(\bfX\times\cZ\times \bfY\times \cZ'\times \bfY\) onto \(\bfX\times\cZ\times \bfY\) and \(\bfY\times \cZ'\times \bfZ\), respectively.

The category of morphisms $\Hom((\cZ,W),(\cZ',W'))$ between the objects $(\cZ,W),(\cZ',W')\in \Hom(\bfX,\bfY)$ is a triangulated one-category
of matrix factorizations
\[\Hom((\cZ,W),(\cZ',W'))=\MF(\bfX\times \cZ\times \cZ'\times \bfY,\pi^{\prime,*}(W')-\pi^*(W)),\]
where \(\pi,\pi'\) are projections onto \(\bfX\times \cZ\times\bfY\) and
\(\bfX\times\cZ'\times \bfY\).

In the discussion above we indicated pull-backs in the formulas for the potentials.
 To lighten up the notations in the discussion below we omit the pull-backs in formulas for the potentials, since the pull-backs are clear from the context.

The objects of the homotopy category \(\MF(\cZ,W)\), \(W\in\CC[\cZ]\) are pairs
\[(M,D),\quad M=M_0\oplus M_1,\quad D\in \Hom_{\CC[\cZ]}(M,M), \quad D^2=W,\]
where \(D\) is \(\ZZ_2\)-graded morphism: \(D(M_i)=M_{i+1}\). If we think of the
matrix factorizations as two periodic curved complexes, then the homotopy is defined in exactly the same way as for usual complexes.

The standard convolution product \(\star\) (see the next section for the details) yields the monoidal structure:
\[\star: \MF(\cZ\times \cZ',W-W')\times
  \MF(\cZ'\times \cZ'',W'-W'')\to \MF( \cZ\times \cZ'',W-W'').\]

For two objects $\calF=(M,D),\;\calG=(M',D')\in \MF(\cZ,W)$ the space of morphisms
is defined by:
\[\mathcal{H}om(\calF,\calG)=\{\phi\in \Hom_{\CC[\cZ]}(M,M')\;|\;D\circ\phi=\phi\circ D'\}.\]

We also define \(\Dper(\cZ)\) to be a derived category of the two-periodic complexes of
\(\CC[\cZ]\)-modules. Given an element \(\calF=(M,D)\in \MF(\cZ,W)\), by inverting the sign of the component of  the differential
from \(M_0\) to \(M_1\) we obtain an element of \(\MF(\cZ,-W)\) which we denote \(\calF^*\). Since the potentials of the matrix factorizations
add if take tensor product, we have:
\[\mathcal{H}om(\calF,\calG)=\mathcal{E}\mathrm{xt}(\calF,\calG)=\calG\otimes \calF^*\in \Dper(\cZ).\]

There is a special class of invertible morphisms  that we want to discuss separately. These morphisms provide an explicit realization of the Kn\"orrer periodicity
functor \cite{Knorrer}.
Namely, if \(V\to\cZ\) is a finite rank vector bundle over \(\cZ\)  and \(V^*\) is its dual then there is a canonical bilinear  function
\[\mathsf{Q}_V\in \CC[V\times V^*]. \]

For any \(W\in \CC[V]\) there is a invertible functor \cite{Knorrer}:
\[\KN_{V}:\MF(\cZ,W)\rightarrow \MF(V\times V^*,W+\mathsf{Q}_V),\quad \KN_{V}(\calF)=\mathcal{KN}_Q\otimes \calF,\]
\[\mathcal{KN}_Q=(\Lambda^\bl V_\theta, D=\sum v_i\frac{\partial}{\partial \theta_i}+v_i^*\theta),\]
where \(\Lambda^\bl V_\theta\) is an exterior algebra with the generators \(\theta_i\), \(v_i^*\) and \(v_i\) are dual coordinates on the spaces
\(V\) and \(V^*\).

That is \(\mathcal{KN}_Q\in \Hom((\cZ,W),(V\times V^*,W+Q_V))\), \((\cZ,W),(V\times V^*,W+Q_V)\in \mathbb{H}om(\bfX,\bfY)\). The functor \(\KN_V\) is an equivalence of categories \cite{Knorrer}. To construct the inverse we need a functor:
\[\KN_{V}^*:\MF(V\times V^*,W+\mathsf{Q}_V)\rightarrow\MF(\cZ,W),\quad \KN_{V}^*(\calF)=\Hom(\calF,\mathcal{KN}_Q).\]
The composition of functors \(\mathrm{KN}_V\) and \(\mathrm{KN}^*_V\) is the duality
functor \(\mathrm{D}\). In general, for a matrix factorization
\(\calF\in \MF(\scZ,W)\) the dual matrix factorization \(\mathrm{D}(\calF)\in
\MF(\scZ,-W)\) is defined by
\[\mathrm{D}(\calF)=\Hom(\calF,\calO_{\scZ}).\]
Since, \(\KN_{V}^*\) and \(\mathrm{D}\) are  functors we can enlarge \(\Hom((V\times V^*,W+Q_V),(\cZ,W))\)
by including \(\KN_V^*\) and \(\mathrm{D}\). Thus the Kn\"orrer functor internally invertible in our category.




\subsection{Relation between the  3-categories}
\label{sec:embedd-three-categ}
There is a functor $j_3\colon\thcm\to \thcs$ which acts on the objects as
$$ \bfX\mapsto \mathrm{T}^*\bfX=\bfX^{\rsm}.$$
This functor explains an introduction of matrix factorizations in our construction, for more discussion of the role of matrix factorizations in context of
3D TQFT see the original \cite{KapustinRozansky10}. 

The embedding at the level of morphism is based on describing lagrangian submanifolds by generating functions \cite{Arnold74}. Let us recall the basic facts.
Given a (complex) manifold $\cZ$ and function $W: \bfX\times \cZ\to \CC$ we define a subvariety
$F_W\subset \mathrm{T}^*\bfX\times \cZ$  by the equations
$$ \partial_{z_i}W(x,p,z)=0,\quad \partial_{x_i}W(x,p,z)=p_i,$$
where $z_i$ are local coordinates along $\cZ$, $x_i$ are local coordinates along $\bfX$ and $p_i$ are
the coordinates on the cotangent space that are dual to the coordinates $x_i$. As shown
in \cite{Arnold74}, the image $L_w$ of $F_w$ in $\mathrm{T}^*\bfX$ under the natural projection
$\pi$  is  a generically  Lagrangian subvariety.

Thus the functor $j_3$ at the level of homomorphisms between the objects is defined as
$$ (\cZ,W)\mapsto (F_W,\pi,L_W).$$

The real problem arises at the  level of morphisms between the morphisms of objects.
It is tempting to
say that we have a functor
$$ \MF(\bfX\times \cZ\times \cZ'\times \bfY,W'-W)\longrightarrow\DGcp(F_w\times_{T^*(\bfX\times \bfY)} F_{w'}).$$
It is not clear to the authors how one could construct such a functor in a canonical way. One option here is
to use the functor \(\mathcal{H}om\), if we fix some element \(\calF\in \MF(\bfX\times \cZ\times \cZ'\times \bfY,W'-W)\); then
we obtain a functor
\[\mathcal{H}om(\calF,\cdot)\colon \MF(\bfX\times \cZ\times \cZ'\times \bfY,W'-W)\rightarrow \DGcp (\bfX\times \cZ\times \cZ'\times \bfY).\]
However, it is not clear whether we can make a choice of \(\calF\) canonical.


\subsection{Equivariant version of the 3-category $\thcm$}
\label{sec:equiv-vers-three}

If $\bfX$ is an affine variety with an action of $\xG$ and $\rmTs\bfX$, then an object in the 2-category $\catB^\xG(\rmTs\bfX)$
is a pair $(\cZ,\xW)$, where the \auv\ $\cZ$ is an affine variety with the action of $\xG$, while $\xW$ is a \sptn: $\xW\in \CC[\bfX\times\cZ]^{\xG}$.

Let us introduce an equivariant version 
of 3-category \(\thcm\) that has as objects pairs:
\[\mathrm{Obj}(\thcm)=\{(\bfX,G)| G \mbox{ is an algebraic group acting on } X\}.\]



Respectively, for \(\bfX\) with action of \(G\) and \(\bfY\) with action of \(H\) the space of morphisms \(\HHom((\bfX,G),(\bfY,H))\) consists of
an the pairs \((\cZ,\xW)\), \(\xW\in \CC[\bfX\times \cZ\times\bfY]^{G\times H} \) and \(\cZ\) has an action of \(G\times H\).

For a pair \((\cZ^1,\xW^1),(\cZ^2,\xW^2)\in \HHom(\bfX,\bfY)\) we define the category of morphisms as category of \(G\times H\)-equivariant matrix
factorizations.
\[
\Hom\bigl( (\cZ^1,\xW^1),(\cZ^2,\xW^2) \bigr) =
\MF_{G\times H}\bigl(\bfX\times\cZ^1\times\cZ^2\times \bfY, \xW^2 - \xW^1\bigr).
\]

In most of the cases we omit the group from the notation, on the other hand if we want to emphasize the presence of the group action we use notation
\((\mathbf{X},G)\) for the object with \(G\)-action. Similarly, sometimes we include the group in the notation of the composition of morphisms, that is
\((\bfX,G),(\bfY,H),(\bfW,K)\in \thcm\) and \((\cZ_{xy},W_{xy})\in \HHom((\bfX,G),(\bfY,H))\), \((\cZ_{yw},W_{yx})\in \HHom((\bfY,H),(\bfW,K))\) then
\begin{equation}\label{eq:2comp}
(\cZ_{xw},W_{xw})=(\cZ_{yw},W_{yw})\circ_H(\cZ_{xy},W_{xy})=(\cZ_{yw}\times \bfY\times \cZ_{xy}/H,W_{yw}-W_{xy})\end{equation}
Furthermore, for
\begin{equation}\label{eq:FF}
\calF_{xy}^{12}\in \Hom((\cZ_{xy}^1,W_{xy}^1),(\cZ_{xy}^2,W_{xy}^2)),\quad
\calF_{yw}^{12}\in \Hom((\cZ_{yw}^1,W^1_{yw}),(\cZ_{yw}^2,W^2_{yw}))
\end{equation}
the vertical composition is defined by
\[\calF^{12}_{xw}=\calF^{12}_{yw}\circ_H\calF^{12}_{xy}:=(\pi_{YW}^*(\calF^{12}_{yw})\otimes \pi_{XY}^*(\calF^{12}_{xy}))^{H^2},\]
where \(\pi_{yw},\pi_{xy}\) are the projections from \((\bfX\times\cZ_{yw}^1\times \cZ_{yw}^2\times \bfY)\bigtimes (\bfY\times\cZ_{xy}^1\times\cZ^2_{xy}\times \bfW)\) to
the first and the last factors, respectively and
\begin{equation}\label{eq:HHom}
(\cZ_{xy}^i,W_{xy}^i)\in \HHom((\bfX,G),(\bfY,H)),\quad (\cZ_{yw}^i,W_{yw}^i)\in\HHom((\bfY,H),(\bfW,K)).
\end{equation}

Similarly, we define the horizontal composition for \(\calF^{12}_{xy}\) as before  and
\begin{equation}\label{eq:F23}
  \calF^{23}_{xy}\in \Hom((\cZ_{xy}^2,W^2_{xy}),(\cZ_{xy}^3,W^3_{xy})),
\end{equation}
here we assume \eqref{eq:HHom}.
In more details, we have
\[\calF^{13}_{xy}=\calF_{xy}^{23}\star\calF^{12}_{xy}=
  \pi_{13*}(\pi_{12}^*(\calF_{xy}^{23})\otimes
  \pi_{23}^*(\calF_{xy}^{12})),\]
where \(\pi_{ij}:\bfX\times \cZ_1\times\cZ_2\times\cZ_3\times \bfY\to
\bfX\times \cZ_i\times \cZ_j\times \bfY\). In particular,
\[\catB^\xG(\rmTs\bfX)=\mathbb{H}\mathrm{om}((X,G),(\pt,\{1\})).\]

We often omit the group super-script when the group is clear from the context. Finally, let us observe that the vertical composition satisfies the exchange rule:

\begin{proposition}\label{prop:exchange}
  For any \((\bfX,G),(\bfY,H),(\bfW,K)\in\thcm\) and \((\cZ_{xy},W^i_{xy})\), \((\cZ_{yw},W^i_{yw})\) as in \eqref{eq:HHom} and
  \(\calF_{xy}^{12},\calF_{yw}^{12}\) as in \eqref{eq:FF}, \(\calF^{23}_{xy}\) as in \eqref{eq:F23} and
  \[\calF^{23}_{yw}\in \Hom((\cZ_{yw}^2,W_{yw}^2),(\cZ_{yw}^3,W_{yw}^3))\]
   there is a three-isomorphism:
   \[ E(\cZ_{xy}^\bullet;\cZ^*_{yw};\calF_{xy}^{\bullet*};\calF_{yw}^{\bullet*}):\quad (\calF^{12}_{yw}\circ\calF^{12}_{xy})\star (\calF_{yw}^{23}\circ\calF_{xy}^{23})\to (\calF_{yw}^{12}\star\calF_{yw}^{23})\circ (\calF_{xy}^{12}\star\calF_{xy}^{23}).\]
   Both sides of the last equation are functors between one-categories and
    the three-isomorphism \(E\) is a natural transformation  of between  these  functors.
\end{proposition}
\begin{proof}
  In the proof below we use the short hand notation \(\cZ^{i_1\dots i_k}_{\bullet\star}=\cZ^{i_1}_{\bullet\star}\times\dots\times \cZ^{i_k}_{\bullet\star}\). First we observe that
 \((\calF^{12}_{yw}\circ\calF^{12}_{xy})\star (\calF_{yw}^{23}\circ\calF_{xy}^{23})\) is computed with the  push-forwards and pull-backs along the maps in the diagram
  and taking \(H^2\)-quotient:
  \[\begin{tikzcd}[column sep=5pt]
      &&\bfX\times\cZ^{123}_{xy}\times\bfY^2\times\cZ^{123}_{yw}\times\bfW\arrow[ddd]\arrow[ddrr]\arrow[ddll]\arrow[drr,dashed]\arrow[dddrr,dashed]
      \arrow[dddll,dashed]\arrow[dll,dashed]&&\\
      \boxed{\bfX\times \cZ^{12}_{xy}\times\bfY}&&&& \boxed{\bfX\times\cZ^{23}_{xy}\times\bfY}\\
      \bfX\times\cZ^{12}_{xy}\times\bfY^2\times\cZ^{12}_{yw}\times\bfW\arrow[u]\arrow[d]&&&&\bfX\times\cZ^{23}_{xy}\times\bfY^2\times\cZ^{23}_{yw}\times \bfW\arrow[u]\arrow[d]\\
      \boxed{\bfY\times\cZ^{12}_{yw}\times\bfW}&&\bfX\times\cZ^{13}_{xy}\times\bfY^2\times\cZ^{13}_{yw}\times\bfW&&\boxed{\bfY\times \cZ^{23}_{yw}\times\bfW}
    \end{tikzcd}\]
  here all the maps are the natural projections. In more details,
  the matrix factorizations \(\calF^{ij}_{\bullet\star}\) are the matrix factorizations on the spaces in the boxes. Thus to obtain \((\calF^{12}_{yw}\circ\calF^{12}_{xy})\star (\calF_{yw}^{23}\circ\calF_{xy}^{23})\) we pull-back and push-forward these matrix  factorizations along the solid arrow until we reach the space in the
  center of the bottom row.

  The dashed arrow maps indicate the natural projections that make the diagram commute. By functoriality of the pull-back we can obtain  \((\calF^{12}_{yw}\circ\calF^{12}_{xy})\star (\calF_{yw}^{23}\circ\calF_{xy}^{23})\) by first pulling back along the dashed arrow and then pushing forward along the central down arrow.

  Similarly, we draw the diagram for the maps that participate in the construction of \( (\calF_{yw}^{12}\star\calF_{yw}^{23})\circ (\calF_{xy}^{12}\star\calF_{xy}^{23})\):
  \[\begin{tikzcd}[column sep=1pt]
           \boxed{\bfX\times \cZ^{12}_{xy}\times\bfY}&&\bigast\arrow[ll,dashed]\arrow[rr,dashed]\arrow[d,dashed]\arrow[rrdd,dashed]\arrow[ddll,dashed]\arrow[dll,dotted]\arrow[drr,dotted] && \boxed{\bfY\times\cZ^{12}_{yw}\times\bfW}\\
      \bfX\times\cZ^{123}_{xy}\times\bfY\arrow[u]\arrow[d]\arrow[dr]&&\bigstar \arrow[dr]\arrow[dl]&&\bfY\times\cZ^{123}_{yw}\times\bfW\arrow[u]\arrow[d]\arrow[dl]\\
      \boxed{\bfY\times\cZ^{23}_{xy}\times\bfW}&\bfX\times\cZ^{13}_{xy}\times\bfY&&\bfY\times\cZ^{13}_{yw}\times\bfW&\boxed{\bfY\times \cZ^{23}_{yw}\times\bfW}
    \end{tikzcd}\]
  here the maps are the natural projections and
  \[\bigstar=\bfX\times\cZ^{13}_{xy}\times\bfY^2\times\cZ^{13}_{yw}\times\bfW.\]

  We obtain \( (\calF_{yw}^{12}\star\calF_{yw}^{23})\circ (\calF_{xy}^{12}\star\calF_{xy}^{23})\) by the sequence of pull-backs and push-forwards that starts from the boxed spaces and ends at \(\bigstar\). By setting
  \[\bigast=\bfX\times\cZ^{123}_{xy}\times\bfY^2\times\cZ^{123}_{yw}\times\bfW\]
  we can complete the diagram to the commuting diagram where the dashed arrows are the natural projections.

  The dashed arrow maps are the same maps as in the previous diagram. Thus to complete our proof we need to show that \( (\calF_{yw}^{12}\star\calF_{yw}^{23})\circ (\calF_{xy}^{12}\star\calF_{xy}^{23})\) can be obtained by the pull-back along the dashed arrows followed with the push-forward along the vertical dashed arrow.

  The last statement follows from a repeated application the base change relation. Indeed,
  first we apply base change to the commuting squares that have as sides  the vertical dashed arrow and  the dotted arrows. Finally, we apply the functoriality of the pull-backs since the composition of the pull-backs along the vertical arrow and the dotted arrows are equal to the pull-back along the corresponding non-vertical dashed arrows.

  Each base change in the construction is facilitated by a particular invertible morphism between the matrix factorizations.
  The composition of these morphisms yields the three-morphism \(E(\cZ_{xy}^\bullet,\cZ^*_{yw};\calF_{xy}^{\bullet*},\calF_{yw}^{\bullet*})\). The base change morphism
  are  natural transformations. Hence \(E\) is a natural transformation.

\end{proof}

In the cases when the first group of arguments of \(E(\dots)\) is clear we abbreviate the notation as
\(E(;\calF_{xy}^{12},\calF_{xy}^{23};\calF_{yw}^{12},\calF_{yw}^{23})\).




Let us also discuss the exchange relation for two-morphisms and three-morphisms. Let us fix the objects \((\bfX,G),(\bfY,K)\) and
let \begin{equation}\label{eq:FGH}
(\cZ^i,W^i)\in \HHom(\bfX,\bfY),\quad \calF^{ii+1},\calG^{ii+1},\calH^{ii+1}\in \Hom((\cZ^i,W^i),(\cZ^{i+1},W^{i+1})).\end{equation}
Respectively, we define \(\calF^{ii+2}=\calF^{ii+1}\star\calF^{i+1,i+2}\), \(\calG^{ii+2}=\calG^{ii+1}\star\calG^{i+1,i+2}.\)

Let us also fix notation for the three-morphisms \begin{equation}\label{eq:uu}
u_{fg}^{ii+1}\in \cHom(\calF^{ii+1},\calG^{ii+1}),\quad u_{gh}^{ii+1}\in \cHom(\calG^{ii+1},\calH^{ii+1}).\end{equation}
We  denote by \(\bullet\) for the composition of morphisms in the category of matrix  factorizations, for example
\(u_{gh}^{12}\bullet u^{12}_{fg}\in \cHom(\calF^{12},\calH^{12})\).

On the other hand  we can define the horizontal composition \(u_{fg}^{23}\star u_{fg}^{12}\in
\cHom(\calF^{13},\calG^{13})\) by
\begin{equation}\label{eq:star-mor} u_{fg}^{23}\star u_{fg}^{12}=\pi_{13*}(\calG^{12}\otimes\pi_{23}^*(u_{fg}^{23})\bullet\pi_{12}^*(u_{fg}^{12})\otimes \calF^{23}),\end{equation}
\[\pi_{12}^*(u_{fg}^{12})\otimes\calF^{23}=\pi_{12}^*(u_{fg}^{12})\otimes \mathrm{Id}_{\pi_{23}^*(\calF^{23})} \in \cHom(\pi_{12}^*(\calF^{12})\otimes \pi_{23}^*(\calF^{23}),\pi_{12}^*(\calG^{12})\otimes\pi_{23}^*(\calF^{23})),\]
\[\calG^{12}\otimes\pi_{23}^*(u_{fg}^{23})=\mathrm{Id}_{\pi^*_{12}(\calG^{12})}\otimes \pi_{23}^*(u_{fg}^{23})
  \in \cHom(\pi_{12}^*(\calG^{12})\otimes \pi_{23}^*(\calF^{23}),\pi_{12}^*(\calG^{12})\otimes\pi_{23}^*(\calG^{23})).\]

The exchange rule in this setting is given by
\begin{proposition}\label{prop:exch-low}
  In notations of \eqref{eq:FGH},\eqref{eq:uu}\eqref{eq:star-mor} we have
  \[(u_{gh}^{23}\star u_{gh}^{12})\bullet(u_{fg}^{23}\star u_{fg}^{12})=
  (u_{gh}^{23}\bullet u_{fg}^{23})\star (u_{gh}^{12}\bullet u_{fg}^{12}).\]
\end{proposition}
 \begin{proof}
   To show that  the morphisms of the matrix factorizations are equal it is enough to check that the morphisms are equal as morphisms of stalks at a point
   \(z\in\bfX\times \cZ^1\times \cZ^3\times \bfY\). The convolution \(\star\) intertwines  \(i^*_{(x,z_1,y)}\times i^*_{(x,z_3,y)}\) with \(i^*_z\) where \(z=(x,z_1,z_3,y)\).
   Thus it is enough to show the statement in the case \(\bfX=\bfY=\cZ_1=\cZ_3=\pt\). In this case the statement of the theorem is equivalent to the functoriality of the
   push-forward.

   Indeed, in this situation \(\calF^{12},\calG^{12},\calH^{12}\in \MF(\cZ^2,W)\), \(\calF^{23},\calG^{23},\calH^{23}\in \MF(\cZ^2,-W)\).  Respectively, the convolutions
   \(\calF^{12}\star\calF^{23}=H^*(\calF^{12}\otimes\calF^{23})\),    \(\calG^{12}\star\calG^{23}=H^*(\calG^{12}\otimes\calG^{23})\),
   \(\calH^{12}\star\calH^{23}=H^*(\calH^{12}\otimes\calH^{23})\)  are the derived push-forwards \(R\pi_*\) to the point, \(R\pi_*:\cZ\to \pt\).

   Since the derived push-forward is a functor, the convolutions of the endomorphism are defined by means of functoriality
   \(u_{gh}^{23}\star u_{gh}^{12}=R\pi_*(u_{gh}^{23}\otimes u_{gh}^{12})\),
   \(u_{fg}^{23}\star u_{fg}^{12}=R\pi_*(u_{fg}^{23}\otimes u_{fg}^{12})\),  \((u_{gh}^{23}\bullet u_{fg}^{23})\star (u_{gh}^{12}\bullet u_{fg}^{112})=R\pi_*((u_{gh}^{23}\bullet u_{fg}^{23})\otimes (u_{gh}^{12}\bullet u_{fg}^{112}))\). Thus the statement of the proposition is equivalent to the relation:
   \[ R\pi_*(u_{gh}^{23}\otimes u_{gh}^{12})\bullet  R\pi_*(u_{fg}^{23}\otimes u_{fg}^{12})=R\pi_*((u_{gh}^{23}\bullet u_{fg}^{23})\otimes (u_{gh}^{12}\bullet u_{fg}^{112}))\]
   which follows from the functoriality of \(R\pi_*\).
   \end{proof}



\subsection{Categories of linear quivers and NS5 interfaces}
\label{sec:categ-line-quiv}

In this paper we consider the subcategories $\catBh(\tAmf)$ (resp. $\catBh(\Amf)$) of $\catBh$ formed by Nakajima quiver varieties of (Dynkin diagram) type $\tAmf$ (resp. $\Amf$) and describe some \gLcs\ connecting them. These correspondences are of two distinct types: NS5 and D4 (or D5). We will describe mostly the NS5 correspondences ~\eqref{eq:ns5sq} for $\nsm=1$.  Their \auvs\ originate from arrows similar to the ones appearing in Nakajima quivers, which could be expected given that both come from NS5 branes in the \sstr\ theory. We will also conjecture the D4 Lagrangian correspondences connecting the so-called \cmvrs, that is, Nakajima varieties of the $\tAz$ quiver without framing~\eqref{eq:d4sq}.



%
%

\subsubsection{NS5 interfaces between $\tAm$ quiver varieties}
The $\tAm$ quiver has the form
\[
\begin{tikzpicture}[baseline=-0.15cm]
\draw[very thick, -] (0,0)-- (2,0);
\draw[very thick,-] (2,0) -- (2.75,0);
\draw[very thick,dashed,-] (-0.75,0) to [out=180, in = 90] (-1.5,-0.75) to [out=-90,in=180] (-0.75,-1.5)
to [out=0,in=180] (2.75,-1.5) to [out=0,in=-90] (3.5,-0.75) to [out=90,in=0] (2.75,0);
\draw[very thick,-] (0,0) -- (0,1.25);
\draw[very thick,-] (2,0) -- (2,1.25);
\draw[very thick,fill=white] (0,0) circle [radius=0.15];
\draw[very thick, -] (-0.75,0) -- (0,0);
\node[below] at (0,-0.15) {$n_i$};
\node[below] at (2,-0.15) {$n_{i+1}$};
\node[above] at (0,1.55) {$r_i$};
\node[above] at (2,1.55) {$r_{i+1}$};
\draw[very thick,fill=white] (0,0) circle [radius=0.15];
\draw[very thick,fill=white] (2,0) circle [radius=0.15];
\draw[very thick,fill=white] (-0.15,1.25) rectangle ++(0.3,0.3);
\draw[very thick,fill=white] (1.85,1.25) rectangle ++(0.3,0.3);
\end{tikzpicture}
\]
where $i = 1,\ldots,\nsm$. It is `colored' by the numbers $\bqn = (\qn_1,\ldots,\qn_{\nsm})$ and $\bqr  = (\qr_1,\ldots\qr_{\nsm})$. We denote this
quiver by \(\tilde{A}_m(\bqn;\bqr)\).
An edge between the $i$-th and $(i+1)$-st circles corresponds to the $i$-th basic, in the sense of section~\ref{sec:phys}, NS5 interface,  the number $\qn_i$ refers to the gauge group $\rU(\qn_i)$ between the $(i-1)$-st and $i$-th NS5-interfaces, and there are $\qr_i$ D5-interfaces there.

\begin{remark}
A `linear' $A_m$ quiver appears as a particular case of the $\tAm$ quiver if one sets $\qn_\nsm=\qr_\nsm = 0$.
\end{remark}

In order to describe additional interfaces as Lagrangian correspondences, we have to present the quiver varieties as cotangent bundles. Hence we
choose the orientation of quiver edges (for example, clockwise and from circles to boxes) and denote the corresponding maps as $X_i$ and $v_i$:
\[
\begin{tikzpicture}[baseline=-0.15cm]
\draw[very thick, ->] (0,0)-- (1.85,0);
\draw[very thick,->] (2,0) -- (2.75,0);
\draw[very thick,dashed,-] (-0.75,0) to [out=180, in = 90] (-1.5,-0.75) to [out=-90,in=180] (-0.75,-1.5)
to [out=0,in=180] (2.75,-1.5) to [out=0,in=-90] (3.5,-0.75) to [out=90,in=0] (2.75,0);
\draw[very thick,->] (0,0) -- (0,1.25);
\draw[very thick,->] (2,0) -- (2,1.25);
\draw[very thick,fill=white] (0,0) circle [radius=0.15];
\draw[very thick, ->] (-0.75,0) -- (-0.15,0);
\node[below] at (0,-0.15) {$n_i$};
\node[below] at (2,-0.15) {$n_{i+1}$};
\node[above] at (0,1.55) {$r_i$};
\node[above] at (2,1.55) {$r_{i+1}$};
\draw[very thick,fill=white] (0,0) circle [radius=0.15];
\draw[very thick,fill=white] (2,0) circle [radius=0.15];
\draw[very thick,fill=white] (-0.15,1.25) rectangle ++(0.3,0.3);
\draw[very thick,fill=white] (1.85,1.25) rectangle ++(0.3,0.3);
\node[above] at (1,0) {$X_i$};
\node[left] at (0,0.625) {$v_i$};
\node[right] at (2,0.625) {$v_{i+1}$};
\end{tikzpicture}
\]

 The quiver variety $\cXsmrn$ is the stable part of the cotangent bundle
\[
\cXsmrn = (\rmTs\cXrn)^{st},\qquad\cXrn = \left(\prod_{i=1}^{\nsm} \Hom(\CC^{\qn_i},\CC^{\qn_{i+1}}) \right)\times
 \left(\prod_{i=1}^{\nsm} \Hom(\CC^{\qn_i},\CC^{\qr_i})\right).
\]
The stable locus is defined by means  the hamiltonian action of the group
\[
\xGbn = \prod_{i=1}^{\nsm} \GLv{\qn_i}.
\]

In more details, given a character of \(\chi =(\chi_1,\dots,\chi_m)\) of \(\xGbn\) the \(\chi\)-stable locus of \(\rmTs\cXrn\) is the projection of the
GIT stable locus of \(\rmTs\cXrn \times \CC^*_{\chi_1}\times\dots\times\CC^*_{\chi_m} \),
see for example \cite[Section 3.1]{Nakajima99}. The \(\chi\)-stable condition can also be
described by explicit linear algebra construction as explained in \cite[Section 3.1]{Nakajima99}, we also use the linear algebra construction in proposition~\ref{prop:comp12} below.
If \(\bqr\) is  zero,  we do not impose any stability constraints.

Consider two affine quivers $Q=\tilde{A}_m(\bqn;\bqr)$ and $Q'=\tilde{A}_m(\bqn';\bqr')$ with the same $\nsm$ but, generally, different dimensions: $(\bqn;\bqr)$ and $(\bqn';\bqr')$. The NS5-interface auxiliary variety $\cZ$ has an arrow description, that is, it is based on the spaces of linear maps. First of all, we introduce the maps $\xphi_i\in\Hom(\CC^{\qn_i},\CC^{\qn'_i})$ and $\yphi_i\in
\Hom(\CC^{\qr_i},\CC^{\qr'_i})$. We \emph{fix} the values of all $\yphi_i$ and require that they all have the highest rank. Then for each `horizontal square'
\begin{equation}
\label{eq:dgsq1}
\begin{tikzpicture}[baseline=1.5cm]
\draw[very thick,->] (0,0.15) -- (0,2.85);
\draw[very thick,->] (3,0.15) -- (3,2.85);
\node[left] at (0,1.5) {$X_i$};
\node[right] at (3,1.5) {$X'_i$};
\draw[->] (0.15,0) -- (2.85,0);
\draw[->] (0.15,3) -- (2.85,3);
\node[below] at (1.5,0) {$\xphi_i$};
\node[above] at (1.5,3) {$\xphi_{i+1}$};
\draw[->] (2.85,2.85) -- (0.15,0.15);
\node[above] at (1.4,1.5) {$\xpsi_i$};
\draw[thick,fill=white] (0,0) circle [radius=0.15];
\draw[thick,fill=white] (3,0) circle [radius=0.15];
\draw[thick,fill=white] (0,3) circle [radius=0.15];
\draw[thick,fill=white] (3,3) circle [radius=0.15];
\node[left] at (-0.15,0) {$\qn_i$};
\node[left] at (-0.15,3) {$\qn_{i+1}$};
\node[right] at (3.15,0) {$\qn'_i$};
\node[right] at (3.15,3) {$\qn'_{i+1}$};
\end{tikzpicture}
\qquad \text{or}\qquad
\begin{tikzpicture}[baseline=1.5cm]
\draw[very thick,->] (0,0.15) -- (0,2.85);
\draw[very thick,->] (3,0.15) -- (3,2.85);
\node[left] at (0,1.5) {$X_i$};
\node[right] at (3,1.5) {$X'_i$};
\draw[->] (0.15,0) -- (2.85,0);
\draw[->] (0.15,3) -- (2.85,3);
\node[below] at (1.5,0) {$\xphi_i$};
\node[above] at (1.5,3) {$\xphi_{i+1}$};
\draw[->]  (2.85,0.15) -- (0.15,2.85);
\node[above] at (1.65,1.5) {$\xpsi_i$};
\draw[->] (0.15,2.85) to  [out=-65,in=65] (0.15,0.15);
\draw[->] (2.85,2.85) to  [out=-115,in=115] (2.85,0.15);
\node at (0.8,1.5) {$\tX_i$};
\node at (2.2,1.5) {$\tX'_i$};
\draw[thick,fill=white] (0,0) circle [radius=0.15];
\draw[thick,fill=white] (3,0) circle [radius=0.15];
\draw[thick,fill=white] (0,3) circle [radius=0.15];
\draw[thick,fill=white] (3,3) circle [radius=0.15];
\node[left] at (-0.15,0) {$\qn_i$};
\node[left] at (-0.15,3) {$\qn_{i+1}$};
\node[right] at (3.15,0) {$\qn'_i$};
\node[right] at (3.15,3) {$\qn'_{i+1}$};
\end{tikzpicture}
\end{equation}
and for each `vertical square'
\begin{equation}
\label{eq:dgsq2}
\begin{tikzpicture}[baseline=1.5cm]
\draw[very thick,->] (0,0.15) -- (0,2.85);
\draw[very thick,->] (3,0.15) -- (3,2.85);
\node[left] at (0,1.5) {$\qvc_i$};
\node[right] at (3,1.5) {$\qvc'_i$};
\draw[->] (0.15,0) -- (2.85,0);
\draw[->] (0.15,3) -- (2.85,3);
\node[below] at (1.5,0) {$\xphi_i$};
\node[above] at (1.5,3) {$\yphi_{i}$};
\draw[->] (2.85,2.85) -- (0.15,0.15);
\node[above] at (1.4,1.5) {$\xfph_i$};
\draw[thick,fill=white] (0,0) circle [radius=0.15];
\draw[thick,fill=white] (3,0) circle [radius=0.15];
\draw[thick,fill=white] (-0.15,2.85) rectangle ++(0.3,0.3);
\draw[thick,fill=white] (2.85,2.85) rectangle ++(0.3,0.3);
\node[left] at (-0.15,0) {$\qn_i$};
\node[left] at (-0.15,3) {$\qr_{i}$};
\node[right] at (3.15,0) {$\qn'_i$};
\node[right] at (3.15,3) {$\qr'_{i}$};
\end{tikzpicture}
\qquad \text{or}\qquad
\begin{tikzpicture}[baseline=1.5cm]
\draw[very thick,->] (0,0.15) -- (0,2.85);
\draw[very thick,->] (3,0.15) -- (3,2.85);
\node[left] at (0,1.5) {$\qvc_i$};
\node[right] at (3,1.5) {$\qvc'_i$};
\draw[->] (0.15,0) -- (2.85,0);
\draw[->] (0.15,3) -- (2.85,3);
\node[below] at (1.5,0) {$\xphi_i$};
\node[above] at (1.5,3) {$\yphi_{i}$};
\draw[->]  (2.85,0.15) -- (0.15,2.85);
\node[above] at (1.65,1.5) {$\xfph_i$};
\draw[->] (0.15,2.85) to  [out=-65,in=65] (0.15,0.15);
\draw[->] (2.85,2.85) to  [out=-115,in=115] (2.85,0.15);
\node at (0.8,1.5) {$\tqvc_i$};
\node at (2.2,1.5) {$\tqvc'_i$};
\draw[thick,fill=white] (0,0) circle [radius=0.15];
\draw[thick,fill=white] (3,0) circle [radius=0.15];
\draw[thick,fill=white] (-0.15,2.85) rectangle ++(0.3,0.3);
\draw[thick,fill=white] (2.85,2.85) rectangle ++(0.3,0.3);
\node[left] at (-0.15,0) {$\qn_i$};
\node[left] at (-0.15,3) {$\qr_{i}$};
\node[right] at (3.15,0) {$\qn'_i$};
\node[right] at (3.15,3) {$\qr'_{i}$};
\node[above] at (-1,2) {$\xfph$};
\end{tikzpicture}
\end{equation}.

Two columns of the formulas~\eqref{eq:dgsq1} and \eqref{eq:dgsq2} are Legendre dual to each other as we explain in section~\ref{sbs:lfs}.
In the same section we explain these two types of interfaces correspond to the interface between a horizontal NS5 brane and vertical NS5 and NS5' brane.
The figure~\eqref{pic:vert-hor-squares} explains how the vertical and horizontal squares a glued together.

In more details, we choose either the first diagonal (on the left) or the second diagonal (on the right) and add the corresponding linear maps $\xpsi_i$, $\xfph_i$ and, in case of the second diagonal choice we also add maps $\tX_i$, $\tX'_i$, $\tqvc_i$ and $\tqvc'_i$. The variety $\cZ$ is formed by all these maps, see the figure~\eqref{pic:vert-hor-squares}, including  $\xphi_i$, while the superpotential $\xW$ is the sum of superpotentials $\xWsqr$ of individual squares. The latter being the 
sums of traces along the boundaries of triangles and digons.  The formula below list these sums for the squares from the formulas~\eqref{eq:dgsq1} and \eqref{eq:dgsq2}:
\begin{equation*}
\xWsqr = \begin{cases}
\Tr ( \xpsi_i X'_i\xphi_i) - \Tr( \xpsi_i \xphi_{i+1} X_i),
\\
\bigl(\Tr (X'_i \tX'_i) - \Tr (\tX'_i\xphi_{i+1}\xpsi_i)\bigr) - \bigl(\Tr(\tX_i X_i) - \Tr(\tX_i\xpsi_i \xphi_i)\bigr)
\\
\Tr( \xfph_i\qvc'_i\xphi_i)-\Tr(\xfph_i\yphi_i\qvc_i )
\\
\bigl(\Tr (\tqvc'_i\qvc'_i) - \Tr(\xfph_i\tqvc'_i\yphi_i)\bigr) -
\bigl(\Tr(\tqvc_i\qvc_i) - \Tr(\xfph_i\xphi_i\qvc_i) \bigr)
\end{cases}
\end{equation*}
We get more interface squares by swapping the quivers $Q$ and $Q'$ in the diagrams~\eqref{eq:dgsq1} and~\eqref{eq:dgsq2}, that is, reversing the directions of horizontal arrows and adjusting the choices of diagonals.

Not all of these choices are compatible with stability conditions on the arrows of the quivers, and some of these choices are Legendre-equivalent if we replace the spaces at the quiver vertices by their dual spaces and, at the same time, reverse the directions of all arrows. For example, the generalized Lagrangian correspondence of the second diagram of~\eqref{eq:dgsq2} can be equivalently presented, if we reverse the directions of vertical arrows:
\[
\begin{tikzpicture}[baseline=1.5cm]
\draw[very thick,<-] (0,0.15) -- (0,2.85);
\draw[very thick,<-] (3,0.15) -- (3,2.85);
\node[left] at (0,1.5) {$\tqvc_i$};
\node[right] at (3,1.5) {$\tqvc'_i$};
\draw[->] (0.15,0) -- (2.85,0);
\draw[->] (0.15,3) -- (2.85,3);
\node[below] at (1.5,0) {$\xphi_i$};
\node[above] at (1.5,3) {$\yphi_{i}$};
\draw[->]  (2.85,0.15) -- (0.15,2.85);
\node[above] at (1.65,1.5) {$\xfph_i$};
\draw[thick,fill=white] (0,0) circle [radius=0.15];
\draw[thick,fill=white] (3,0) circle [radius=0.15];
\draw[thick,fill=white] (-0.15,2.85) rectangle ++(0.3,0.3);
\draw[thick,fill=white] (2.85,2.85) rectangle ++(0.3,0.3);
\node[left] at (-0.15,0) {$\qn_i$};
\node[left] at (-0.15,3) {$\qr_{i}$};
\node[right] at (3.15,0) {$\qn'_i$};
\node[right] at (3.15,3) {$\qr'_{i}$};
\end{tikzpicture}
\]



In our previous work \cite{OblomkovRozansky16},\cite{OblomkovRozansky18},\cite{OblomkovRozansky18a}
and in the section~\ref{sec:defects-knot-invar} we work with the last choice  of
the direction of circle-square edge. In the rest of this section we work with
convention of formulas \eqref{eq:dgsq1}, \eqref{eq:dgsq2} since this choice is closer the standard quiver description of Grassmannians.

\begin{remark}
\label{rm:nakj}
If $\bqr' = \bqr$ while $\bqn\leq \bqn'$, that is, $\qn_i \leq \qn_i'$ for all $i$,  and one requires all horizontal arrow maps to have highest rank, then the first correspondences of~\eqref{eq:dgsq1} and~\eqref{eq:dgsq2} describe the `standard' interface which has a simple symplectic interpretation. Vertical arrow maps describe representations of $Q$ and $Q'$, the horizontal arrows establish the representation of $Q$  as a subrepresentation of $Q'$, and one takes the conormal bundle to this arrangement within the product of two full quiver varieties. These correspondences are
also known under the name {\it Hecke correspondences}. The theory of Hecke correspondences was developed by Nakajima, please consult \cite{Nakajima17a} for a survey and further references. 
\end{remark}

We consider several special cases the special cases of NS5 interfaces between
Grassmanians  in section~\ref{sec:Grass-interf}. More generally,  the NS5 interfaces are discussed in section~\ref{sec:ns5-inst}. Also  we discuss D5 and NS5 interfaces
between the commuting varieties as described in section~\ref{sec:ns5d5}.

\section{Categorical aspects of the theory.}\label{sec:cats}


In this section we discuss the categorical setting for the proposed TQFT. In particular,
we explain how de-looping procedure allows  us to state a precise mathematical statement that underlies a construction of the TQFT.
The delooping procedure conjecturally produces a  monoidal 2-category \(\catB(Q)\) for a  quiver \(Q\).

\subsection{Delooping}
\label{sec:delooping}

In the previous section discussed the general setting for a construction of the 3-category \(\thc\) as well as the simplest equivariant version of the
3-category \(\thc(Q)\) for a linear quiver \(Q\). It is natural to conjecture that \(\thc\) and \(\thc(Q)\) are weak 3-categories \cite{GarnerGursky09}.
However, checking the conditions for the structural morphisms of a weak 3-category is a rather challenging task. We choose a different path for constructing
\(\thc(Q)\). First, we define a delooped monoidal 2-category \(\catB(Q)\), then we
define a large 3-category \(\underline{\thc}(Q)\) as category of module 2-categories over \(\catB(B)\). The 3-category \(\thc(Q)\) is a subcategory
inside \(\underline{\thc}(Q)\).


To construct  2-category \(\catB(Q)\) we first reduce the 3-category \(\thc(Q)\) to a 3-category with one object \(\mathbb{O}_Q=\oplus_{\mathbf{n},\mathbf{r}}Q(\mathbf{n};\mathbf{r})\). Respectively, the category of the endomorphisms \(\mathbb{H}om(\mathbb{O}_Q,\mathbb{O}_Q)\)
of \(\mathbb{O}_Q\) has objects that are collections of elements  \[(\cZ_{\mathbf{n};\mathbf{r}},W_{\mathbf{n};\mathbf{r}})\in
\mathbb{H}om(Q(\mathbf{n};\mathbf{r}),Q(\phi(\mathbf{n});\psi(\mathbf{r}))),\quad \mathbf{n},\mathbf{r}\in \ZZ_+^Q,\quad \phi,\psi: \ZZ_+^Q\to\ZZ_+^Q.\]
The composition law for the morphisms in \(\thc(Q)\) yields a monoidal structure
on \(\mathbb{H}om(\mathbb{O}_Q,\mathbb{O}_Q)\).

Now we declare that the 2-category \(\catB(Q)\) has elements of \(\mathbb{H}om(\mathbb{O}_Q,\mathbb{O}_Q)\) as objects and the rest of the structure is inherited from the
3-category \(\thc(Q)\). Thus we state our main categorical statement:

\begin{theorem}\label{thm:mono-two-cat}
  The delooped 2-category \(\catB(Q)\) has a structure of a  monoidal 2-category (without a unit) in the sense of Kapranov-Voevodsky \cite{KapranovVoevodsky94a}, \cite{BaezNeuchl96}.
\end{theorem}

The essence of the monoidal 2-category is its monoidal structure and the corresponding coherence conditions that are imposed by the compatibility with the monoidal
structure. If we ignore the monoidal structure  we essentially put ourself in the standard setting similar to the one considered in \cite{CaldararuWillerton07}.

\begin{proposition}\label{prop:2cat-no-monoid}
  The vertical composition  \(\star\) and the horizontal composition \(\bullet\) of \(\mathcal{H}om\) functors define a structure of a weak 2-category on \(\catB(Q)\).
\end{proposition}


\begin{proof}
  It is enough to show that \(\HHom(Q(\mathbf{n},\mathbf{r}),Q(\mathbf{n}',\mathbf{r}'))\) is a weak 2-category for a fixed choice
  \(\mathbf{n},\mathbf{r},\mathbf{n}',\mathbf{r}'\).
  For a given \((\scZ,W)\in \HHom(Q(\mathbf{n},\mathbf{r}),Q(\mathbf{n}',\mathbf{r}'))\) there is a canonical choice of the convolution unit:
  \[\mathbf{1}^{\scZ}\in \Hom((\scZ,W),(\scZ,W))=\MF_{G(\mathbf{n})\times G(\mathbf{n}')}(Q(\mathbf{n},\mathbf{r})\times\scZ^2\times Q(\mathbf{n}',\mathbf{r}'),W^{(1)}-W^{(2)}),\]
  where \(W^{(1)}=\pi^*_i(W)\) is the pull-back of the potential \(W\) along two natural projections to \(Q(\mathbf{n},\mathbf{r})\times \scZ\times Q(\mathbf{n}',\mathbf{r}')\). The unit is given by the push-forward of the structure sheaf along the inclusion   \(1\times \Delta\times 1\) of
  \(Q(\mathbf{n},\mathbf{r})\times \scZ\times Q(\mathbf{n}',\mathbf{r}')\) into \(Q(\mathbf{n},\mathbf{r})\times \scZ\times \scZ\times Q(\mathbf{n}',\mathbf{r}')\), here \(\Delta\) is the diagonal inclusion.

  To show that \(\mathbf{1}^{\scZ}\) is the convolution unit we need to use the base change transformation we need to use the base change relations in the
  commuting diagram of maps:
  \[\begin{tikzcd}
    \scZ^3\arrow[d,"\pi_{23}"]\arrow[r,"\Delta\times 1\times 1"]& \scZ^4\arrow[d,"\pi_{124}"]&\arrow[l,"\pi_{12}\times\pi_{23}"]\arrow[d,"\pi_{13}"]\scZ^3\\
    \scZ^2\arrow[r,"\Delta\times 1"]&\scZ^3&\scZ^2
  \end{tikzcd},\]
here \(\pi_{i_1\dots i_s}\) is the projection on the corresponding product of copies of \(\scZ\).
Thus we have \begin{multline*}
  \mathbf{1}^\scZ\star \calF=\pi_{13*}\circ (\pi_{12}\times\pi_{23})^*\circ(\Delta\times 1\times 1)_*\circ\pi_{23}^*(\calF)=\\
  \pi_{13*}\circ (\pi_{12}\times\pi_{23})^*\circ \pi_{124}^*\circ (\Delta\times 1)_*(\calF)=\pi_{13*}\circ (\Delta\times 1)_*(\calF)=\calF.\end{multline*}
Here the second equality follow from the base-change along the commuting square and the last two equalities follow from
\(\pi_{124}\circ\pi_{12}\times \pi_{23}=1\) and \(\pi_{13}\circ (\Delta\times 1)=1\). Since we use the base-change we obtain an invertible natural transformation
\(U_l(\scZ)\) between the functor of the left multiplication by \(\mathbf{1}^{\scZ}\) and the identity functor. Similarly, one defines the natural transformation
\(U_r(\scZ)\) for the right multiplication.

  The first group of axioms of the weak 2-category that is concerned with the interaction of the vertical composition \(\star\) and horizontal composition \(\circ\) is
  exactly the interchange relation from proposition~\ref{prop:exchange}.

  Our 2-category is weak because in the standard proof of the associativity
  of the convolution product \(\star\) one has to use the base change relation. Thus the base change natural transformations yield the associator natural
  transformation:
  \[A(\calF,\calG,\calH): \left(\calF\star\calG\right)\star\calH\to \calF\star\left(\calG\star\calH\right). \]

  The associator in a weak two category satisfies the pentagon axiom and the unitors \(U_l,U_r\) satisfy the triangle axioms see for example
  \cite{EtingofGelakiNikshychOstrik15}. The construction of the associator and the unitors is in terms of base-change relations. Thus we can use the standard
  agrument for convolution of coherent sheaves to show the axioms in our setting.
\end{proof}

We refer to the original paper \cite{KapranovVoevodsky94a} and subsequent paper \cite{BaezNeuchl96} for a streamlined treatment of the (semi-strict) monoidal 2-category.
In our case,  the composition operation \(\circ\) is strictly associative since the equivariant pull-backs in the definition of the composition in \eqref{eq:2comp}.

Thus the only nontrivial structural two-morphism \(R=R(\calF,\calG)\) of our  2-category is the two-isomorphism
\(\circ_{\calF,\calG}\) that relates the morphisms of pairs of  objects \(\calF\in\Hom(\cZ,\cZ')\), \(\calG\in\Hom(\cX,\cX')\) as in the diagram below:
\[\begin{tikzcd}
  \cZ\circ\cX\arrow[r,"1_{\cZ}\circ \calG",""{name=U,below}]\arrow[d,"\calF\circ 1_{\cX}"',""{name=D}]&\arrow[d,"\calF\circ 1_{\cX'}"]\cZ\circ\cX'\\
  \cZ'\circ\cX\arrow[r,"1_{\cZ'}\circ\calG"']&\cZ'\circ\cX'
  \arrow[Rightarrow,from=U,to=D,"R"]
\end{tikzcd}
\]
We define  \(R(\calF,\calG)\) as  composition of exchange morphisms from
proposition~\ref{prop:exchange}:
\[R(\calF,\calG)=E(\cZ,\cZ,\cZ';\cX,\cX',\cX';\mathbf{1}^{\cZ}_{xy},\calF;\calG,\mathbf{1}^{\cX'}_{yz})\circ E^{-1}(\cZ,\cZ',\cZ';\cX,\cX,\cX';\calF,\mathbf{1}^{\cZ'}_{xy};\mathbf{1}^{\cX}_{yz},\calG),\]
where \(\mathbf{1}^{\mathscr{Y}}_{\star,*}\) stands for the identity element in \(\Hom(\mathscr{Y},\mathscr{Y})\).

\begin{proof}[Proof of theorem~\ref{thm:mono-two-cat}]
  The 2-category \(\catB(Q)\) is almost semi-strict. Indeed, the convolution \(\circ\) is strictly associative on the objects. In more details,
  in notations \cite{KapranovVoevodsky94a} \(\circ\) corresponds to the composition \(\otimes\). The structural one-morphisms that appear in the
  definition of the monoidal of \cite{KapranovVoevodsky94a} are  \(a_{A,B,C}: A\otimes (B\otimes C)\to (A\otimes B)\otimes C\), in notations of
  of \cite{KapranovVoevodsky94a}. The analogous one-morphisms in our setting are the identity morphisms. Respectively, the related analogs of two-morphisms
  \(a_{A,B,C,D},a_{u,A,B},a_{A,u,B},a_{A,B,u}\) are also identity in our setting.

  The analogs of the structural two-morphisms \(\otimes_{u,u',B}\), \(\otimes_{A,v,v'}\), \(\otimes_{u,v}\) in our setting are expressed in terms
  of the exchange natural-transformation \(E\) from proposition~\ref{prop:exchange}:
  \[\begin{aligned}
    &\otimes_{u,u',B}& &R(\calF,\calG;\scX)=E(\scZ,\scZ',\scZ'';\scX,\scX;\calF,\calG;\mathbf{1}^{\scX}_{xy},\mathbf{1}^{\scX}_{yw})\\
    &\otimes_{A,v,v'}& &R(\scZ;\calF,\calG)=E(\scZ,\scZ,\scZ;\scX,\scX',\scX''; \mathbf{1}^{\scZ}_{xy},\mathbf{1}^{\scZ}_{yw};\calF,\calG)\\
    &\otimes_{u,v} & & R(\calF,\calG).
  \end{aligned}\]
where we set \(u=\calF,u'=\calG,B=\scX\) for the first relation,
\(A=\calZ,v=\calF,v'=\calG\) for the second relation and \(u=\calF,v=\calG\)
for the last relation. We do not repeat definitions of the Kapranov-Voevodsky tensor products
from the left column of the last three formulas, see the original text for the definition. Our formulas give a definition of
the operators \(R(\calF,\calG;\scX),R(\scZ;\calF,\calG),R(\calF,\calG)\) as well as
reference to the relevant operation from \cite{KapranovVoevodsky94a}.
We also translate Kapranov-Voevodsky axioms to our language in our argument below.

The rest  non-trivial two-morphisms in our setting are structural morphisms of the 2-category from proposition \ref{prop:2cat-no-monoid}. One can apply the  method of MacLane strictness theorem to turn the 2-categories of 1-morphisms into the strict 2-categories, thus we suppress associator and unitor two-morphism in the discussion below.

The axioms that need to be checked for the structural two-morphisms \(R(\bullet,\bullet;\bullet)\),\(R(\bullet,\bullet)\), are on the pages 221 and
222 of \cite{KapranovVoevodsky94a}.
Let us first explain how  one can prove the relations on the page 222 of \cite{KapranovVoevodsky94a} which are also the essential
relations of the semi-strict monoidal 2-category (see Lemma 4 in \cite{BaezNeuchl96}). The rest of the axioms can be check by a similar computation.

The relation from the top of the page 221 of \cite{KapranovVoevodsky94a} as well as relation  (viii) in Lemma 4 of \cite{BaezNeuchl96} can be read of the commuting diagram:
\[\begin{tikzcd}
    \scX\circ \scZ\arrow[rr,dotted,"(\calF\star\calF')\circ \mathbf{1}^{\scZ}"]\arrow[dd,dotted,"\mathbf{1}^{\scX}\circ\calG"']\arrow[rd,"\calF\circ\mathbf{1}^{\scZ}"]&&\scX''\circ\scZ\arrow[dd,dotted,"\mathbf{1}^{\scX''}\circ\calG"]\arrow[ld,"\calF'\circ\mathbf{1}^{\scZ}"]\\ & \scX'\circ\scZ\arrow[dd,near start,"\mathbf{1}^{\scX'}\circ\calG"']& \\  \scX\circ\scZ'\arrow[rd,dotted,"\calF\circ\mathbf{1}^{\scZ'}"'] \arrow[rr,dashed, near end,"(\calF\star\calF')\circ\mathbf{1}^{\scZ'}"]&&\scX''\circ\scZ'\\
    &\scX'\circ\scZ'\arrow[ru,dotted,"\calF'\circ \mathbf{1}^{\scZ'}"']&
  \end{tikzcd}
\]
here \(\calF\in \Hom(\scX,\scX'),\calF'\in \Hom(\scX',\scX''),\calG\in \Hom(\scZ,\scZ')\) and \(\scZ,\scZ'\in\HHom(\bfY,\mathbf{Z})\)
\(\scX,\scX',\scX''\in \HHom(\bfX,\bfY)\), we omit spaces \(\bfX,\bfY,\mathbf{Z}\) from our notations.

There are two dotted paths that connect \(\scX\circ\scZ\) to \(\scX''\circ\scZ'\). The relation that  we are after is the equality of two compositions of two-morphisms that relate these two paths. In more details, the first sequences of two-morphisms is:
\[(\scX\circ\scZ\to \scX\circ\scZ'\to\scX'\circ\scZ'\to\scX'\circ\scZ'')\xRightarrow{R(\calF,\calF';\calZ')}(\scX\circ\scZ\to\scX\circ\scZ'\to\scX''\circ\scZ'),\]
\[(\scX\circ\scZ\to\scX\circ\scZ'\to\scX''\circ\scZ')\xRightarrow{R(\calF\star\calF',\calG)}(\scX\circ\scZ\to\scX''\circ\scZ'\to\scX''\circ\scZ'). \]
The second sequence of two-morphisms is
\[(\scX\circ\scZ\to \scX\circ\scZ'\to\scX'\circ\scZ'\to\scX'\circ\scZ'')\xRightarrow{R(\calF,\calG)}(\scX\circ\scZ\to\scX'\circ\scZ\to\scX'\circ\scZ'\to\scX''\circ\scZ'),\]
\[(\scX\circ\scZ\to\scX'\circ\scZ\to\scX'\circ\scZ'\to\scX''\circ\scZ')\xRightarrow{R(\calG,\calF')}(\scX\circ\scZ\to\scX'\circ\scZ\to\scX''\circ\scZ'\to\scX''\circ\scZ'),\]
\[(\scX\circ\scZ\to\scX'\circ\scZ\to\scX''\circ\scZ\to\scX''\circ\scZ')\xRightarrow{R(\calF,\calF';\scZ)} (\scX\circ\scZ\to\scX''\circ\scZ'\to\scX''\circ\scZ').\]

The two morphisms in the sequences are compositions of the base-change two-morphisms of some geometric maps.  Let us discuss  sequences of two-morphism in details. There is  a sequence of maps \(g_{i+1/2;k}\), \(k=1,\dots,r\) between spaces \(S_{i;k}\), \(g_{i+1/2;k}:S_{i+(-\epsilon_{i;k}+1)/2}\to S_{i+(\epsilon_{i;k}+1)/2} \),
\(\epsilon_{i;k}=\pm 1\) and the functors
\(F_k=f_1\circ\dots\circ f_{m_k}\) where \(f_i=g_{i+1/2;k*}\) if \(\epsilon_{i;k}=1\) and \(f_i=g_{i+1/2;k}^*\) if \(\epsilon_{i;k}=-1\) such that
\begin{enumerate}
\item \(F_1: (\calF,\calF',\calG)\mapsto ((\mathbf{1}^\scX\circ\calG)\star(\calF\circ\mathbf{1}^{\scZ'}))\star \calF'\circ\mathbf{1}^{\scZ'}\),
\item There is \(s>1\) such that \(F_s: (\calF,\calF',\calG)\mapsto ((\mathbf{1}^\scX\circ\calG)\star(\calF\star\calF'\circ\mathbf{1}^{\scZ'}))\),
\item \(F_r: (\calF,\calF',\calG)\mapsto ((\calF\star\calF'\circ\mathbf{1}^{\scZ'})\star (\mathbf{1}^\scX\circ\calG))\),
\item \(F_{k+1}\) is obtained from \(F_k\) either by the base-change morphism or composition of some the push-forwards or or pull-backs in the
  sequence \(f_{i;k}\).
\item the corresponding sequence of natural transformations connecting \(F_1\)  with \(F_s\) is \(R(\calF,\calF';\scZ)\) and the sequence connecting
  \(F_s\) and \(F_r\) is \(R(\calF\star\calF',\calG)\).
\end{enumerate}

Similarly, there is a sequence of maps \(h_{i+1/2;k}\), \(\epsilon'_{i;k}=\pm 1\) and spaces \(S'_{i;k}\)  \(k=1,\dots,l\) that realizes the sequence of
the two-morphisms \(R(\calF,\calG)\), \(R(\calG,\calF')\), \(R(\calF,\calF';\scZ)\). Moreover,  \(h_{i+1/2;k}=g_{i+1/2;k'}\), \(\epsilon'_{i;k}=\epsilon_{i;k'}\)
for \(k=k'=1\) and \(k=r,k'=l\).

Thus we can paste the sequence of the maps \(h_{i+1/2;k}\) and the sequence \(g_{i+1/2;k}\) along the sequences
\(g_{i+1/2;1}\) and \(g_{i+1/2;r}\). One half of the pasted diagram yields the compositions of the natural transformations \(R(\calF,\calF';\scZ)\) and
\(R(\calF\star\calF',\calG)\) and the other half the composition of  \(R(\calF,\calG)\), \(R(\calG,\calF')\), \(R(\calF,\calF';\scZ)\). The commutativity
of the diagram of maps implies the equality of two compositions of the natural transformations.

The next group of the relations on the page 221 of \cite{KapranovVoevodsky94a} are equivalent to the naturality of \(R(\calF,\calG)\), in our case the
naturality is automatic since \(R(\calF,\calG)\) is a composition of natural transformations of base change. The last group of relation  relation on that page
follows immediately from the exchange relation from proposition~\ref{prop:exch-low}.

The first group of relations on the page 222 of \cite{KapranovVoevodsky94a} is equivalent to naturality of \(R(\calF,\calG;\scZ)\). Finally, the last relation
on this page is based on the commuting diagram:
\[\begin{tikzcd}
    &\scX'''\circ\scZ&\\
    \scX\circ\scZ\arrow[ru,dotted,"\calF\star\calF'\star\calF''\circ \mathbf{1}"]\arrow[rr,dashed,near end,"\calF\star\calF'\circ\mathbf{1}"]\arrow[rd,dotted,"\calF\circ\mathbf{1}"']&&\scX''\circ\scZ\arrow[lu,dotted,"\calF''\circ\mathbf{1}"']\\
    &\scX'\circ\scZ\arrow[uu,near start,"\calF'\star\calF''\circ\mathbf{1}"]\arrow[ur,dotted,"\calF'\circ 1"']&
  \end{tikzcd}.
\]

There are two compositions of the dotted arrow that start at \(\scX\circ\scZ\)
and end at \(\scX'''\circ\scZ\). Respectively, there are two ways to connect these compositions by the two-morphisms which suppose to result into the same
two-morphism:
\[R(\calF,\calF'\star\calF'';\scZ)\circ R(\calF',\calF'',\scZ)=
  R(\calF\star\calF',\calF'';\scZ)\circ R(\calF,\calF';\scZ).\]
The last equality follows from the pasting argument as in the case discussed at the beginning of the proof.
\end{proof}


\subsection{Module category construction of \(\thc(Q)\)}\label{sec:modular-cat-def}

Let us define \(\underline{\thc}(Q)\) to be a 3-category of 2-categories with the two-monoidal action of \(\catB(Q)\). This 3-category is too big and
the 3-category \(\thc(Q)\) is a sub 3-category inside \(\underline{\thc}(Q)\). Indeed, the product of 2-categories
\[\mathbb{M}_{\bfn,\mathbf{r}}=\prod_{\bfn',\mathbf{r}'}\HHom(Q(\bfn,\mathbf{r}),Q(\bfn',\mathbf{r}'))\]
is a 2-category with the monoidal action of \(\catB(Q)\), that is an object in \(\underline{\thc}(Q)\).

Respectively, the 2-category \(\HHom(Q(\bfn,\mathbf{r}),Q(\mathbf{m},\mathbf{s}))\) is a sub 2-category
of 2-category of morphisms from \(\mathbb{M}_{\mathbf{n},\mathbf{r}}\) to \(\mathbb{M}_{\mathbf{m},\mathbf{s}}\).
Thus we define the 3-category \(\thc(Q)\) to be the  3-category 
with objects
\(\mathbb{M}_{\bfn,\mathbf{r}}\) and the morphisms between these objects as above.

\subsection{The case \(Q=\tilde{A}_0\)}\label{sec:monoidal-cat}

In the case of general quiver \(Q\) it is not immediately clear how to construct a unit in the monoidal 2-category \(\catB(Q)\). However, in the case
\(Q=\tilde{A}_0\) there is a natural candidate for the unit.
\[\mathbb{L}_{\tid}=\{\mathbb{L}_{\tid}^{n,r}\}_{n,r\ge 0},\quad\mathbb{L}^{n,r}_{\tid}\in \HHom(\tilde{A}_0(n;r),\tilde{A}_0(n;r)),\]
the one-morphism \(\mathbb{L}^{n;r}_{\tid}\) is defined in the next section where we provide a proof for
\begin{theorem}\label{thm:unit}
  The 2-category \(\catB(Q)\), \(Q=\tilde{A}_0\) is a monoidal 2-category with a unit \(\mathbb{L}_{\tid}\).
\end{theorem}

In the notations of the previous section
the 3-category \(\thc_{\tgl}\) is a subcategory of \(\thc(\tilde{A}_0)\) with
objects \(\tilde{A}_0(n,0)\), \(n\ge 0\). We also  define the framed version of the 3-category \(\thc_{\tgl}^{f}\) in the second half of this section.
The later category is a sub-category of \(\thc(\tilde{A}_0)\) with objects
\(\tilde{A}_0(n,1)\).
 Implicitly, this category was studied in \cite{OblomkovRozansky16}, \cite{OblomkovRozansky17},
 \cite{OblomkovRozansky17a}  where several results in theory of knot homology were derived.
 We spell out the details of the categories
 \(\thc_{\tgl}\) and \(\thc_{\tgl}^{f}\) and state some conjectures about these categories.

\subsubsection{Objects and morphisms}
\label{sec:objects-morphisms}

The objects of $\thc_{\tgl}$ are labeled by $\ZZ_{\ge 0}$:
\[\mathrm{Obj}\,\thc_{\tgl}=\{\mathbf{n}|\mathbf{n}\in \ZZ_{\ge 0}\}.\]
The objects in the 2-category of morphisms are pairs $(Z,w)$, where $Z$ is an algebraic variety with an action of $\GL_n\times\GL_m$:
\[\tObj\,\HHom(\mathbf{n},\mathbf{m})= \{(Z,w),\quad w\in \CC[\gl_n\times Z \times \gl_m]^{\GL_n\times\GL_m}\}.\]


The composition of morphism $(Z,w)\in\HHom(\mathbf{n},\mathbf{m})$, $(Z',w')\in \HHom(\mathbf{m},\mathbf{k})$ is defined as
$$ (Z,w)\circ(Z',w') = (Z\times \gl_m\times Z'/_+\GL_m,w'+w)\in \HHom(\mathbf{n},\mathbf{k}).$$
Here the quotient is defined via GIT theory as follows. Suppose that $X$ is a variety with a $\GL_m$-action.
A character $\chi$ of $\GL_m$ determines the trivial line bundle  $L_\chi$
with $\GL_m$-equivariant  structure  defined by $\chi$.
Recall that a point $x\in X$ is semistable (with respect to $L_\chi$) if there is
$m>0$ and $s\in \Gamma(X,L_\chi)^{GL_m}$ such that $s(x)\ne 0$. Denote
$$ X/_\chi \GL_m:=X^{ss}/\GL_m.$$
Since the group of characters of $\GL_m$ is generated by $\det$, we introduce short-hand notations:
$$ X/_{\pm} \GL_m:=X/_{\det^{\pm 1}} \GL_m,\quad X/_0\GL_m:=X/_{\det^0}\GL_m.$$

Two-morphism are  the objects of  the corresponding category of equivariant matrix factorizations.
Given \((Z,w),(Z',w')\in  \HHom(\mathbf{n},\mathbf{m})\) we define:
\[\Hom\bigl((Z,w),(Z',w')\bigr)=\MF_{\GL_n\times\GL_m}(\gl_n\times Z\times Z'\times \gl_m,w-w').\]
The group \(\GL_n\times \GL_m\) is  reductive, hence the equivariance of  the matrix factorization  \((M,D)\) is
equivalent to the condition that the group action on \(M\) commutes with the differential \(D\).

The space of morphisms \(\mathcal{H}om(\cdot,\cdot)\)  between the equivariant matrix factorizations is defined
to be the space of morphisms between the underlying matrix factorizations that commute the group action.

\subsubsection{Framed version of the 3-category}
\label{sec:framed-version-three}

We enlarge slightly our category to include the framing. The objects of the new 3-category
\(\thc_{\gl}^{\tfrm}\) are again labeled by the positive integers:
\[\mathrm{Obj}(\thc_{\gl}^{\tfrm})=\{\mathbf{n}^{\tfrm}\,|\,n\in\ZZ_{\ge 0}\}.\]
For the space of morphisms we have  \[\HHom(\mathbf{n}^{\tfrm},\mathbf{m}^{\tfrm})=(Z,w),\quad w\in \CC[V_n\times\gl_n\times Z\times\gl_m\times V_m]^{\GL_n\times\GL_m}\] here $Z$ is an algebraic variety  with an
action of $\GL_n\times \GL_m$ and \(V_n=\CC^n,V_m=\CC^m\) with the standard \(\GL_n\) and \(\GL_m\) actions.

The rest of definitions are identical to the constructions from the previous subsection, after we replace
\(\gl_n\) with \(V_n\times \gl_n\). For brevity we introduce the following shorthand notation:
\[\gl_n^{\tfrm}:=\gl_n\times V_n.\]
In general, many definitions in our paper are parallel in framed and unframed cases, in the cases when the
definitions are parallel we use \(\bullet\) notation to indicate that \(\bullet \) could be "f" or empty set.



First of all, note that the category \(\HHom(\mathbf{n}^\bl,\mathbf{m}^\bl)\) has a natural monoidal structure.
Given \(\cf\in \Hom((Z,w),(Z',w'))\) and \(\cg\in \Hom((Z',w'),(Z'',w''))\) we define
\(\cf\star\cg\in \Hom((Z,w),(Z'',w''))\)
\[\cf\star\cg:=\pi_{*}(\cf\otimes\cg),\quad \pi\colon \gl_n\ti Z\ti Z'\ti Z''\ti\gl_m\to \gl_n\ti Z\ti  Z''\ti\gl_m \]

\begin{remark}
  The 3-categories \(\thc_{\gl},\thc_{\gl}^f\) can be realized in terms of 4-category \(\BGL\) by
  assigning to \(\mathbf{n}\) the morphisms
  \[\gl_n\in \mathbf{H}\mathrm{om}(\GL(n),\GL(0)),\quad \gl_n\times V_n\in \mathbf{H}\mathrm{om}(\GL(n),\GL(0)),\]
  respectively.
  That is \(\thc_{\gl}=\thc(\tilde{A}_0(0))\),   \(\thc_{\gl}^f=\thc(\tilde{A}_0(1))\) in
  notations of section~\ref{sec:compact}.
\end{remark}

\subsection{Monoidal properties of \(\thc_{\gl}^\bullet\)}

The objects of \(\thc_{\gl}^\bullet\) are labeled by the positive numbers and it seems natural to use the addition for the monoidal structure on the level of objects.
It is much harder to extend this monoidal structure to \(1\), \(2\), \(3\)-morphisms. Nevertheless, we conjecture that it is possible to construct such extension. Below, we provide detailed motivation for the

\begin{conjecture}
  The 3-category \(\thc_{\gl}^\bullet\) can be enhanced to a monoidal 3-category.
\end{conjecture}

Let us introduce the auxiliary object that facilitates the extension of monoidal structure to level of \(1\)-morphisms.
 The object is the pair
\[\mathbb{I}_{n,k;n+k}=(T^*\GL_{n+k}, W_{n,k;n+k}),\]
\[  W_{n,k;n+k}(X,z,X')=\Tr(X \mu_{\gl_{n+k}}(z))-\Tr(X'\mu_{\mathfrak{p}_{n,k}})\in \CC[\gl_{n+k}\times T^*\GL_{n+k}\times \mathfrak{p}_{n,k}]^{\GL_{n+k}\times P_{n,k}},\]
here \(P_{n,k}\subset \GL_{n+k}\) is the standard parabolic subgroup with the Levi group
\(\GL_n\times \GL_k\), respectively \(\mathfrak{p}_{n,k}=\mathrm{Lie}(P_{n,k})\) and
\(\mu_{\gl_{n+k}}\) and \(\mu_{\mathfrak{p}_{n,k}}\) are the moment maps for the left and right actions of the corresponding groups.

  Given two objects \(\mathbb{O}_1\in \HHom(\bfm^\bullet_1,\bfn^\bullet)\), \(\mathbb{O}_2\in \HHom(\bfm^\bullet_2,\bfk^\bullet)\) we define the induced object
  \[\mathbb{O}_1\oplus\mathbb{O}_2:=\mathbb{O}_1\times\mathbb{O}_2\circ_{P_{n,k}}\mathbb{I}_{n,k;n+k}\in \HHom(\bfm^\bullet_1\times\bfm^\bullet_2,(\bfn+\bfk)^\bullet).\]

  Thus defined operation extends to the level of two-morphisms. Moreover, the exchange relation from proposition~\ref{prop:exchange} implies that this operation is monoidal with
  respect to the monoidal structure \(\star\). In the case \(m_1=m_2\) the above operation was studied in \cite{OblomkovRozansky16} where it is called induction
  functor and it was used to construct the braid realization discussed in the next section.

  We expect that the object \(\mathbb{I}_{n,k;n+k}\) will play an important role in the foam extension of our TQFT, this element is the element one needs for
  defining a value of the partition function on the configurations that contain three planes of defect intersecting along the line of defect. Also this element
  interacts well with \(\mathrm{NS5}\) defects discussed in the next section:
\begin{equation*}
\mathbb{I}_{l|n}\times \mathbb{O}\circ_{P_{n,k}}  \mathbb{I}_{n,k;n+k}=\mathbb{L}_{\tid}^l\times \mathbb{O}\circ_{P_{n,k}}\mathbb{I}_{l,k;l+k}\circ_{\GL_{n+k}}\mathbb{I}_{l+k| n+k},\quad l\le n,
\end{equation*}
where \(\mathbb{O}\in \HHom(\bfm^\bullet,\bfk^\bullet)\) is any one-morphism and \(\mathbb{I}_{l|n}\in \HHom(\mathbf{l}^\bullet,\bfn^\bullet)\) is NS5 type one-morphism
discussed in the next section.

\section{Defects and knot invariants}
\label{sec:defects-knot-invar}

In this section we explain how the we interpret the results of \cite{OblomkovRozansky16} in terms of 3-categories $\thc_{\tgl}$ and $\thc_{\tgl}^{f}$. In particular, we
make a connection with the theory of foams and provide an explanation for the Chern character construction. We also construct a partition functions  \(\sfZ^{\mathrm{f}}\) for the KRS theories the targets \(\Hilb_n(\CC^2)\), \(n\in \ZZ_{\ge 0}\) with
NS5 type defects in \(S^2\times \RR\). In particular, we show that for a circle \(S^1\) inside of a connected component of
the defect complement \(S^2\times \RR\ \Def\) the partition function yields \(\sfZ^{\mathrm{f}}(S^1)=\mathrm{D}^{\mathrm{per}}(\Hilb_n(\CC^2))\) where \(n\) is the label of the connected component.

Recall the standard setup of the 3D topological field theory with defects. A 3D QTFT is characterized by its partition-function evaluation  \(\sfZ\). The
partition function is an assignment:
\[\mbox{ closed connected three-manifold } X\mapsto \sfZ(X)\in\CC,\]
\[\mbox{ closed connected surface } S\mapsto \sfZ(S)\in\mathrm{Vect},\]\[\mbox{ three-manifold $X$ with boundary $\partial X = \bigcup_i S_i$} \mapsto \sfZ(X)\in \sfZ(\partial X) = \bigotimes_{i=1}^m \sfZ(S_i),\]
\[\mbox{ closed connected curve } C\mapsto \sfZ(C)\in \mathrm{Cat},\]\[ \mbox{ surface with boundary } S\mapsto \sfZ(S)\in \sfZ(\partial S ) = \otimes_{i=1}^k \sfZ(C_i),\]
\[\mbox{ point } p\mapsto \sfZ(p)\in \thcgb,\]
\[\mbox{ interval } I=[b,e]\mapsto \sfZ(I)\in \Hom\bigl(Z(b),Z(e)\bigr).\]

This collection of data behaves naturally under the gluing operation.
For example, suppose that a three-manifold $X$ without a boundary is cut into two pieces over a surface $S$:
\[X=X_1\cup X_2, \quad\partial(X_1)=S,\quad\partial(X_2)=S^{\vee}.\]
Then $S$ and $S^\vee$ have opposite orientations, hence $\sfZ(S)$ and $\sfZ(S^\vee)$ are dual vector spaces and
the partition function $\sfZ(X)$ is a pairing between their elements:
\[\sfZ(X)=\sfZ(X_1)\cdot \sfZ(X_2),\]
%
%
More generally, the formalism of TQFT provides a method
for computing values of \(\sfZ(Y)\) by cutting \(Y\) into pieces, evaluating \(\sfZ\) on the pieces and pairing them
in a standard way.
More details can be found in \cite{Lurie08}. To define a 3D \tqft, we need to include into the domain
of \(\sfZ\) also manifolds with corners and work with more subtle setting of \((\infty,k)\)-categories.
We postpone the discussion of such extension to our future publication \cite{OblomkovRozansky18d}.

Often \tqft s may be defined not only on smooth manifolds but also on `smooth' CW-complexes. In particular, a \tqft\ may have defects (coming from lower-dimensional cells). Topologically, defects are unions of embedded surfaces and curves. Surfaces and curves may intersect.
The cuts must be transverse to the defects. All other properties of \tqft\ without defects are preserved.

The  full  categories \(\thc_{\gl}^\bl \) will be used to construct 3D \tqft\ and we discuss the construction in the forthcoming publication
\cite{OblomkovRozansky18d}  where we construct the corresponding maps \(\sfZ\) and \(\sfZf\).

In this note we concentrate on \(\RR\)-invariant part of the TQFT.
That is from now on all defects are invariant with respect to shift along last coordinate in \(S^2\ti \RR\) and thus we can recover a configuration
of defects from the two-dimensional cut which is an intersection with \(S^2\ti \{0\}\).

We present \(S^2\) as \(\mathbb{R}^2\) compactified by a point \(\infty\) at infinity and
we can assume that the defects do not contain \(\infty\).
The two-dimensional cut is \(\mathbb{R}^2\cup \infty\). Studying two-dimensional
slice is equivalent to restricting ourselves to the defects of the form \(C\ti \RR\)
where \(C\) is a curve in \(\RR^2\).

The surfaces of defect intersect our fixed \(\mathbb{R}^2\) along the union of oriented curves. Suppose that the curves on
the plane lies in an annulus and their union is a projection of the closure of a braid. By assigning a sign to each intersection, we obtain an interpretation
of the union of curves as a projection of a link in \(\mathbb{R}^3\) presented as a closure \(L(\beta)\) of a braid \(\beta\in \Br_n\). Denote
the plane with defect \(\mathbb{R}^2_{\beta}\).

From the results of \cite{OblomkovRozansky16} it follows, that our \tqft\  provides an isotopy invariant of \(L(\beta)\), namely, the vector space
\( \sfZf(\mathbb{R}^2_\beta),\) assigned to $\mathbb{R}^2_\beta$.
We show that this invariant  categorifies HOMFLYPT polynomial and coincides with previously defined Khovanov-Rozansky invariant \cite{OblomkovRozansky20}.

Now let us give details of the \(\mathbb{R}^2\)-sliced \tqft. A small neighborhood of the plane contains two-dimensional surfaces of defect that are
products of the defect  curves in \(\RR^2\) with an interval. The surfaces of defect divide \(\RR^3\) into connected components and each connected
component has an integer marking.
We choose the markings in such a way that if we move along the (oriented) intersection of a defect surface with $\RR^2$ and the marking on the left is $k$, then the marking on the right is $k+1$.

We assume that the curves of intersection are compact, thus their union is contained in a large disc.
The marking of the disc complement (`the infinite marking') determines all other markings.%

In this paper we only consider the theories with the infinite marking equal to \(0\). Hence, if the intersection of the surface defect and $\RR^2$ is
the braid closure then all markings are positive. Thus the picture of the braid \(\beta\)
determines the marking of our theory uniquely. Slightly abusing notation, we denote such data as \(\RR^2_\beta\). The figure is a slice \(\RR^2_\beta\)
for \(\beta=\sigma_1^3\). The closure of \(L(\beta)\) is a trefoil and we explain
below how we can compute the homology of this knot.

\begin{figure}
  \includegraphics[width=2in]{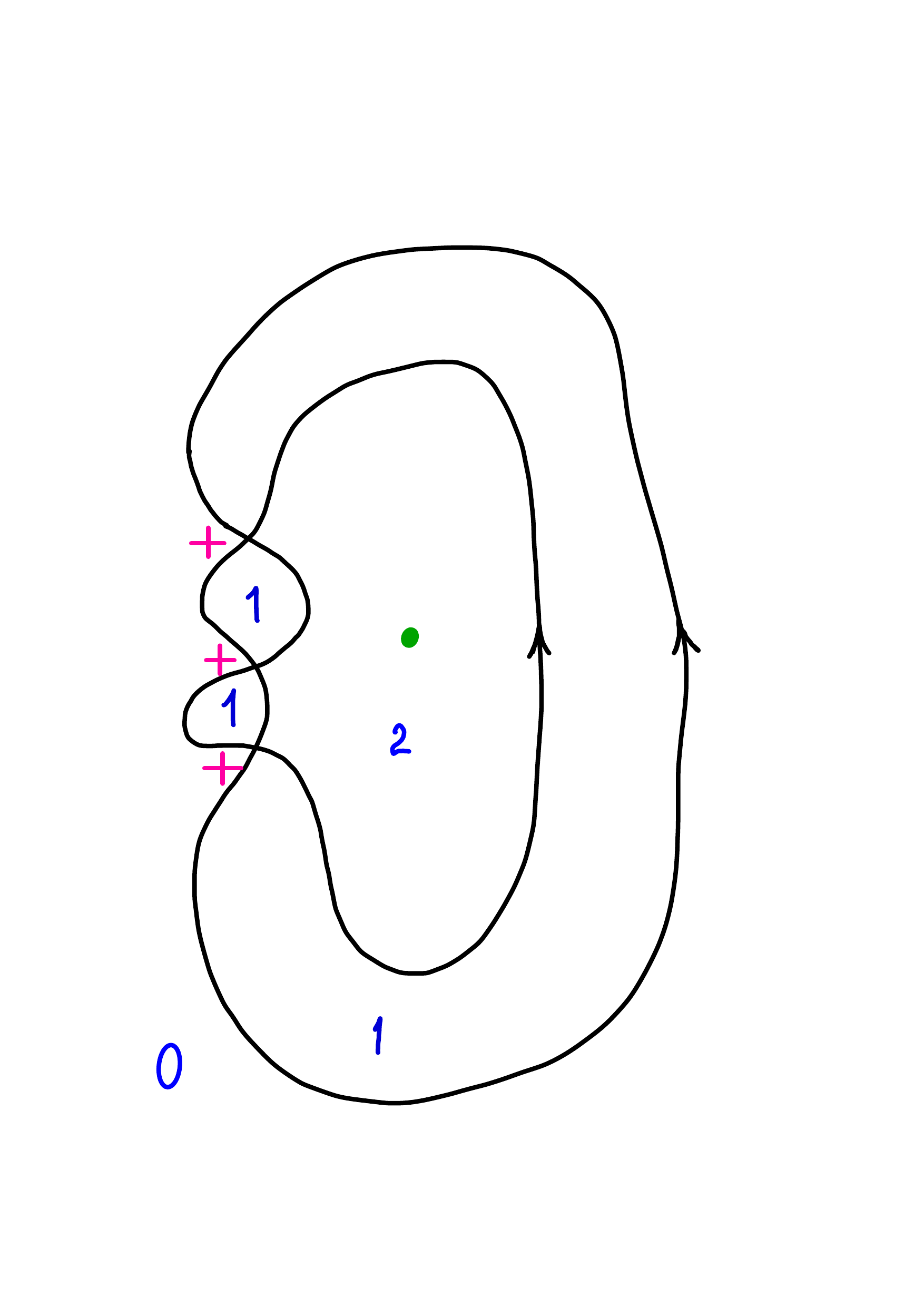}
  \caption{\(\mathbb{R}^2_\beta\) for \(\beta=\sigma_1^3\).}
  \label{pic:1}
\end{figure}

\subsection{Values on points and  defect trivial intervals}
\label{sec:valu-points-interv}
There is a canonical way to upgrade the marking of \(\RR^2_\beta\) to a categorical marking.
Denote by $\sfp_n$ a point lying inside a region marked by $n$. To a pair of points $(\sfpno,\sfpnt)$ we assign a \tctg\
\begin{equation}
\label{eq:mnass}
\sfZb(\sfpno,\sfpnt) = \HHom(\bfnbo,\bfnbt).
\end{equation}
Recall that an object of this \tctg\ is a pair $(Z,W)$, where $Z$ is a variety with a $(\GLno\times\GLnt)$-action and
$W \in\CC[\glno\times\glnt\times Z]^{\GLno\times\GLnt}$.

In accordance with the assignment~\eqref{eq:mnass},  to an interval $I$ connecting $\sfpno$ and $\sfpnt$ (and possibly crossing defect surfaces) we assign an object $\sfZb(I)$ of $\HHom(\bfnbo,\bfnbt)$, so that if $I$ is the result of gluing together the intervals $I_1$ and $I_2$ over the common middle point, then the object of $I$ is the composition:
\[\sfZb(I) = \sfZb(I_2)\circ\sfZb(I_1).\]
This relation implies that if the interval $I$ lies within a single region marked by $\sfpn$, then the corresponding object $\sfZb(I)$ is the identity with respect to the monoidal structure (\ie, the composition) of the category $\HHom(\bfnb,\bfnb)$.
%
%


Theorem~\ref{thm:unit} states that the following pair is the identity object of
 the `unframed' \tctg\ $\HHom(\bfn,\bfn)$:
\[\mathbb{L}^{n,0}_{\tid}=\mathrm{L}^n_{\tid}=\Bigl(\mathrm{T}^*\GL_n,W_{\tid}:=\Tr\,\phi\bigl(X-\Ad_g(X')\bigr)\Bigr)\]
in which $(g,\phi)\in \GLn\times\gln\cong \TsGLn$, while $(X, X')\in\gln\times\gln$, and the action of $\GLn\times\GLn$ on the total variety $\gln\times\gln\times\TsGLn$ is given by the following formula:
\[
(a,b)\cdot(X,X',g,\phi) = \bigl(\Adv{a}X,\Adv{b}Y,agb^{-1},\Adv{a}\phi\bigr)
\]
%
The identity object in the framed \tctg $\HHom(\bfnf,\bfnf)$ is the case \(r=1\) of
%
\[\mathbb{L}^{n,r}_{\tid}=\bigl(\mathrm{T}^*\GL_n\ti\Hom( V_n,\CC^r),W_{\tid}^{\tfrm}
:=W_{id}(X,g,\phi,X')+\Tr(w\cdot(v-gv'))\bigr),\]
where $(X,v), (X,v')\in\gln\times \Hom(\CC^r,V_n)$, while $(g,\phi,w)\in\TsGLn\times  \Hom( V_n,\CC^r)$ and the action of $\GLn\times\GLn$ on framing related variables is
\[
(a,b)\cdot(v,v',w) = (av,bv',aw).
\]

\begin{proof}[Proof of theorem~\ref{thm:unit}]
  We need to define the left and right unitizing 1-morphisms \[l(\scX): \scX\to\mathbb{L}_{\tid}^{n;r}\circ \scX,\quad \scX\in\HHom(\tilde{A}_0(n,r),\tilde{A}_0(n',r')),\]
  \[r(\scX):  \scX\to\scX\circ\mathbb{L}_{\tid}^{n;r},\quad \scX\in\HHom(\tilde{A}_0(n',r'),\tilde{A}_0(n,r)).\]
  We the left unitizer to be   the tensor product of matrix factorizations
  \[l(\scX):=\mathbf{1}_{\scX}\circ \mathcal{KN},\quad \mathcal{KN}=[\phi,X-\Ad_g(X')]\otimes [w,v-gv'],\]
  where we use the notation from the definition of the unit object as well as the standard notations for the Koszul matrix factorization.
  The right unitizer is defined by the identical formula.

  The matrix factorization \(\mathcal{KN}|_{g=1}\) is the matrix factorization that yields the Kn\"orrer's equivalence of categories \cite{Knorrer}.
  The composition \(\circ\) in the product \(\mathbb{O}\circ\mathbb{L}_{\tid}\) involves taking quotient by the \(\GL_n\)-action. Since the left and
  right actions on \(\GL_n\) on \(T^*\GL_n\) is free, we can trade the \(\GL_n\)-quotient for the gauge fixing \(g=1\). Thus we are in the setting of
  the Kn\"orrer periodicity.
  Respectively, the inverse to the Kn\"orrer equivalence is composition of
  \(\mathrm{KN}^*\) and the duality functor.

  The only non-trivial structural two-morphisms that involves the unit object are

 \[\begin{tikzcd}
  \cX\arrow[r," \calF",""{name=U,below}]\arrow[d,"l(\cX)"',""{name=D}]&\arrow[d,"l(\cX')"]\cX'\\
  \mathbb{L}\circ\cX\arrow[r,"1_{\mathbb{L}}\circ\calF"']&\mathbb{L}\circ\cX'
  \arrow[Rightarrow,from=U,to=D,"l(\calF)"]
\end{tikzcd},\quad
\begin{tikzcd}
  \cX\arrow[r," \calF",""{name=U,below}]\arrow[d,"r(\cX)"',""{name=D}]&\arrow[d,"r(\cX')"]\cX'\\
  \cX\circ\mathbb{L}\arrow[r,"\calF\circ 1_{\mathbb{L}}"']&\cX'\circ\mathbb{L}
  \arrow[Rightarrow,from=U,to=D,"r(\calF)"]
\end{tikzcd},
\]
which are defined by
\[l(\calF)=R(\mathcal{KN},\calF),\quad r(\calF)=R(\calF,\mathcal{KN}).\]

Thus these structural two-morphisms are compositions of the base-change natural transformations. The relevant conditions for the structural morphism related to the unit object
of monoidal 2-category are listed on the page 224 of \cite{KapranovVoevodsky94a}.
Since the structural two-morphisms in our setting are of geometric origin one
prove the conditions by the pasting method used in theorem
\end{proof}




It follows from the axioms of the monoidal 2-category (see for example \cite{BaezNeuchl96}) that the category of endomorphisms of unit object is a universal
Drinfeld center, in the following sense

\begin{corollary}\label{cor:Dr}
  The monoidal category \(\Hom(\mathbb{L}_{\tid}^{n;r},\mathbb{L}_{\tid}^{n;r})\) is
  a braided monoidal category. Moreover for any \(\scX\in \HHom(\tilde{A}_0(n;r),\tilde{A}_0(n';r'))\) there is a monoidal functor
  to the Drinfeld center of \(\Hom(\scX,\scX)\):
  \[\HC[\scX]: \Hom(\mathbb{L}_{\tid}^{n;r},\mathbb{L}_{\tid}^{n;r})\to Z(\Hom(\scX,\scX),\star).\]
\end{corollary}

In the special case of \(n'=0\) and \(r=r'=1\) and special choice of \(\scX\) this functor was studied in \cite{OblomkovRozansky18a}. The exact connection is explained in the next subsections.
In this work we also define a right adjoint functor
to the \(\HC\). We expect that the adjoint functor exists in more general setting:

\begin{conjecture}\label{conj:ch}
  For any  \(\scX\in \HHom(\tilde{A}_0(n;r),\tilde{A}_0(n';r'))\) there is  a
  functor
  \[\CH[\scX]:\Hom(\scX,\scX)\to\Hom(\mathbb{L}_{\tid}^{n;r},\mathbb{L}_{\tid}^{n;r}) \]
  that is the right adjoint to \(\HC[\scX]\) and is a categorical trace
  \[\CH[\scX](\calF\star\calG)=\CH[\scX](\calG\star\calF).\]
\end{conjecture}

\subsection{Value on the intervals intersecting defects}
\label{sec:value-interv-inters}

In this section we discuss the intersection of  small interval with the
three-dimensional trace of \(NS5\)-brane, as discussed in section~\ref{sec:ns5d5}.
We assume that \(I\) is a small interval intersecting a surface of defect in a smooth point. The surface of defect separates the regions with labels \(k\) and \(l\) as on the figure below:
\[\begin{tikzpicture}
    \filldraw[gray] (1,0) -- (1,2) -- (2,3) -- (2,1);
    \draw[thick] (1.6,1.6) -- (2.6,1.6);
    \draw[thick,dashed] (1.6,1.6) -- (1,1.6);
    \draw[thick] (1,1.6) -- (0.5,1.6);
    \node (L5) at (0,1.6) {\(k\)};
    \node (M5) at (3,1.6) {\(l\)};
\end{tikzpicture}
\]
Let us assume that the curve of defect separates regions with marking \(k \) and \(l\).
Denote by  \((\phi_{kl},\phi_{lk})\) the coordinates on \(\rmT^* \Hom(V_k,V_l)\), where
\(\phi_{kl}\in \Hom(V_k,V_l)\) and \(\phi_{lk}\in \Hom(V_l,V_k)\). Also denote by $v_l$ and $v_k$ the coordinates on $V_l$ and $V_k$ and denote by $X_l$ and $X_k$ the coordinates on $\gl_l $ and $\gl_k$. Now the object assigned to the interval $\vec{I}$ in the unframed category is a pair
\[ \sfZ(\vec{I})=\mathbb{I}_{k|l}=\bigl(\rmT^* \Hom(V_k,V_l),\quad W_{k,l}:=\Tr(X_k\phi_{kl}\phi_{lk})-\Tr(X_l\phi_{lk}\phi_{kl})\bigr).\]
In the context of the section \ref{sec:categ-line-quiv}  \(\mathbb{I}_{k|l}\) corresponds
to the NS5 defect of charge \(k-l\) with no framing i.e. \(r=0\).

In case of framed \tctgs,
if the interval starts  at \(k\) and ends in \(l\) and the shortest path from the head of the vector \(\vec{I}\) to the head of the vector of direction
of the defect goes clockwise, then we choose
\[ \sfZf(\vec{I})=\mathbb{I}^{f}_{k|l}=\bigl(\rmT^* \Hom(V_k,V_l)\times V_k^*,\quad W^f_{k,l}:=\Tr(X_k\phi_{kl}\phi_{lk})-\Tr(X_l\phi_{lk}\phi_{kl})+\Tr(\psi (v_k-\phi_{kl}v_l)\bigr).\]
In case of opposite orientation we set
\[ \sfZf(\vec{I})=\bigl(\rmT^* \Hom(V_k,V_l)\times V_l^*,\quad W^f_{k,l}:=\Tr(X_k\phi_{kl}\phi_{lk})-\Tr(X_l\phi_{lk}\phi_{kl})+\Tr(\psi (v_l-\phi_{lk}v_k)\bigr).\]

The three-dimensional picture of these two cases could visualized as two cases below:
\[\begin{tikzpicture}
    \filldraw[gray] (1,0) -- (1,2) -- (2,3) -- (2,1);
    \draw[thick,->] (1.6,1.6) -- (2.6,1.6);
    \draw[thick,dashed] (1.6,1.6) -- (1,1.6);
    \draw[thick] (1,1.6) -- (0.5,1.6);
    \node (L5) at (0,1.6) {\(k\)};
    \node (M5) at (3,1.6) {\(l\)};
    \draw[->] (1,1) -- (2,2);
\end{tikzpicture}\quad\quad\quad
\begin{tikzpicture}
    \filldraw[gray] (1,0) -- (1,2) -- (2,3) -- (2,1);
    \draw[thick,->] (1.6,1.6) -- (2.6,1.6);
    \draw[thick,dashed] (1.6,1.6) -- (1,1.6);
    \draw[thick] (1,1.6) -- (0.5,1.6);
    \node (L5) at (0,1.6) {\(k\)};
    \node (M5) at (3,1.6) {\(l\)};
        \draw[<-] (1,1) -- (2,2);
\end{tikzpicture}.
\]
Two pictures above present a local picture of the intersection in the case when all
defects (the grey wall on the picture) are invariant with respect to the vertical translation.

To match the objects \(\mathbb{I}^f_{k|l}\)  with the description of
NS5 defects of charge \(k-l\) with \(r\) from section~\ref{sec:ns5-inst} we need to set \(\varphi^{\mathrm{fr}}=1\) and \(\psi=\psi^{\mathrm{fr}}\).
%
The element \(\sfZ^\bl(\vec{I})\) is an element of the 2-category \(\HHom(\mathbf{k}^\bl,\mathbf{l}^\bl)\) thus the composition construction allow us
to interpret an element \(\sfZ^\bl(\vec{I})\) as morphism from \( \HHom(\mathbf{k}^\bl,\mathbf{0}^\bl)\) to
\(\HHom(\mathbf{l}^\bl,\mathbf{0}^\bl)\).

Let us denote the intervals as above by \(\vec{I}_{k\uparrow l}\) and
\(\vec{I}_{k\downarrow l}\). More generally we denote by
\[\vec{I}_{k_1\uparrow k_2\uparrow\dots \uparrow k_l},\quad \vec{I}_{k_1\downarrow k_2\downarrow\dots \downarrow k_l}\]
the interval that connects the connected components with the labels \(k_1\) and \(k_l\) and traverses the
domains with the labels \(k_2,\dots,k_l\) in between with the orientation of the
intersections as indicated by the arrows. We also allow a mixture of down/up orientations of the intersections.

According to the definition of our TQFT we have:
\begin{equation}\label{eq:comp-Z}
  \sfZ(\vec{I}_{k_1\uparrow k_2
    \dots\uparrow k_l})=\sfZ(\vec{I}_{k_1\uparrow k_2})\circ \sfZ(\vec{I}_{k_2\uparrow k_3})\circ\dots\circ\sfZ(\vec{I}_{k_{l-1}\uparrow k_l}).\end{equation}
The GIT quotient in the definition of the composition can be made explicit in many important cases:
\begin{proposition} \label{prop:comp12}For any \(n\ge 0\) we have:
  \[\sfZf(\vec{I}_{0\uparrow 1\uparrow \dots n})=(\rmT^*\Fl_n \times V_n,w),\quad
    \sfZf(\vec{I}_{0\downarrow 1\downarrow \dots n})=(\rmT^* \Fl_n,w),\]
  \[\sfZ(\vec{I}_{0\vert 1\vert \dots n})=(\rmT^*\Fl_n,w),\quad w=\mu\cdot X\in \CC[\gl_n\times\rmT^*\Fl_n]^{\GL_n},\]
  where \(\mu:\rmT^*\Fl_n\to \gl_n^*\) is the moment map and \(X\) are the coordinates on  \(\gl_n\).
\end{proposition}
\begin{proof} Let us first prove the last equation, the other equations are analogous and it will be indicated
  at the end of the proof how one needs to modify the proof to get the first two formulas.
  We proceed by induction on \(n\). Thus we need to compute the composition:
  \[\sfZ(\vec{I}_{0\vert 1\vert\dots\vert n-1})\circ \sfZ(\vec{I}_{n-1\vert n}).\]

  It is convenient to think of \(\rmT^*\Fl_n\) as \(B_n\)-quotient because the trace map gives a natural
  pairing on \(\gl_n\) thus we can think of \(\mu\) as a map \(\rmT^*\Fl_n\rightarrow\gl_n\):
  \[\rmT^*\Fl_n=\GL_n\times\frn_n/B_n, \quad \mu(g,Y)=\Ad_gY,\]
  where \(g\) and \(Y\) are the coordinates on \(\GL_n\) and \(\frn_n\).

  In these notations
    the composition in question is the pair of the GIT quotient space and a potential:
  \[\rmT^*\Fl_{n-1}\times \rmT^*\Hom(V_{n-1},V_n)/_+\GL_{n-1}, w_{n-1,n}=\Tr(X'\Ad_g Y')+\Tr(X'\phi\psi)-\Tr(X''\psi\phi),\]
  where \(X'\in\gl_{n-1}\), \(X''\in\gl_{n}\), \(g,Y\) are the coordinates along \(\rmT^*\Fl_{n-1}\) and
  \(\psi\in \Hom(V_n,V_{n-1})\), \(\phi\in \Hom(V_{n-1},V_n)\) are the coordinates along \(\rmT^*\Hom(V_{n-1},V_n)\).

  The GIT quotient in last formula could be made very explicit, we choose to describe the quotient by constructing
  explicit charts in the quotients. Then we show that in each chart we can apply the Knorrer periodicity to
  simplify the potential.

  The GIT stable locus consists of points where \(\phi\) is injective. Thus we can assume that there is
  \(k\) such that \[\det(\phi_{\widehat{k}})\ne 0,\] where \(\phi_{\widehat{k}}\) is \(\phi\) with \(k\)-th row removed.
  Let us denote the locus where the last inquality holds by \(U_k\). It is clear that the quotient is covered
  by the charts \(U_k/\GL_{n-1}\) and  we can analyze the potential in each chart.

  To simplify notations let us consider the case \(k=n\). The natural slice to the \(\GL_{n-1}\)-action is the closed
  subset of elements constrained by:
  \[\phi_{ij}=\delta_{i,j},\quad 1\le i\le j\le n-1.\]
  Let us also denote the last row of \(\phi \) by \(v\) and the matrix of the first \(n-1\) columns
  of \(\psi\) by \(\tilde{\psi}\) and the last row of \(\psi \) by \(\psi'\). Then the potential \(w_{n-1,n}\) becomes:
  \begin{equation}\label{eq:redw}\Tr(X'\Ad_gY')+\Tr(X'\tilde{\psi})+\Tr(vX'\psi')-\Tr(X''\phi\psi).\end{equation}
  The sum of the first three terms is quadratic and we can apply Knorrer reduction. The reduction forces the following vanishing
  of the coordinates:
  \[X'=0,\quad  \Ad_gY'+\tilde{\psi}+\psi'v=0.\]
  Thus the new coordinates on the Knorrer reduced space are \(X'',Y',v,\psi'\) and in these coordinates we have:
  \[\phi\psi=
    \begin{bmatrix}
      -\psi'v-\Ad_gY'&\psi'\\
      -v\psi'v-v\Ad_gY'&v\psi'
    \end{bmatrix}
  \]
  Thus a direct computation shows that the last term of \eqref{eq:redw} is equal to
    \[\Tr(X''\Ad_hY),\mbox{ with}\quad Y=
    \begin{bmatrix}
      Y&g^{-1}\psi'\\0&0
    \end{bmatrix},\quad
  h=    \begin{bmatrix}
      g&0\\vg&1
    \end{bmatrix}.
  \]
  Hence we proved the last formula in the chart \(U_n\) and the computations in other charts are analogous. The argument in the
  framed case is basically the same.
\end{proof}

The last proposition computes the object \(\calF_n\) from the introduction since \(
\sfZf(\vec{I}_{0\uparrow 1\uparrow \dots n})=\LaTeXunderbrace{\xPhiwo\cdots\xPhiwo}_{\qn} \calO.\)

\subsection{The categories of closed curves}
\label{sec:value-closed-curves}
The choice of defect-related objects \(\sfZb(\vec{I})\) determines categories assigned to closed curves: a curve $C$ is presented as a gluing of two intervals, then its category $\sfZb(C)$ must be the category of morphisms between their objects. Two curves are of special importance for our braid-related constructions.
%

The first type is a curve that does not intersects any defects. So the curve is a circle that lies inside of the connected component
with the marking \(n\). We denote such circle \(S^1_n\).  To a point \(p\in S^1_n\) we assign 2-category \(\HHom(\mathbf{n}^\bl,\mathbf{0}^\bl)\).
For brevity, we start using notation \[[\mathbf{n}^\bl,\mathbf{m}^\bl ]:=\HHom(\mathbf{n}^\bl,\mathbf{m}^\bl),
\]
for the corresponding 2-category.

The interval \(I\) connecting
\(\pt_n\) to itself get assigned the identity \(\sfZb(I)=\bbL_{\tid}^n\in [n,n]\). Since $S^1_n$ is a result of gluing two such intervals, its category is the category of endomorphisms of the interval object:
\[\sfZb(S^1_n)=\Hom(\bbL_{\tid}^n,\bbL_{\tid}^n).\]

In the introduction explained the matrix factorization realization of the braid
group \eqref{eq:MFst} and the Drinfeld center construction of the corresponding
monoidal category \eqref{eq:mnchcchf} from
\cite{OblomkovRozansky18a}. To relate these constructions to corollary~\ref{cor:Dr}
and motivate the conjecture~\ref{conj:ch} we need to study the Drinfeld category
from \cite{OblomkovRozansky18a}:
\begin{equation}
\label{eq:drcat}
\MF_{\tDr}^\bullet = \MF_{G_n}\bigl((\frg^\bullet\times\frg^\bullet\times G)^{\ttst}, W^\bullet_{\tDr}\bigr),
\end{equation}
\[W_{\tDr}(X,Y,g)=\Tr X(\Ad_g(Y)-Y), \quad W_{\tDr}^{\tfrm}(X,v,Y,u,g)=W_{\tDr}(X,Y,g)+u^*\cdot v- u^*\cdot gv.\]
\begin{proposition}
For \(\bullet=\emptyset,\tfrm\)
the categories $\MF_{\tDr}^\bullet$
are equivalent to the categories of
endomorphisms of the identity objects $\bbL^n_{\tid}$:
  \[\MF_{\tDr}^\bullet=\Hom(\bbL^n_{\tid},\bbL^n_{\tid})\]
\end{proposition}
\begin{proof}
  We consider only the case \(\bullet=\emptyset\) the other case is analogous.
  The statement follows from the Knorrer periodicity \cite{Knorrer}. Indeed, by
  our definition we have
  \[\Hom(\rmL_{\tid}^{n},\rmL_{\tid}^n)=\MF_{\GL_n}(\rmT^*\GL_n\times\gl_n\times\gl_n\times\rmT^*\GL_n/\GL_n,W),\]
  \[W=\Tr\, \phi\bigl(X-\Ad_g(X')\bigr)-\Tr\,\phi\bigl(X'-\Ad_{g'}(X)\bigr)\]
  where \((\phi,g)\) and \((\phi',g')\) are the coordinates on two copies of
  \(\rmT^*\GL_n\) and \(X,X'\) are the coordinates on two copies of \(\gl_n\).
  By setting \(g'=1\) we take slice to the \(\GL_n\)-orbits. On the slice
  the second term in the formula for \(W\) becomes
  \(\Tr\,\phi'(X-X').\)
  Thus the Knorrer periodicity implies that the restriction on the
  locus \(X=X',\phi'=0\) is the equivalence of the corresponding categories
  of matrix factorizations.
\end{proof}

We can linearize the potential \(W_{\tDr}^\bullet\) by introducing a new variable \(U=Yg^{-1}\):
\[W_{\tlin}(X,U,g)=\Tr(X[g,U]),\quad W_{\tlin}^{\tfrm}(X,v,U,u,g)=W_{\tlin}(X,U,g)+u^*\cdot  v -u^*\cdot gv.\]

The  group \(G\) naturally embeds inside its lie algebra \(\frg\), \(j_G\colon G\to \frg\). Induced by this embedding we have
the localization functor:
\[\mathrm{\tloc}^\bullet\colon \MF^\bullet_{\tDr}\to \underline{\MF}^\bullet_{\tDr}=\MF_G\bigl((\frg^\bullet\ti\frg^\bullet\ti \frg)^{\ttst},W_{\tlin}^\bullet\bigr).\]

It turns out that in the framed case the localization functor is an isomorphism:

\begin{proposition} \cite[Proposition 5.5.2]{OblomkovRozansky18a}
  The localization functor \(\mathrm{\tloc}^{\tfrm}\) is an isomorphism.
\end{proposition}

Since the potential \(W^{\tlin}\) is linear along the last copy of \(\frg\),
the  Koszul duality (see for example \cite{ArkhipovKanstrup15a} or \cite{OblomkovRozansky18a}) provides
an equivalence:
\begin{equation}
\label{eq:drcato}
\mathrm{KSZ}\colon \underline{\MF}^{\tfrm}_{\tDr}\longrightarrow \Dper(\Hilb_n(\CC^2)).
\end{equation}
Thus we have completed proof of theorem~\ref{thm:DrZ}.

The second type of a closed curve is the line that intersects our braid transversally. The line goes through the regions with the
marks \(0,1,\dots,n-1,n,n-1,\dots,0\). Figure~\ref{fg:crvt} gives an example.
\begin{figure}
  \includegraphics[width=2in]{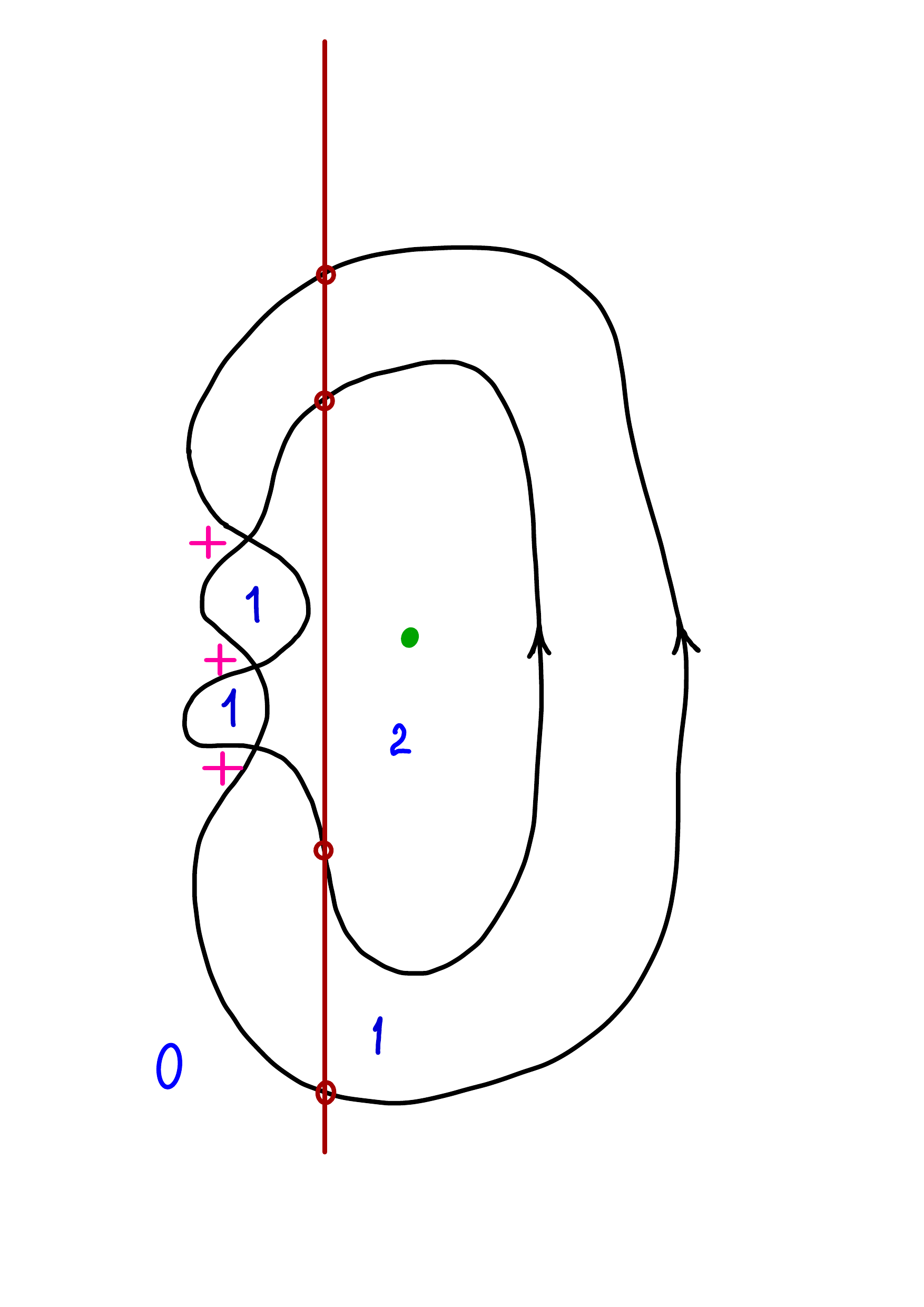}
  \caption{Plane \(\mathbb{R}^2_{\sigma_1^3}\) cut by \(\mathbb{R}_{0,1,2,1,0}\)}
  \label{fg:crvt}
\end{figure}
We denote such a line  compactified by a point at infinity as  \(S^1_{0\uparrow 1\uparrow \dots\uparrow n \downarrow n-1\downarrow \dots\downarrow 0}.\) The value
\(\sfZ^\bl\) follows immediately from the proposition~\ref{prop:comp12}:
\begin{equation}
\label{eq:cats}
\sfZf(S^\bl_{0\uparrow 1\uparrow\dots\uparrow n\downarrow\dots \downarrow 0})=\MF^{\ttst}_{G_n}(\gl_n^\bl\times\rmTs \Fl_n\times \rmT^* \Fl_n,w_1-w_2),
\end{equation}
where "st" indicate that we restrict to the GIT stable locus of the corresponding space.

The categories~\eqref{eq:cats} are closely related to the main categories of
\cite{OblomkovRozansky16}. Recall that in \cite{OblomkovRozansky16} we worked with \(G_n\times B\times B\)-equivariant categories of matrix factorizations on the space
\[\calX^\bl:=\gl_n^\bl\times (G_n\times\frn_n)\times (G_n\times\frn_n)\]
with the equivariant structure preserving the potential
\[W(X,g_1,Y_1,g_2,Y_2)=\Tr\,X(\Ad_{g_1}X_1-\Ad_{g_2}X_2),\]
where \(X\) is the coordinate in \(\gl_n\) and \(g_i,Y_i\) are the coordinates in \(G\) and \(\frn\).

\subsection{Values on discs}
\label{sec:values-discs}

As a final step of our construction we need to discuss the values of \tqft\ on discs.  The first type of disc is the disc \(D_\emptyset\) that bounds \(S^1_n\) and
does not contain tautological defect point. The category $\sfZb(S^2) = \Hom(\rmL_{\tid}^n,\rmL_{\tid}^n)$ is monoidal and $\sfZb(D_\emptyset)$ represents the identity object in it. Hence we set
%
\[\mathsf{Z}^{\tfrm}(D_\emptyset):=\calO\in \Dper(\Hilb_n(\CC^2)).\]
If the disc contains the point of tautological defect then we set
\[\mathsf{Z}^{\tfrm}(D_{\mathrm{\ttaut}}):=\Lambda^\bullet \calB\in \text{\it Coh}(\Hilb_n(\CC^2)),\]
where \(\calB\) is the tautological vector bundle.

The other important type of a disc is a half-plane $H$ bordered by the line \(S^1_{0\uparrow 1\uparrow \dots\uparrow n \downarrow n-1\downarrow \dots\downarrow 0}\). Its object $\sfZb(H)$ lies in the monoidal category
\begin{equation}
\label{eq:moncatbr}
\sfZb(S^1_{0\uparrow 1\uparrow \dots\uparrow n \downarrow n-1\downarrow \dots\downarrow 0}) =
\End(\vec{I}_{0\uparrow 1\uparrow\cdots\uparrow n-1\uparrow n}).
\end{equation}
and it depends on the configuration of defects inside $H$.
%
Denote by $H_1$ the simplest configuration which is the collection of non-intersecting curves connecting the points of the same type as in the right half-plane in Figure~\ref{fg:manyred}.
 In this situation $\sfZb$ is the identity object:
\[\sfZ^\bullet(H_1)=\calC_1\in \MF_n^{\bullet}.\]

More generally, denote by $H_\beta$  the half-plane containing a braid \(\beta\) as in the left of Figure~\ref{fg:manyred}.
The value of \(Z(H_\beta)\) for more complicated configurations of defects can be computed by
using the monoidal structure of the category~\eqref{eq:moncatbr} through
cutting \(H_\beta\) into the union:
\[ H_\beta=
\bigcup_k S_{\sigma^{\epsilon_k}_{j_k}},\]
where \(\beta=\sigma^{\epsilon_1}_{j_1}\dots \sigma^{\epsilon_l}_{j_l}\) and \(S_{\sigma^{\epsilon_k}_{j_k}}\) is the disc with the boundary \(S^1_{0\uparrow 1\uparrow \dots\uparrow n \downarrow n-1\downarrow \dots\downarrow 0}\)
and defects inside the strip form an elementary braid on the $j_k$-th and $(j_k+1)$-st stands, see the figure below for the case \(\beta=\sigma_1^3\).
\begin{figure}
  \includegraphics[width=2in]{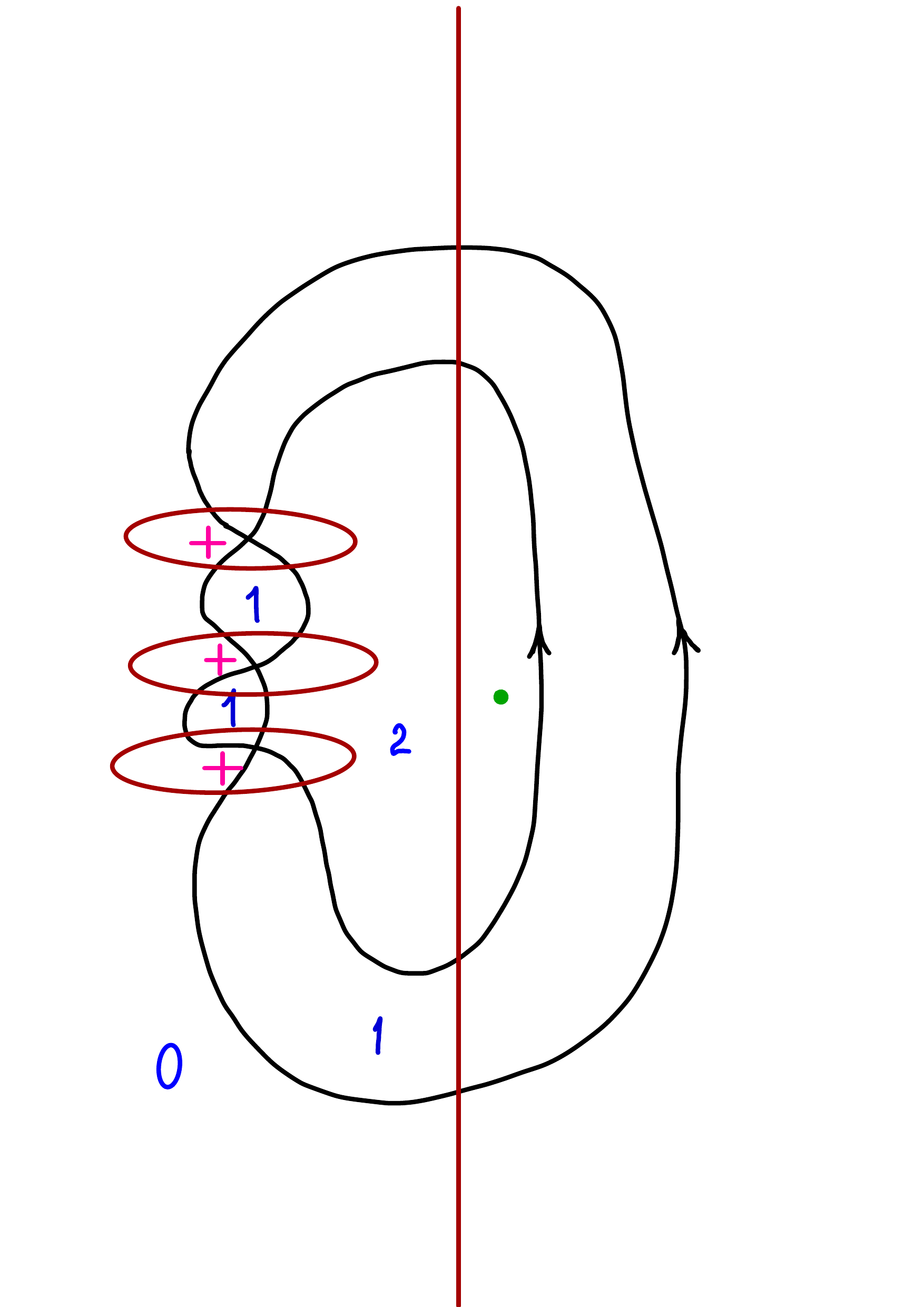}
  \caption{Decomposing \(\mathbb{R}^2_{\beta}\) on two half-planes}
\label{fg:manyred}
\end{figure}
Since
\[\sfZ^\bullet(H_\beta)=\sfZ^\bullet(S_{\sigma_{j_1}^{\epsilon_1}})\star\dots\star \sfZ^\bullet(S_{\sigma_{j_l}^{\epsilon_l}}),
\]
it is enough to define \(\sfZ^\bullet(S_{\sigma_k^{\pm 1}})\in \MF_n^\bullet\),
as in \cite{OblomkovRozansky16}:
\[\sfZ^\bullet(S_{\sigma_k^{\pm 1}}):=\calC_\pm^{(k)}\in \MF^\bullet,\]
It is shown in \cite{OblomkovRozansky16} that the element \(Z(H_\beta)\) only depends on the braid \(\beta\)
but not on the braid presentation, thus our disc assignment indeed is a well-defined partition function of \tqft.

Finally, we define the value of \(\sfZf\) on the half-plane \(H_1^{\ttaut}\) containing the
unit braid and the tautological point defect as
\[\sfZf(H_1^{\mathrm{\ttaut}}):=\calC_1\otimes \Lambda^\bullet\calB.\]

We leave the following statement as conjecture and will provide a proof in the forthcoming publication.

\begin{conjecture}
  The above assignments of the values of \(Z\) are the part of the data of a well-defined 3D \tqft.
\end{conjecture}

\subsection{Value on \(\RR^2_\beta\)}
\label{sec:value-rr2_beta}
There are two ways to cut the plane with a closed braid defect \(\RR^2_\beta\) into two pieces. As a result, the \tqft\ formalism implies two presentations of the corresponding vector space $\sfZf(\RR^2_\beta)$ as the space of morphisms between two objects in the category of the cutting line.

The first cut splits \(\RR^2_\beta\)
in two half-planes \(H_1^{\mathrm{\ttaut}}\) and  \(H_\beta\), and the corresponding presentation is
\[\sfZf(\RR^2_\beta)=\Hom\bigl(\sfZf(H_1^{\mathrm{\ttaut}}),\sfZf(H_\beta)\bigr)= \bigl(\Hom(\calC_\beta,\calC_1)\otimes\Lambda^\bullet\calB\bigr)^{B^2\times G}.\]
The vector space \(\sfZf(\RR^2_\beta)\) is triply-graded and
 the main result of \cite{OblomkovRozansky16} could be restated as
\begin{theorem}\cite[Theorem 13.3]{OblomkovRozansky16}
  The triply-graded vector space \(\sfZf(\RR^2_\beta)\) is an isotopy invariant of the closure of the braid \(\beta\) after a special shift of the grading.
\end{theorem}

 The second cut (see Figure~\ref{fg:circlcut}) splits $\RR^2_\beta$ into the inner disc $D_{\ttaut}$ (containing the tautological bundle defect) and its complement $D^\infty_\beta$ which contains the closed braid defect. The cut goes over a circle, that lies in the region marked by $n$ and does not intersect defect lines, hence its category is
\begin{equation}
\label{eq:vspco}
\sfZf(S^1_n) = \MF_{\tDr}^{\tfrm}\cong\Dper(\Hilb_n(\CC^2))
\end{equation}
of~\eqref{eq:drcat} and~\eqref{eq:drcato}. The object of $D_{\ttaut}$ is just the defect bundle: $\sfZf(D_{\ttaut})= \Lambda^\bullet\calB$. The object of $D^\infty_\beta$ is determined by the categorical Chern character functor
\[\CH^{\tfs}_{\tloc}\colon \MF_n^{\tfrm}\to\sfZf(S^1_n),\]
that is, $\sfZf(D^\infty_\beta) = \CH^{\tfs}\bigl(\sfZf(H_\beta)\bigr)$. Thus we get the second presentation of the vector space $\sfZf(\RR^2_\beta)$ as the $\Ext$ space between two complexes of sheaves within the derived category of 2-periodic sheaves on the Hilbert scheme $\Hilb_n(\CC^2)$:
\begin{equation}
\label{eq:vspct}
\sfZf(\RR^2_\beta) = \calH om\bigl( \CH^{\tfs}\bigl(\sfZf(H_\beta)\bigr), \Lambda^\bullet\calB\bigr)
\end{equation}
%
%
%
The isomorphism between the vector spaces~\eqref{eq:vspco} and~\eqref{eq:vspct} is
one of the  key  properties of the functor \(\CH^{\tfs}_{\tloc}\)
(one may call it a simple case of the categorical Riemann-Roch formula):
\begin{theorem}\cite[Theorem 6.0.4]{OblomkovRozansky18a}
  For any \(\calC\in\MF_n\) we have:
  \[\calH om(\Lambda^\bullet\calB,\CH^{\tfs}_{\tloc}(\calC))^G=\calH om(\calC_1\otimes\Lambda^\bullet\calB,\calC)^{B^2\times G}.\]
\end{theorem}

Thus we constructed a complex of sheaves \(S_\beta:=\CH^{\tfs}_{\tloc}(\calC_\beta)\) such that its global sections are the knot homology. In the language of \tqft\ the categorical Riemann-Roch and our main theorem is just a gluing property of \tqft: the picture \ref{fg:manyred} and picture
\ref{fg:circlcut} present two different ways of computing
the same partition sum \(\sfZf(\mathbb{R}^2_\beta)\).
\begin{figure}
  \includegraphics[width=2in]{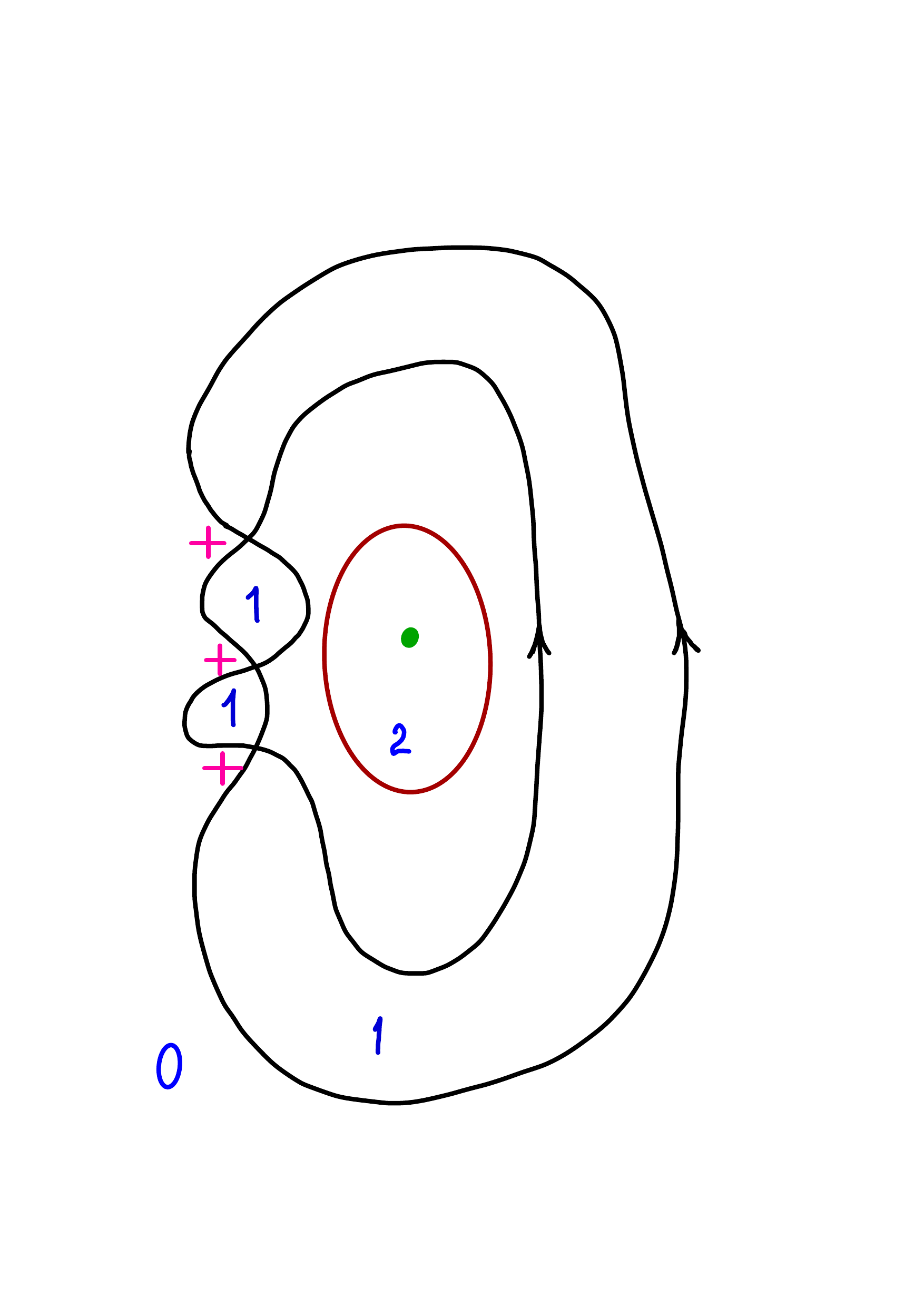}
  \caption{Decomposition of \(S^2_\beta\) on two discs.}
\label{fg:circlcut}
\end{figure}



\section{Categorification of traces}\label{sec:traces}
In this section we expand on  the table from the introduction. The first column of the table is  a summary of the previous sections. The second column is a physical interpretation of
the results of \cite{OblomkovRozansky20} where a fully-faithful functor from the category of Soergel bimodules to the category of stable matrix factorizations was constructed.
The last column is closely related to the work \cite{AnnoNandakumar16}, \cite{NakajimaTakayama17} and forthcoming preprint of the authors \cite{OblomkovRozansky22}.

\subsection{Soergel bimodules}
\label{sec:soergel-bimodules}

A slight modification of the computation in the proof of proposition~\ref{prop:comp12} shows that:
\[\widetilde{\calF}_n=\widehat{\blacksquare}_1\LaTeXunderbrace{\xPhiwo\cdots\xPhiwo}_{n-1} \calO,\quad \widetilde{\calF}_n=((\widetilde{\gl}_n\times V_n)^{\mathrm{st}},\widetilde{W}_n),\]
where \(\widetilde{\gl}_n=\GL_n\times \frb/B\) is the Grothendieck-Springer alteration with the moment map \(\mu:\widetilde{\gl}_n\to \gl_n\)
and \(\widetilde{W}_n(X,z)=\Tr(X\mu(z))\). The stability condition is given by
requiring \(\CC[\mu(z)]v=V_n\).

Indeed, the computation in the proposition~\ref{prop:comp12} shows that the composition results into the object the \(B\)-equivariant object
\((\gl_n\times \gl_n\times\GL_n\times\mathfrak{n}\times V_n,W'')\) where \(W''(Y',g,X,Y,v)=\Tr(Y'X)+\Tr(X\Ad_gY)\). The last potential, has a quadratic
term:
\[W''(Y',g,X,Y,v)=\Tr((\Ad_g^{-1}X)_{--}(Y+\Ad_g^{-1}Y'))+\Tr((\Ad_g^{-1}X)_+\Ad_g^{-1}Y).\]
where the sub-indices \(--\) and \(+\)  stand for the strictly lower-triangular and the upper-triangular parts of the matrices.
Thus the Kn\"orrer reduced object has potential \(\Tr(X'\Ad_g^{-1}Y)\), \(X'=(\Ad_g^{-1}X)_+\) and it is isomorphic to the object \(\widetilde{\calF}_n\).

The algebra of the MOY graphs \cite{MurakamiOhtsukiYamada98} has an algebraic incarnation in terms of the algebra \(\Br_n^\flat\) of flat braids.
The last algebra can be realized by means of endomorphisms of \(\widetilde{\calF}_n\)

\begin{proposition}\cite[Theorem 1.0.1]{OblomkovRozansky20}
  For any \(n\) there is a monoidal functor \[\Phi^{\flat}_n:\Br_n\to \Hom(\widetilde{\calF}_n,\widetilde{\calF}_n).\]

\end{proposition}

To connect the previous discussion with the second column of  the table from the introduction we examine the picture \eqref{pic:circles} for case when the region  \(\mathbb{B}\) contains the circular defect of type \(\mathrm{NS5}^{(n)}\).
In this case, the ray \(I_{rad}\) that starts at the center of the circle from the picture
\eqref{pic:circles} intersects the domains with labels \(0,n,n-1,\dots,1,0\) and the
corresponding category of morphisms \(\sfZ^{\mathrm{f}}(I_{rad})
=\widehat{\blacksquare}_{-n}(\widetilde{\calF})\) is computed in

\begin{proposition}
  For any \(n\) we have
  \[\widehat{\blacksquare}_{-n}(\widetilde{\calF}_n)=\mathrm{
    Mod}(\cc[y_1,\dots,y_n]).\]
\end{proposition}
\begin{proof}
  The functor \(\widehat{\blacksquare}_{-n}\) corresponds to the object
  \(\tilde{\mathbb{I}}_{0|n}=(\gl_n, \Tr(ZX))\in \HHom(\mathbf{n}^f,\mathbf{0}^f)\). Respectively,
  the composition with the object \(\widetilde{\calF}_n=((\tilde{\gl}_n\times V_n)^{st},\widetilde{W}_n)\in \HHom(\mathbf{0}^f,\mathbf{n}^f)\) is by the \(G\times B\)-equivariant object:
  \begin{equation}\label{eq:0to0}
    \widehat{\blacksquare}_{-n}(\widetilde{\calF}_n)=\tilde{\mathbb{I}}_{0|n}\circ \widetilde{\calF}_n=(\gl_n\times\tilde{\gl}_n\times V_n,\Tr(ZX)+\Tr(X\Ad_gY)).\end{equation}
  The potential of the last object is quadratic in \(Z\) and \(X\) thus after the
  Kn\"orrer reduction we obtain the \(G\times B\) equivariant object
  \(G\times B\)-equivariant object \((G\times \frn\times V_n)^{st}\).
  The action of \(G\times B\) on this object is free and a slice to the orbits of the action
  is defined by the triple \((1,K(y),v^0)\), where \(K(y)_{ij}=\sum_{i=1}^ny_i\delta_{i,i}+\sum_{i=1}^{n-1}\delta_{i,j-1}\),
  \(v^0_i=\delta_{i,n}\).
  \end{proof}

  The category of endomorphism of \(\mathrm{Mod}(R_n)\), \(R_n=\cc[y_1,\dots,y_n]\) is the
  category of two-periodic complexes of \(R_n\)-bimodules
  \(\bim_n\). Respectively, the formula~\eqref{eq:0to0} implies that there is
  a monoidal functor:
\[\mathbb{B}: \Hom(\widetilde{\calF}_n,\widetilde{\calF}_n)\colon \to \bim_n.\]

Combining the last functor with \(\Phi^\flat\) we obtain a realization of
the algebra \(\Br_n^\flat\) inside \(\bim_n\). Thus if we have two radial segments
\(I_{rad}\), \(I'_{rad}\)  with defects between the rays forming \(\beta\in \Br_n^\flat\), the value of the partition function \(\sfZ^f\) on the segment is
\(\mathbb{B}\circ\Phi^\flat(\beta)\).

The bimodule realization of \(\Br_n^\flat\) inside \(\bim\) used for \(\HHH_{alg}\)
is due to Soergel \cite{Soergel00}:
\[\Phi^{S}:\Br_n^\flat\to \Sbim_n\subset \bim_n,\quad \HHH_{alg}(\beta):=\mathrm{HH}_*(\Phi^S(\beta)),\]
where \(\mathrm{HH}_*=\mathrm{Tor}_R\) is functor of Hochschild homology. The Hochschild
homology functor is an trace functor \(\HHH_{alg}(\alpha\cdot \beta)=\HHH_{alg}(\beta\cdot \alpha)
\)

We show in \cite{OblomkovRozansky20} that two realizations of \(\Br_n^\flat\)
actually match and there is a natural isomorphism between the corresponding
trace functors:

\begin{theorem}\cite[Theorem 1.3.1]{OblomkovRozansky20}
  For any \(n\) we have a commuting diagram of monoidal functors:
  \[\begin{tikzcd}
      \Br_n^\flat\arrow[r,"\Phi^S"]\arrow[d,"\Phi^\flat"]&\Sbim_n\arrow[d]\\
       \Hom(\widetilde{\calF}_n,\widetilde{\calF}_n)\arrow[r,"\mathbb{B}"]&\bim_n
    \end{tikzcd}
  \]
  Moreover, the functors intertwine the trace functors:
  \[\mathrm{HH}^*(\Phi^S(\beta))=\calH\mathrm{om}(\Phi^\flat(\beta),\Phi^\flat(1)\otimes\Lambda^*V_n).\]
\end{theorem}

\subsection{Annular homology}
\label{sec:annular-homology}
In this section we explain the relations between the objects in the last column of the table from the introduction.
The column has a natural physical explanation in terms of NS5 and D5 branes as we explained before, see sections and \ref{sec:unified-tr} and \ref{sec:ns5d5}. The mathematical side of the story
is mostly contained in the work of Nakajima-Takayma \cite{NakajimaTakayama17},\cite{Nakajima18} and Anno-Nandakumar \cite{AnnoNandakumar16}.

The work of Nakajima-Takayama \cite{NakajimaTakayama17} explains realization of the \(\GL_n\)-Slodowy slices and their resolutions in terms of Cherkis quivers. These quiver
varieties are related to the objects we study by the linear Koszul duality. We provide more details below.

Let is also point out that the relation between the Slodowy slices has a long history. Indeed, a quiver realization of the \(\GL_n\)-Slodowy slices was conjectured by Nakajima
\cite{Nakajima94} and proven by Maffei \cite{Maffei05}. On other hand, the relation between the Slodowy slices, quiver varieties and slices in the Beilenson-Drinfeld Grassmannian were
studied by Mirkovic and Vybornov \cite{MirkovicVybornov03},\cite{MirkovicVybornov22}.

The quiver realization of the Slodowy slices from the above-mentioned work
is related to Nakajima-Takayama realization \cite{NakajimaTakayama17} by the Hannany-Witten transitions \cite{HannanyWitten97}. The mathematical theory of
Hannany-Witten transitions was developed by Nakajima-Takayma \cite{NakajimaTakayama17}.

Let us denote by \(Fl_{\mu}\), \(\mu_n\le\mu_{n-1}\le\dots\), \(\sum_{i=1}^m\mu_i=n\)
the partial flag variety consisting of nested subspaces of dimension \(M_1=\mu_1,M_2=\mu_1+\mu_2,M_3=\mu_1+\mu_2+\mu_3,\dots\).
Similarly, we denote by \(G_\mu\) the product of the groups \(\GL_{M_1}\times\GL_{M_2}\times\dots\times \GL_{M_{m-1}}\).
Respectively, \(\mathcal{N}(\mu)\subset \gl_n\) is  the nilpotent orbit with the Jordan blocks of size \(\mu_i\) and \(\mathcal{S}(\mu)\) is a transversal slice to the the orbit
\(\mathcal{N}(\mu)\).

\begin{proposition}\cite[Theorem 7.11]{NakajimaTakayama17}\label{prop:Naka}
  The \(G_\lambda\times \GL_n\times G_\mu\)-quotient of the quiver \(Q^{bow}(\lambda,\mu)\)
  pictured below
  \[
    \begin{tikzcd}[column sep=small]
       \CC\arrow[dr] &\CC\arrow[dr]&\dots\arrow[dr]&\CC\arrow[dr]&{}&&&&\\
      \CC^{L_1}\arrow[u]\arrow[r]&\CC^{L_3}\arrow[r]\arrow[u]& \dots\arrow[r]&\CC^{L_{l-1}}\arrow[u]\arrow[r]&\CC^n\arrow[r,shift left]&\arrow[l]\arrow[r,shift left]\CC^{M_{m-1}}&\arrow[l]\arrow[r,shift left]\dots&\arrow[l]\arrow[r,shift left]\CC^{M_2}&\arrow[l]\CC^{M_1}
  \end{tikzcd}
\]
is  isomorphic to \(\overline{\mathcal{N}}(\mu)\cap \mathcal{S}(\lambda)\)  and the quotient of the stable locus is the simultaneous resolution
\(\mathscr{R}_{\lambda,\mu}=p^{-1}(\overline{\mathcal{N}}(\mu)\cap \mathcal{S}(\lambda))\subset T^*Fl_\mu\) with
\(p: T^*Fl(\mu)\to \frg\) being the Springer resolution map.
\end{proposition}

The relevant object in the TQFT setting discussed in our paper is a segment \(I^{\lambda,\mu}\) with ends in the regions marked
with \(\mathbf{0}\) and such that the list of defects it traverses  is
\(\mathrm{NS5}^{\lambda_1},\mathrm{NS5}^{\lambda_2},\dots,\mathrm{NS5}^{\lambda_l},\mathrm{D5}^{\mu_m},\dots,\mathrm{D5}^{\mu_1}\).
As we explained in the section~\ref{sec:ns5d5} the value \(\sfZ(I^{\lambda,\mu})\) of partition function is given by composition of morphisms
\[\sfZ(I^{\lambda,\mu})=\widecheck{\Box}_{\lambda_1}\circ\widecheck{\Box}_{\lambda_2}\circ\dots\widecheck{\Box}_{\lambda_{l}}\circ\widehat{\Box}_{\mu_m}\circ\dots\circ\widehat{\Box}_{\mu_2}\circ\widehat{\Box}_{\mu_1}.\]

Thus  the previous proposition gives a geometric description of the morphism
\[\sfZ(I^{\lambda,\mu})=(Q^{bow}\times \mathrm{Lie}(G_\lambda\times \GL_n\times G_\mu), W^{\lambda,\mu}(z,x)=\Tr(\mu(z)x)).\]

To specialize to the case treated in the table in the introduction we need to set \(\mu=(1^n)\) and \(\lambda=(k,n-k)\).
Since the braid group \(\Br_n\) acts on \(\widehat{\Box}_1^n\) there is monoidal functor
\[\Phi^{\lambda}: \Br_n\to\Hom(\sfZ(I^{\lambda,1^n}),\sfZ(I^{\lambda,1^n})).\]

In the setting of the last column of the table from the introduction the radial rays in the picture \eqref{pic:circles} are the homotopic to the interval
\(I^{(k,n-k),(1^n)}\). Thus the last column of the table is equivalent to the following conjecture that we plan to address in the forthcoming paper \cite{OblomkovRozansky22}:

\begin{conjecture}
  For any \(n\)  a trace functor \(\mathcal{T}r:\Br_n\to 2-\mathrm{gr} \mathrm{Vect}\)
  \[\mathcal{T}r_{k,n-k}(\beta)=\mathcal{H}om(\Phi_{(k,n-k)}(\beta),\Phi_{(k,n-k)}(1_n))\]
  categorifies the trace \(\Tr(\beta)[H^\lambda_{1^n,\lambda+k}]\).
\end{conjecture}

The space \(H^\lambda_{1^n,\lambda+k}\) is the weight \(k\) part of
\(L_1^{\otimes n}\), it is of dimension \({n\choose k}\). To provide an evidence to the conjecture we compute the Euler characteristic of the trace for the identity.
\begin{proposition}
  For any \(n,k\)   we have:
  \[\chi(\mathcal{T}r_{k,n-k}(1_n))= {n\choose k}.\]
\end{proposition}
\begin{proof}
  First let us observe that by the linear Koszul duality for the matrix factorizations and by the proposition~\ref{prop:Naka}
  the category of morphism \(\Hom(\sfZ(I^{\lambda,\mu}),\sfZ(I^{\lambda,\mu}))\) is equivalent to the dg category two-periodic complexes
   \(D^{per}_{\CC^*\times\CC^*}(\mathscr{R}_{\lambda,\mu}\times\mathscr{R}_{\lambda,\mu})\). Respectively, \(\Phi_{\lambda}(1_n)\) is the structure
  sheaf of the diagonal \(\Delta\subset \mathscr{R}_{\lambda,1^n}\times\mathscr{R}_{\lambda,1^n}\).

  Thus \(\mathcal{T}r(1_n)\) is the homology self-intersection of the diagonal and since \(\mathscr{R}_{\lambda,1^n}\) is smooth we get
  \(\mathcal{T}r(1_n)=H^*(\oplus_i \Omega^i_{\mathscr{R}_{\lambda,1^n}})\). Since the \(\CC^*\times\CC^*\)-fixed locus in \(\mathscr{R}_{\lambda,1^n}\) is zero-dimensional,
  the localization formula implies that Euler characteristics \(\chi(\mathcal{T}r(1_n))\) is equal to the number of torus fixed points.

  The 3D mirror partner of \(\mathscr{R}_{(k,n-k),1^n}\)  is the Grassmanian \(\mathrm{T}^*Gr(k,n)\) \cite{NakajimaTakayama17}. The mirror symmetric partners have the same
  number of the torus fixed points and \(\mathrm{T}^*Gr(k,n)\) has \({n\choose k}\) fixed points.
\end{proof}

Finally, let us remark that in the work \cite{AnnoNandakumar16} the category of coherent sheaves on \(\mathscr{R}_{(k,k),(1^{2k})}\) was used to
construct the invariants of the affine tangles and relate this invariant to the annular Khovanov homology \cite{AsaedaPrzytyckiSikora04}. The under/over-crossing  the generators of the affine tangles in \cite{AnnoNandakumar16} match with
the braid generators used in this paper.
The cap/cup of the affine tangle monoid from \cite{AnnoNandakumar16} have a natural interpretation in terms of NS5 defects. Thus in the forthcoming paper \cite{OblomkovRozansky22}
we also address the

\begin{conjecture}
  The construction  of the trace functor \(\mathcal{T}r_{(k,k)}\) can be extended to  affine tangles and the corresponding invariant provides a realization of
  the \(sl_2\) annular Khovanov homology.
\end{conjecture}

\section{Conflict of Interests}
\label{sec:conflict-interests}
On behalf of all authors, the corresponding author states that there is no conflict of interest.


\newcommand{\etalchar}[1]{$^{#1}$}

\end{document}